\documentclass[a4paper,10pt,leqno,english]{amsart}
\usepackage[utf8]{inputenc}
\usepackage[T1]{fontenc}
\usepackage{microtype}
\usepackage[french, main=english]{babel}
\usepackage{amssymb,latexsym}

\usepackage{mathrsfs,tgschola,mathabx}
\usepackage[normalem]{ulem}
\usepackage{url}
\usepackage{xcolor}
\usepackage{comment}
\definecolor{violet}{rgb}{0.0,0.2,0.7}
\definecolor{rouge2}{rgb}{0.8,0.0,0.2}
\usepackage{tikz}
\usepackage{empheq}
\usepackage{tikz-cd}
\usetikzlibrary{matrix,arrows,decorations.pathmorphing}
\usepackage{hyperref}
\usepackage{mathpazo}
\usepackage{enumerate}
\usepackage{geometry}

\usepackage[all]{xy}
\usepackage{booktabs}
\usepackage{longtable}

\hypersetup{
    unicode=false,          
    pdftoolbar=true,        
    pdfmenubar=true,        
    pdffitwindow=false,     
    pdfstartview={FitH},    
    pdftitle={Logarithmic Donaldson-Thomas 
theory for local Calabi-Yau 4-folds},    
    pdfauthor={Xiaolong Liu},     
    colorlinks=true,       
   linkcolor=rouge2,          
    citecolor=violet,        
    filecolor=black,      
    urlcolor=cyan}           
\usepackage{enumitem}
\usepackage{appendix}

\theoremstyle{definition}
\newtheorem{defi}{Definition}[section]
\newtheorem{exam}[defi]{Example}
\newtheorem{nota}[defi]{Remark}

\newtheorem{conj}[defi]{Conjecture}
\newtheorem{Q}[defi]{Question}

\theoremstyle{plain}
\newtheorem{thm}[defi]{Theorem}
\newtheorem{prp}[defi]{Proposition}
\newtheorem{lem}[defi]{Lemma}
\newtheorem{coro}[defi]{Corollary}
\newtheorem{ass}[defi]{Assumption}

\newcommand{\scr}{\mathscr}
\newcommand{\bb}{\mathbb}
\newcommand{\mcal}{\mathcal}
\newcommand{\spec}{\operatorname{Spec}}

\newcommand{\RR}{\mathbf{R}}
\newcommand{\LL}{\mathbf{L}}

\newcommand{\loc}{\mathsf{loc}}
\newcommand{\vir}{\mathsf{vir}}

\newcommand{\Hilb}{\mathsf{Hilb}}
\newcommand{\dHilb}{\mathsf{dHilb}}

\newcommand{\PP}{\mathsf{P}}
\newcommand{\dP}{\mathsf{dP}}

\newcommand{\Crit}{\mathsf{Crit}}
\newcommand{\dCrit}{\mathsf{dCrit}}
\newcommand{\QM}{\mathsf{QM}}
\newcommand{\dQM}{\mathsf{dQM}}

\newcommand{\dMap}{\mathsf{dMap}}

\newcommand{\CH}{\operatorname{CH}}
\newcommand{\vdim}{\operatorname{vdim}}

\newcommand{\PT}{\mathsf{PT}}
\newcommand{\DT}{\mathsf{DT}}
\newcommand{\Tot}{\operatorname{Tot}}

\newcommand{\Res}{\mathsf{Res}}

\newcommand{\rank}{\operatorname{rank}}
\newcommand{\coker}{\operatorname{coker}}
\newcommand{\BM}{\operatorname{BM}}
\newcommand{\TT}{\mathbf T}
\newcommand{\Exp}{\mathsf{Exp}}
\newcommand{\pr}{\mathrm{pr}}

\newcommand{\sslash}{\mathbin{/\mkern-14mu/}}

\title[On logarithmic Donaldson-Thomas 
invariants for local Calabi-Yau $4$-folds]{On logarithmic Donaldson-Thomas 
invariants for local Calabi-Yau $4$-folds}

\author{Xiaolong Liu}
\address{Institute of Mathematics, AMSS, Chinese Academy of Sciences, 55 Zhongguancun East Road, Beijing, 100190, China}
\email{liuxiaolong@amss.ac.cn}
\date{\today.}

\begin{document}

\begin{abstract}
In this paper, we study logarithmic Donaldson-Thomas invariants for local log Calabi-Yau $4$-folds with simple normal crossing divisors. 
Using the family version of shifted Lagrangian classes announced by Khan-Kinjo-Park-Safronov, we construct relative and family logarithmic $\mathsf{DT}_4$-theories 
for logarithmic Hilbert schemes of curves, and prove a degeneration formula. 
For logarithmic Hilbert schemes of points, we construct virtual classes and 
prove a degeneration formula in Chow groups; in particular, it is independent of shifted Lagrangian classes. 
Finally, combining our degeneration formula with logarithmic cobordism, we explicitly 
compute the zero-dimensional logarithmic $\mathsf{DT}_4$-invariants for local surface snc pairs.
\end{abstract}

\maketitle

\hypersetup{linkcolor=black}

\section{Introduction}
\subsection{Background}
Donaldson-Thomas theory originated from Thomas' thesis \cite{Thomas2000}
as a sheaf-theoretic framework for enumerative geometry of Calabi-Yau and Fano $3$-folds,
built on the two-term perfect obstruction theories of Li-Tian \cite{LT98} and Behrend-Fantechi \cite{BF97}.
Later, Maulik-Nekrasov-Okounkov-Pandharipande \cite{MNOP1,MNOP2} extended 
Donaldson-Thomas invariants of Hilbert schemes on any $3$-folds.

Unlike the $3$-fold case, the obstruction theory for Calabi-Yau
$4$-folds is a $3$-term symmetric complex, requiring substantially
different methods. The corresponding $\mathsf{DT}_4$-theory
was initially studied by Cao-Leung \cite{CL14} by
constructing virtual classes in several special cases.
A complete formulation was given by Borisov-Joyce
\cite{BJ17}, using the shifted
symplectic geometry of Pantev-Toën-Vaquié-Vezzosi \cite{PTVV13} and derived differential geometry.
Later, Oh-Thomas gave an algebraic construction in \cite{OT23} and established its comparison with the 
Borisov-Joyce theory in \cite{OT24}. 
To make $\mathsf{DT}_4$ virtual classes
functorial, Park \cite{park21,park24} developed virtual pullbacks for locked
$(-2)$-shifted symplectic fibrations. More recently, Khan-Kinjo-Park-Safronov have
announced forthcoming work \cite{KKPS3} constructing Lagrangian
classes for oriented exact $(-1)$-shifted Lagrangian correspondences,
proving a conjecture of Joyce \cite{JS19,ABB17}, and providing a new
construction of $\mathsf{DT}_4$ theory.

Relative Donaldson-Thomas theory for $3$-folds with smooth boundary
divisors was developed via expanded degenerations by Li-Wu
\cite{LW15}, and was extended to smooth Deligne-Mumford stacks by
Zhou \cite{Zhou18}. More recently, Maulik-Ranganathan
\cite{MR24,MR25,MR26} developed logarithmic Donaldson-Thomas theory
for general normal crossing pairs of $3$-folds, establishing
degeneration formulae and proposing logarithmic GW/PT and DT/PT
correspondences.

For $4$-folds, relative $\mathsf{DT}_4$ theory with smooth
anti-canonical divisors was studied by Cao-Leung \cite{CL17} in special cases.
Recently, Cao-Zhao-Zhou \cite{CZZ24} developed relative
$\mathsf{DT}_4$ theory for \emph{smooth} log Calabi-Yau pairs using
shifted symplectic geometry, proved a degeneration formula, and
computed the zero-dimensional invariants for $\bb C^4$ and local curves.

The present paper extends the framework to \emph{simple normal
crossing} anti-canonical divisors of $4$-folds by combining the logarithmic Hilbert scheme machinery in
\cite{MR24,MR25} with the shifted symplectic geometry \cite{PTVV13}.
For the general curve-counting theory, we construct the relative logarithmic
$\mathsf{DT}_4$-theory and prove a degeneration formula 
using the family version of shifted
Lagrangian classes announced in the forthcoming work \cite{KKPS3},
which we formulate precisely in Conjecture \ref{conj-Joyce-conj} 
and Remark \ref{rmk-KKPS-lagrangian-classes}. By contrast,
the zero-dimensional theory, including the construction of virtual classes and the proof of the degeneration formula, 
can be constructed without shifted Lagrangian classes. As an
application, we compute the zero-dimensional
logarithmic $\mathsf{DT}_4$-invariants with tautological insertions for
the local surface snc pairs explicitly.

\subsection{Logarithmic \texorpdfstring{$\mathsf{DT}_4$}{DT4}-theory for simple normal crossing pairs}
In order to develop logarithmic $\mathsf{DT}_4$-theory of Hilbert stacks of curves, we 
generalize the shifted-symplectic construction of \cite{CZZ24} and formulate it via
the logarithmic Hilbert schemes $\Hilb_{\beta,\chi}(X|D)$ and stacks
of expansions $\mathsf{Exp}_{\beta,\chi}(X|D)$ developed in \cite{MR24,MR25}. 
Consequently, we have the following 
result.

\begin{thm}[Relative theory, {Theorem \ref{thm-curves-shifted-lag}}]\label{thm-intro-log-DT4}
Let $(X,D)$ be an snc pair with $X$ a smooth quasi-projective $4$-fold which is log Calabi-Yau, that is, $K_X+D\sim0$.
Choose compactifiable cone structures as in Definition \ref{def-compactifiable-cone-structures}.
Then we have the following properties.
\begin{enumerate}
\item The restriction morphism has a canonical derived enhancement\footnote{The stack $\dHilb^{\mathbf n}(D|\partial D)$
is defined over the universal divisor over $\mathsf{Exp}_{\beta,\chi}(X| D)$,
rather than over the stack used in \cite{MR25}. See \S~\ref{sec-rel-log-DT4} for details.}
\[
\mathsf{r}_{X\to D}:\dHilb_{\beta,\chi}(X| D)\to\dHilb^{\mathbf n}(D|\partial D)
\]
which is a $(-1)$-shifted Lagrangian relative to $\mathsf{Exp}_{\beta,\chi}(X| D)$.
\item If there exists an snc pair $(Y,S)$ with $\dim Y=3$ such that $X=\mathrm{Tot}(\omega_Y(S))$ and $D=\mathrm{Tot}(\omega_S)$ (we call such an $(X,D)$ a \textsf{local pair}),
    then $\mathsf{r}_{X\to D}$ is moreover exact and
    admits a canonical orientation relative to $\mathsf{Exp}_{\beta,\chi}(X| D)$.
\end{enumerate}
Therefore, under the identifications $X=\mathrm{Tot}(\omega_Y(S))$ and $D=\mathrm{Tot}(\omega_S)$,
we have
the \textsf{relative logarithmic $\mathsf{DT}_4$-theory}:
\[
[\mathsf r_{X\to D}]^{\mathrm{Lag}}:
H^*_{\mathrm{crit}}(\dHilb^{\mathbf n}(D|\partial D))\to H^{\BM}_{\vdim \dHilb_{\beta,\chi}(X| D)-*}(\dHilb_{\beta,\chi}(X| D),\bb Q).
\]
\end{thm}

\begin{nota}\label{rmk-stable-pairs-results}
See Remark \ref{rmk-stable-pairs} for stable pairs. Note that all 
the constructions here hold for stable pairs by similar arguments.
We only state them for Hilbert schemes for simplicity.
\end{nota}

Fix a compactifiable cone structure for expanded pairs. Let 
$\alpha_{\bullet}\in H^{k}_{\mathrm{crit},\bb T}
\left(
\dHilb^{\mathbf n}(D|\partial D)_{\bullet}
\right)$ be a series of
compatible logarithmic boundary insertions
as $\bullet$ varies, in the sense of \S~\ref{sec-DT-PT-corr}. 
Under suitable proper condition, it defines the logarithmic
DT/PT series with tautological insertions
\[
\DT_{\beta}(X|D;\alpha)_{\scr L,m},
\qquad
\PT_{\beta}(X|D;\alpha)_{\scr L,m}.
\]
We conjecture the following logarithmic analogue of the absolute
tautological DT/PT correspondence conjectured in \cite[Conj.~0.13]{CKM22}.
\begin{conj}[Tautological logarithmic DT/PT correspondence, Conjecture \ref{conj-tautological-log-DT-PT}]
For any such compatible 
$\alpha$, one has
\[
\frac{
\DT_{\beta}
(X|D;\alpha)_{\scr L,m}
}{
\DT_0(X|D)_{\scr L,m}
}
=
\PT_{\beta}
(X|D;\alpha)_{\scr L,m}.
\]
\end{conj}

The expected closed formula for the denominator $\DT_0(X|D)_{\scr L,m}$ is stated separately in 
Conjecture \ref{conj-log-zero-dim-DT} below.

\subsection{Simple normal crossing degeneration and family theory}
To formulate a degeneration formula, we need a family theory that compares virtual classes 
of the smooth and singular fibres of a one-parameter degeneration. 
Recall that a simple normal crossing degeneration is a flat morphism
\[\pi:\mcal X\to\bb A^1\]
from a smooth quasi-projective variety such that $\mcal X_b$ is smooth
for $b\neq0$ and the central fibre $\mcal X_0$ is a reduced snc divisor
in $\mcal X$; it is Calabi-Yau if $\omega_{\mcal X/\bb A^1}\cong\scr O_{\mcal X}$.
See Definition \ref{defi-snc-degeneration} for the precise definition.

Following the expansion formalism of Maulik-Ranganathan \cite{MR24,MR25}, the 
tropicalization of $\pi$ determines a stack of vertical expansions
\[\mathsf{Exp}_{\beta_t,\chi_t}(\mcal X/\bb A^1)\to\bb A^1,\] 
which keeps track of expanded targets and their combinatorial types throughout the 
degeneration. Its universal expanded target supports the
family logarithmic derived Hilbert stack. We have the following existence 
result of the virtual class.

\begin{thm}[Family theory, {Proposition \ref{prp-family-shifted}}]
Let
\[
\pi:\mcal X\to\bb A^1
\]
be a Calabi-Yau simple normal crossing degeneration of $4$-folds
in the sense of Definition \ref{defi-snc-degeneration},
satisfying Assumption \ref{ass-compactifiable-families}.
\begin{enumerate}
\item Then there exists a derived enhancement
\[
\varpi_{\beta_t,\chi_t}:
\dHilb_{\beta_t,\chi_t}(\mcal X/\bb A^1)\to\mathsf{Exp}_{\beta_t,\chi_t}(\mcal X/\bb A^1)
\]
which admits an exact $(-2)$-shifted symplectic structure.
\item If there exists an snc degeneration $\mcal Y\to\bb A^1$ of relative dimension $3$ such that $\mcal X=\mathrm{Tot}(\omega_{\mcal Y/\bb A^1})$,
    then $\varpi_{\beta_t,\chi_t}$ admits a canonical orientation.
\end{enumerate}
Therefore, under the identification $\mcal X=\mathrm{Tot}(\omega_{\mcal Y/\bb A^1})$,
we have the
\textsf{family virtual Lagrangian class}:
\[
[\dHilb_{\beta_t,\chi_t}(\mcal X/\bb A^1)]^{\vir}\in H^{\BM}_{\vdim\dHilb_{\beta_t,\chi_t}(\mcal X/\bb A^1)+1}(\dHilb_{\beta_t,\chi_t}(\mcal X/\bb A^1),\bb Q).
\]
\end{thm}

\begin{nota}\label{rmk-intro-compactifiability}
By Lemma \ref{lem-compactifiable-local-families}, Assumption \ref{ass-compactifiable-families}
can be satisfied for total spaces of vector bundles over projective snc degenerations,
including the local curve, surface and threefold families arising from projective bases.
All degeneration formulae below use this assumption, with the horizontal boundary included in the log Calabi-Yau case.
\end{nota}

The preceding constructions depend a priori on a choice of compactifiable cone
structure on the expansion stack. With compatible orientations, for the zero-dimensional theory and
for the restrictions of the family class to individual fibres, the
resulting virtual classes are independent of this choice under proper pushforward
from common refinements;
see \cite[Thm.~5.2.1]{MR24}, Theorem \ref{thm-independent-cone-models}
and its remark. For the relative curve-counting theory, the corresponding independence statement is formulated in Conjecture \ref{conj-cone-structure-relative-log}.

\subsection{Degeneration formulae for logarithmic \texorpdfstring{$\mathsf{DT}_4$}{DT4}-theory}
The family virtual class allows us to compare the theory of a smooth
fibre with the contributions carried by the expanded components of the
central fibre. 
Our main theorem proves the following
degeneration formula in logarithmic $\mathsf{DT}_4$-theory using the associativity
of family shifted Lagrangian classes.
\begin{thm}[{Theorem \ref{thm-degeneration-general-curves}}]\label{spthmA}
Let
\[
\pi:\mcal X\to\bb A^1
\]
be a Calabi-Yau simple normal crossing degeneration of $4$-folds
satisfying Assumption \ref{ass-compactifiable-families}.
Assume that there is an snc degeneration
$\mcal Y\to\bb A^1$ such that $\mcal X=\mathrm{Tot}(\omega_{\mcal Y/\bb A^1})$.
Then we have the following results.
\begin{enumerate}
\item If $0\neq b\in\bb A^1$, then
\[
i_b^![\dHilb_{\beta_t,\chi_t}(\mcal X/\bb A^1)]^{\vir}=[\dHilb_{\beta_b,\chi_b}(\mcal X_b)]^{\vir}.
\]
\item If $b=0\in\bb A^1$, then
\[
i_0^![\dHilb_{\beta_t,\chi_t}(\mcal X/\bb A^1)]^{\vir}
=n_*\sum_\gamma m_\gamma[\dHilb_\gamma(X_0)]^{\vir}.
\]
Here $\gamma$ indexes the reduced central expansion components,
$n$ is induced by their disjoint union, and
$m_\gamma=\operatorname{ord}_{\mathsf{Exp}_\gamma(X_0)}(t)$ for the degeneration parameter $t$.
Moreover, for each $\gamma$, let $Y_\nu$ label the components of the generic
expanded fibre, including tube components, and set $D_\nu:=\partial Y_\nu$.
Then we have
\[
i_{\Delta,\gamma,*}[\dHilb_\gamma(X_0)]^{\vir}
=
\left[{\prod_\nu}_{/\mathsf{Exp}_\gamma(X_0)}
\mathsf r_{Y_\nu\to D_\nu}\right]^{\mathrm{Lag}}
\left(\Delta_\gamma^{\mathrm{Lag}}\right).
\]
Here $i_{\Delta,\gamma}$ is the base change of the diagonal
(Proposition \ref{prp-fiberprod-gluing-general}), and the operator on the right
is the Lagrangian pullback of the fibre product of restriction maps
$\mathsf r_{Y_\nu\to D_\nu}$ over $\mathsf{Exp}_\gamma(X_0)$;
see \S~\ref{sec-cutting-maps} for details.
The diagonal class
\[
\Delta_\gamma^{\mathrm{Lag}}\in
H^*_{\mathrm{crit}}\left(
{\prod_\nu}_{/\mathsf{Exp}_\gamma(X_0)}
\dHilb_{\Delta,\gamma(\nu)}(D_\nu|\partial D_\nu)
\right)
\]
is described in Theorem \ref{thm-degeneration-general-curves}.
\end{enumerate}
\end{thm}

\begin{nota}\label{rmk-sec-modify-dege}
The component Hilbert stacks in the second identity are defined over
$\mathsf{Exp}_\gamma(X_0)$ by partially normalizing the original universal family,
with the inherited expansion data. So the sources of restriction
maps $\mathsf r_{Y_\nu\to D_\nu}$ here are not Deligne-Mumford
and the formula is over a common base and does not require a comparison
with independently chosen expansion stacks for the components.
\end{nota}

\begin{Q}\label{Q-dege-log-rel}
Under what conditions can the component restriction maps in
Theorem \ref{spthmA} be compared with
those of independently defined relative logarithmic $\mathsf{DT}_4$ theories in Theorem \ref{thm-intro-log-DT4},
after suitable modifications?
Can such a comparison express the degeneration formula in terms
of these relative theories, as in \cite[\S~8.2]{MR25}?
\end{Q}
A corresponding degeneration formula for simple degenerations can also be stated
using Li-Wu's construction; see Corollary \ref{coro-degeneration-formula-smooth-pair}.
Hence we can answer Question \ref{Q-dege-log-rel} in Li-Wu's case.

\subsection{Zero-dimensional logarithmic \texorpdfstring{$\mathsf{DT}_4$}{DT4}-theory}
We now specialize to point counting. Since only $0$-complexes occur,
it suffices to use the stack of $0$-complexes $\mathsf{Exp}_{0}(X|D)$,
and both the virtual classes and the degeneration formula can be
constructed in Chow groups without using shifted Lagrangian classes.

\begin{prp}[{Proposition \ref{prp-points-exact-sympl}}]\label{prp-zero-dim-virtual}
Let $(X,D)$ be a log Calabi-Yau pair with reduced simple
normal crossing divisor $D$ and $\dim X=4$ and $R\in K^{\mathrm{num}}_{\leq0,\mathrm{cpt}}(X)$.
Choose compactifiable rank-zero cone structures as in Definition \ref{def-compactifiable-cone-structures}.
\begin{enumerate}
\item The canonical morphism
\[
r_R:\dHilb^R(X|D)\to\mathsf{Exp}_{0,R}(X|D)
\]
admits an exact $(-2)$-shifted symplectic structure.
\item If there exists an snc pair $(Y,S)$ with $\dim Y=3$ such that $X=\mathrm{Tot}(\omega_Y(S))$ and $D=\mathrm{Tot}(\omega_S)$,
then
\[
r_R:\dHilb^R(X|D)\to\mathsf{Exp}_{0,R}(X|D)
\]
admits a canonical orientation.
\end{enumerate}
Therefore,
assuming the existence of an orientation $o$, we obtain the \textit{virtual class}:
\[
[\Hilb^R(X|D)]^{\vir}_{\bb T,o}:=\sqrt{r_R^!}[\mathsf{Exp}_{0,R}(X|D)]\in \CH_*^{\bb T}(\Hilb^R(X|D)),
\]
where $\bb T$ is any (possibly trivial) torus preserving the Calabi-Yau form; we call it a Calabi-Yau torus.
\end{prp}

For the logarithmic Hilbert scheme of points, we can generalize it to log Calabi-Yau degenerations
instead of Calabi-Yau case:

\begin{thm}[{Theorem \ref{coro-zero-degeneration}}]\label{coro-zero-degeneration-intro}
Let $(\mcal X,\mcal D')\to\bb A^1$ be a log Calabi-Yau snc degeneration
with flat horizontal divisor $\mcal D'$ which is fibrewise snc.
Let $\mcal D=\mcal D'\cup X_0$ be snc
and $R_t\in K^{\mathrm{num}}_{\leq0,\mathrm{cpt}}(X_t)$.
Suppose that Assumption \ref{ass-compactifiable-families} holds for the vertical $0$-complexes,
including the horizontal boundary.

Assume moreover that there exists an snc degeneration
$(\mcal Y,\mcal S)\to\bb A^1$ of relative dimension $3$ such that
$\mcal X=\mathrm{Tot}(\omega_{\mcal Y/\bb A^1}(\mcal S))$ and
$\mcal D'=\mathrm{Tot}(\omega_{\mcal S/\bb A^1})$,
then we have:
\begin{itemize}
\item If $0\neq b\in\bb A^1$, then
\[
i_b^![\Hilb^{R_t}((\mcal X|\mcal D)/\bb A^1)]^{\vir}
=[\Hilb^{R_b}(X_b|\partial X_b)]^{\vir}.
\]
\item If $b=0$, then
\[
i_0^![\Hilb^{R_t}((\mcal X|\mcal D)/\bb A^1)]^{\vir}
\leftrightsquigarrow\sum_\gamma \bigboxtimes_{\nu\text{ from }\gamma}
[\Hilb^{\gamma(\nu)}(X_{0,\nu}|\partial X_{0,\nu})]^{\vir},
\]
where $\leftrightsquigarrow$ means they connected by some
common refinement;
see Theorem \ref{coro-zero-degeneration} for details.
\end{itemize}
\end{thm}

\begin{nota}
Note that the multiplicities are retained in the relation $\leftrightsquigarrow$.
Now this theorem answers Question \ref{Q-dege-log-rel} in the zero-dimensional case.
\end{nota}

The preceding degeneration formula provides a tool for computing
zero-dimensional logarithmic $\mathsf{DT}_4$-invariants with
tautological insertions. We first formulate the expected closed
formula for these invariants.

\begin{conj}\label{conj-log-zero-dim-DT}
Let $(X,D)$ be a log Calabi-Yau pair with reduced snc divisor $D$ such that $\dim X=4$.
Assume $\bb T$ is a possibly trivial Calabi-Yau torus acting on $X$ such that 
$\Hilb^n(X|D)^{\bb T}$
is proper for all $n$. Let $\scr L$ be a $\bb T$-equivariant line bundle.
Then there exists a choice of orientation such that
\begin{align*}
	\mathsf{DT}_0(X|D)_{\scr L,m}&:=1+\sum_{n>0}q^n
\int_{[\Hilb^n(X|D)]^{\vir}_{\bb T}}
e^{\bb T\times\bb C^*_m}\left((\scr L^{[n]})^{\vee}\otimes e^m\right)\\ 
&=M(-q)^{\int_Xc_1^{\bb T\times\bb C^*_m}(\scr L\otimes e^{-m})c_3^{\bb T}(T_X(-\log D))},
\end{align*}
where $M(q)=\prod_{n\geq1}(1-q^n)^{-n}$ is the MacMahon function, $\scr L^{[n]}$ is the tautological bundle and $\bb C^*_m$ is a trivial
torus with weight $m$; see Definition \ref{def-0dim-log-DT4} for details.
\end{conj}

When $D=\emptyset$, Conjecture \ref{conj-log-zero-dim-DT} specializes to the conjectural 
formula of Cao-Kool \cite{CK18}. 
Cao-Kool verified this formula for small numbers of points when $\scr L=\scr O_X(E)$ 
for a smooth divisor $E$. Under the same type of hypothesis, with 
$E$ smooth and connected and $X$ projective, 
Park \cite[Cor.~3.3]{park21} proved the formula for all numbers of points using the Lefschetz 
principle. 
Using the
Gross-Joyce-Tanaka wall-crossing conjecture, which was recently
proved by Bojko-Kuhn-Liu-Thimm \cite{BKLT26}, Bojko \cite{Bojko24} proved the formula for arbitrary line bundles
on projective Calabi-Yau $4$-folds.
All these results are proved under the assumption that the required orientations
exist\footnote{See \cite{CGJ20,CGJ25} and \cite{JU25} for results on the existence of orientations.}.

In the local Calabi-Yau case, where canonical orientations always exist, the formula for 
$(\bb C^4,\emptyset;\scr L)$ was established by two approaches: 
by Cao-Zhao-Zhou \cite[Thm.~1.13]{CZZ24} using degeneration methods, and by 
Kool-Rennemo \cite[Thm.~1.1]{KR25} using factorizability.
This was extended in 
\cite[Thm.~A]{Liu25} to $(X,\emptyset;\scr L)$ with $X=Y\times\bb C$ for any toric Calabi-Yau 
$3$-fold $Y$.

With nonempty logarithmic boundary, Cao-Zhao-Zhou
\cite[Thm.~6.16]{CZZ24} proved Conjecture
\ref{conj-log-zero-dim-DT} for the local curves
$
(\mathrm{Tot}_C(\scr L_1\oplus\scr L_2\oplus\scr L_3),
 \sqcup_{i=1}^r\{p_i\}\times\bb C^3;\scr L),
$
where $C$ is a projective curve,
$\scr L_1\otimes\scr L_2\otimes\scr L_3\cong\omega_C(p_1+\cdots+p_r)$,
and $\scr L$ is pulled back from $C$.

As the main application of Theorem
\ref{coro-zero-degeneration-intro}, we combine the degeneration
formula with the logarithmic cobordism groups of Guzman \cite{Guz25} to
verify Conjecture \ref{conj-log-zero-dim-DT} for local surfaces.

\begin{thm}[Theorem \ref{thm-total-surface-log-DT4}]
Consider an snc pair $(S,C)$ with $S$ a smooth projective surface. Let $X=\mathrm{Tot}(\omega_S(C))\times\bb C$
and $D=\mathrm{Tot}(\omega_C)\times\bb C$ and 
consider a Calabi-Yau torus $\bb T=\{t_1t_2=1\}\subset(\bb C^*)^2_{t_1,t_2}$ acting on the fibres of these rank-two vector bundles.
Let $\scr L=\eta^*\scr L'$ for $\scr L'\in\mathrm{Pic}(S)$ and $\eta:\mathrm{Tot}(\omega_S(C))\times\bb C\to S$.
Then Conjecture \ref{conj-log-zero-dim-DT} holds.
\end{thm}
As a special case, we compute the zero-dimensional $\mathsf{DT}_4$-invariants 
for $K_S\times\bb C$ for any smooth projective surface $S$ (possibly non-toric).

\subsection{Related works and future directions}
For completeness, Appendix \ref{appendix-GLSM} records the
corresponding splitting and genus reduction formulae for GLSMs
as another application.

The first direction is to develop logarithmic surface-counting theories
for (log-) Calabi-Yau $4$-folds, extending the theories of \cite{BKP22,BKP24},
and to establish the corresponding degeneration formulae. This would require an analogue of the
logarithmic Hilbert scheme for $2$-dimensional subschemes, whose
degenerations should be encoded by $2$-dimensional polyhedral
complexes in the tropicalization of the target, rather than the
tropical curves governing the curve-counting theory.
The construction of such logarithmic surface-counting theories
is the subject of the forthcoming work of Pim Spelier.

A second direction is to extend the theory to the logarithmic Quot schemes described in \cite{KHR25}
and the moduli space of higher rank coherent sheaves in logarithmic
geometry developed in \cite{DHK+26}.

\subsection*{Conventions and Notation}
\begin{itemize}
	\item All schemes and stacks are defined over $\bb C$. We always consider the counting theory 
	of varieties with $\dim\geq3$.
	\item All constructible sheaves, (co-) homology groups and Chow groups have $\bb Q$-coefficients.
	\item For a Gorenstein space $X$, $\omega_X$ always denotes the coherent dualizing sheaf.
	\item Throughout the paper, all logarithmic Hilbert stacks parameterize properly supported subschemes of dimension at most one, whose numerical classes lie in
$K_{\mathrm{cpt},\leq1}^{\mathrm{num}}(-)$. In the $1$-dimensional case, we write a numerical class
as $(\beta,\chi)$, where $\beta$ is the curve class and $\chi$ is the Euler characteristic.
\end{itemize}

\subsection*{Acknowledgments}

The author would like to thank Professor Yalong Cao for his continued support, 
encouragement, and generous help throughout this project. 
He is also grateful to Martijn Kool, Hyeonjun Park and Pim Spelier for valuable discussions and 
comments on the draft of the paper.
In particular, he thanks
Hyeonjun Park for clarifying the details of the forthcoming work \cite{KKPS3} on 
shifted Lagrangian classes. 

\subsection*{AI Usage Declaration}

During the preparation of this manuscript, the author used GPT-5.6 Sol and GPT-6 Astra only for language polishing and 
for assisting in checking the mathematical proofs. 
The author independently reviewed and verified all AI-assisted changes and all mathematical 
results, and takes full responsibility for the content of the manuscript.

\section{Preliminaries}
\subsection{Six-functor formalism of (ind-)constructible sheaves}
We collect the aspects of the six-functor formalism for ind-constructible sheaves on derived Artin stacks that will be used below.

\begin{thm}\label{thm-six-functor}
For any derived Artin stack $\mathcal X$ locally of finite type over
$\bb C$, there exists a stable presentable $\bb Q$-linear
	$\infty$-category \[\mathbf D(\mathcal X):=\mathbf D(\mathcal X_{\mathrm{cl}})\]
such that the following statements hold.
\begin{enumerate}
	\item There is a canonical equivalence
	\[
	\mathbf D(\mathcal X)
	\simeq
	\lim_{(U\to \mathcal X_{\mathrm{cl}})\in
	(\mathsf{Sch}^{\mathrm{sep,ft}}_{\bb C}/\mathcal X_{\mathrm{cl}})}
	\operatorname{Ind}\bigl(\mathbf D_c^b(U,\bb Q)\bigr),
	\]
	where the transition functors are given by pullbacks.

	\item For every morphism
	$f:\mathcal X\to\mathcal Y$ of derived Artin stacks locally of
	finite type, there are functors
	\[
	f^*,f^!:
	\mathbf D(\mathcal Y)\longrightarrow\mathbf D(\mathcal X),
	\qquad
	f_*,f_!:
	\mathbf D(\mathcal X)\longrightarrow\mathbf D(\mathcal Y),
	\]
	together with the tensor product and internal Hom. They satisfy
	the usual adjunctions
	\[
	f^*\dashv f_*,
	\qquad
	f_!\dashv f^!,
	\]
	as well as the usual composition isomorphisms.

	\item More precisely, consider a diagram
	\[
	\begin{tikzcd}
		\mathcal X' \arrow[r,"g"] \arrow[d,"{f'}"']&
		\mathcal X \arrow[d,"f"]\\
		\mathcal Y' \arrow[r,"h"']&
		\mathcal Y
	\end{tikzcd}
	\]
	whose induced diagram of classical truncations is Cartesian,
	there are canonical exchange isomorphisms
	\[
	\mathrm{Ex}_{!}^{*}:
	(f')_!g^*
	\xrightarrow{\ \cong\ }
	h^*f_!,
	\qquad
	\mathrm{Ex}_{*}^{!}:
	f^!h_*
	\xrightarrow{\ \cong\ }
	g_*(f')^!.
	\]
	The second isomorphism is the right-adjoint mate of the first one.
	These exchange isomorphisms are compatible with compositions of
	Cartesian diagrams.

	\item If $f:\mathcal X\to\mathcal Y$ is schematic and proper, then
	$f$ is cohomologically proper. In particular, the canonical
	natural transformation
	\[
	f_!\longrightarrow f_*
	\]
	is an isomorphism.

	\item If $f:\mathcal X\to\mathcal Y$ is smooth of pure relative
	dimension $d$, then there is a canonical purity isomorphism
	\[
	f^!\cong f^*[2d].
	\]

	\item If $X$ is a separated scheme of finite type over $\bb C$,
	then
	\[
	\mathbf D(X)
	=
	\operatorname{Ind}\bigl(\mathbf D_c^b(X,\bb Q)\bigr).
	\]
	Moreover, for a morphism $f:X\to Y$ between separated schemes of
	finite type, the six operations restrict to the usual bounded
	constructible derived categories. See
	\cite[Prop.~2.9]{KKPS1} for these assertions.
\end{enumerate}
\end{thm}
\begin{proof}
(1) and (2) follow from \cite[Prop.~2.10]{KKPS1}.
(4) follows from 
	\cite[\S~2.2, after Prop.~2.15]{KKPS1}.
(5) follows from \cite[Prop.~2.14]{KKPS1}.
(6) follows from \cite[Prop.~2.9]{KKPS1}.

For (3), the existence of the lisse-extended preweave on Artin stacks
	follows from \cite[Prop.~2.10]{KKPS1}, while the invertibility of
	the exchange transformations follows formally from the preweave
	structure; see
	\cite[Appendix~A.2.1, formulas~(A.7)--(A.8)]{Kha25}.
\end{proof}

\begin{nota}[{\cite[\S~2.2, after Prop.~2.10]{KKPS1}}]
	An object $F\in\mathbf D(\mathcal X)$ is called
	\textit{constructible} if, for every smooth morphism
	$u:U\to \mcal X_{\mathrm{cl}}$ from a separated scheme of finite type over $\bb C$,
	one has
	\[
	u^*F\in \mathbf D_c^b(U,\bb Q).
	\]
	We denote the full subcategory of such objects by
	$\mathbf D_c^b(\mathcal X,\bb Q)$.
For a general Artin stack $\mathcal X$, the six operations are not asserted to preserve
$\mathbf D_c^b$ for arbitrary morphisms of Artin stacks; see also \cite[Rmk.~2.20]{KKPS1}.

There is a $t$-structure on $\mathbf D_c^b(\mathcal X,\bb Q)$ such that for every smooth morphism 
$u:U\to \mcal X$ from a separated scheme of finite type over $\bb C$ the functor $u^![-\dim u]$ is $t$-exact.
This defines the notion of \textit{perverse sheaves} on Artin stacks.
\end{nota}

\begin{defi}
For any derived Artin stack $\mathcal X$ locally of finite type over
$\bb C$, we define its \textit{Borel-Moore homology} as 
\[
H^{\BM}_k(\mcal X,\bb Q):=H^{-k}(\mcal X,\bb D\bb Q_{\mcal X}).
\]
\end{defi}
\begin{nota}\label{rmk-constructions-BM}
We will use the following constructions later:
\begin{enumerate}
	\item For any $f:\mcal X\to \mcal Y$ and $F\in \mathbf D(\mcal Y)$, any morphism 
	$\psi:f^*F[k]\to\bb D\bb Q_{\mcal X}$ induces
	\begin{align*}
    &H^n(\mathcal Y,F)\to H^n(\mathcal Y,f_*f^*F)=H^n(\mathcal X,f^*F)
    \xrightarrow{\psi}H^{n-k}(\mathcal X,\bb D\bb Q_{\mcal X})=H^{\BM}_{k-n}(\mcal X,\bb Q).
    \end{align*}
	\item For a cohomologically proper morphism $f:\mcal X\to \mcal Y$ (for example, if it is 
	a schematic proper morphism), the pushforward $f_*:H^{\BM}_k(\mcal X)\to H^{\BM}_k(\mcal Y)$ is induced by
    \[
        (\bb Q_{\mcal X}[k]\xrightarrow{\alpha}\bb D\bb Q_{\mcal X})\mapsto
        (\bb Q_{\mcal Y}[k]\to f_*\bb Q_{\mcal X}[k]
        \xrightarrow{\alpha} f_*\bb D\bb Q_{\mcal X}\cong f_!f^!\bb D\bb Q_{\mcal Y}\to\bb D\bb Q_{\mcal Y}).
    \]
\end{enumerate}
\end{nota}

\subsection{Family version of shifted Lagrangian classes}
Here we state the main conjecture we will use in general logarithmic $\mathsf{DT}_4$-theory,
generalizing the original conjecture due to Joyce in \cite{JS19,ABB17}.
See \cite{PTVV13,park24,PY26} for the detailed definitions of shifted symplectic 
structures and shifted Lagrangian correspondences.

\begin{thm}[Perverse pullbacks, {\cite[Thm.~5.24,~\S~6.2]{KKPS1}}]\label{thm-perverse-pullback}
Let $\pi:X\to B$ denote an lfp morphism between derived Artin stacks equipped with a relative exact $(-1)$-shifted symplectic structure
and an orientation.
Then there exists a $t$-exact \textit{perverse pullback} functor:
\[
\pi^{\varphi}:\mathbf D(B)\to\mathbf D(X)
\]
with the following properties.
\begin{itemize}
	\item On local Darboux charts, it restricts to the vanishing-cycle functor.
	\item It is compatible with smooth pullback, finite pushforward, the Thom--Sebastiani isomorphism, and Verdier duality.
\end{itemize}
\end{thm}

\begin{conj}[Family version of shifted Lagrangian classes]\label{conj-Joyce-conj}
Consider the following oriented $(-1)$-shifted exact Lagrangian correspondence over $B$:
\[\begin{tikzcd}
	& L & \\
	X && Y
	\arrow["s"', from=1-2, to=2-1]
	\arrow["t", from=1-2, to=2-3]
\end{tikzcd}\]
with associated morphism $\pi_X:X\to B$ and $\pi_Y:Y\to B$ and compatible orientations.
Then we have a canonical morphism 
\[
[L]^{\mathrm{Lag}}:s^*\pi_X^{\varphi}(-)[\vdim (L/B)]\to t^!\pi_Y^{\varphi}(-)
\] of functors from $\mathbf D(B)$ to $\mathbf D(L)$,
with the following properties.
\begin{enumerate}
	\item \textbf{Unitality:} If $s$ and $t$ are identity morphisms with induced diagonal Lagrangian structures, 
	then the desired morphism is the identity.
	\item \textbf{Associativity:} Assume we have the following composition of oriented $(-1)$-shifted exact Lagrangian correspondences over $B$:
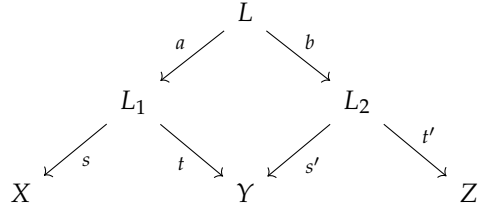
\[\begin{tikzcd}
	&& L && \\
	& {L_1} && {L_2} \\
	X && Y && Z
	\arrow["a"', from=1-3, to=2-2]
	\arrow["b", from=1-3, to=2-4]
	\arrow["s", from=2-2, to=3-1]
	\arrow["t"', from=2-2, to=3-3]
	\arrow["{s'}", from=2-4, to=3-3]
	\arrow["{t'}", from=2-4, to=3-5]
\end{tikzcd}\]
where the middle square is a cartesian diagram. Then we have the following commutative diagram:
\[\begin{tikzcd}
	{f^*\pi^{\varphi}_X(-)[\vdim(L/B)]} && {g^!\pi^{\varphi}_Z(-)} \\
	{a^*s^*\pi^{\varphi}_X(-)[\vdim(L/B)]} && {b^!(t')^!\pi^{\varphi}_Z(-)} \\
	{a^*t^!\pi^{\varphi}_Y(-)[\vdim(L_2/B)]} && {b^!(s')^*\pi^{\varphi}_Y(-)[\vdim(L_2/B)]}
	\arrow["{[L]^{\mathrm{Lag}}}", from=1-1, to=1-3]
	\arrow["\cong", from=1-1, to=2-1]
	\arrow["\cong"', from=1-3, to=2-3]
	\arrow["{a^*[L_1]^{\mathrm{Lag}}}", from=2-1, to=3-1]
	\arrow[from=3-1, to=3-3]
	\arrow["{b^![L_2]^{\mathrm{Lag}}}", from=3-3, to=2-3]
\end{tikzcd}\]
where $f=s\circ a$ and $g=t'\circ b$.
\item \textbf{Base-change via finite morphisms:} Consider a cartesian diagram 
\[\begin{tikzcd}
	{L_S} & {X_S} & S \\
	L & X & B
	\arrow["{u_S}", from=1-1, to=1-2]
	\arrow["{c_L}", from=1-1, to=2-1]
	\arrow["\square"{description}, draw=none, from=1-1, to=2-2]
	\arrow["{\pi_{X_S}}", from=1-2, to=1-3]
	\arrow["{c_X}", from=1-2, to=2-2]
	\arrow["\square"{description}, draw=none, from=1-2, to=2-3]
	\arrow["c", from=1-3, to=2-3]
	\arrow["u", from=2-1, to=2-2]
	\arrow["{\pi_X}", from=2-2, to=2-3]
\end{tikzcd}\]
where $u:L\to X$ is an oriented $ (-1) $-shifted exact
Lagrangian morphism over $B$, and $u_S:L_S\to X_S$ is its
derived base change equipped with the induced Lagrangian structure and orientation.
Assume that $c$ is a finite morphism and 
let $t:L\to B$ and $t_S:L_S\to S$ denote the compositions.
Then we have the following commutative diagram 
\[\begin{tikzcd}
	{u^*\pi_X^{\varphi}c_*(-)[\vdim L/B]} && {t^!c_*(-)} \\
	{u^*c_{X,*}\pi_{X_S}^{\varphi}(-)[\vdim L/B]} \\
	{c_{L,*}u_S^*\pi_{X_S}^{\varphi}(-)[\vdim L/B]} && {c_{L,*}t_S^!(-)}
	\arrow["{[L]^{\mathrm{Lag}}c_*}", from=1-1, to=1-3]
	\arrow["\cong", from=1-1, to=2-1]
	\arrow["{\mathrm{Ex}_*^!,\cong}", from=1-3, to=3-3]
	\arrow["{\mathrm{BC}}", from=2-1, to=3-1]
	\arrow["{c_{L,*}[L_S]^{\mathrm{Lag}}}", from=3-1, to=3-3]
\end{tikzcd}\]
where the right vertical isomorphism is the exchange isomorphism in
Theorem~\ref{thm-six-functor}(3),
the upper-left isomorphism is given by \cite[Thm.~5.24(2)]{KKPS1} and $\mathrm{BC}$
is the Beck-Chevalley transform.
	\item \textbf{Comparison with Oh-Thomas' class:} Let $f:X\to Y$ be an oriented $(-2)$-shifted exact symplectic fibration which induces
an oriented $(-1)$-shifted Lagrangian structure $F:X\to\TT^*[-1]Y$.
From the construction in Remark \ref{rmk-constructions-BM}(1), the shifted Lagrangian class 
\[
F^*\varphi_{\TT^*[-1]Y}[\vdim X]\to \bb D\bb Q_X
\]
induce 
\[
[F]^{\mathrm{Lag}}:H^{\BM}_{k}(Y,\bb Q)\cong H^{\vdim Y-k}(\varphi_{\TT^*[-1]Y})
\to H^{\BM}_{\vdim(X/Y)+k}(X,\bb Q),
\]
where the first isomorphism follows from dimension reduction proved in 
\cite{kin22,KKPS2}, satisfies
\[[F]^{\mathrm{Lag}}=\mathrm{cl}\left(\sqrt{f^!}\right).\]
Here $\mathrm{cl}$ is the cycle class map $\mathrm{cl}:\CH_*\to H^{\BM}_*$
and $\sqrt{f^!}$ is the square-root virtual pullback defined 
in \cite{park21,park24}.
\end{enumerate}
\end{conj}

If $B=\mathrm{pt}$, statements (1) and (2) reduce to Joyce's original conjecture;
see also \cite{JS19,ABB17}. 

\begin{nota}\label{rmk-KKPS-lagrangian-classes}
Hyeonjun Park has informed the author that the cases of
Conjecture \ref{conj-Joyce-conj} needed in this paper are covered
by the forthcoming work \cite{KKPS3}.
These include Lagrangian classes over smooth bases evaluated on
the constant sheaf, together with associativity,
and the compatibility with finite base change.
\end{nota}

Here we recall that the logarithmic curve-counting theory in
Sections $\ref{sec-logDT4-curves}$ and \ref{sec-general-degeneration}, including 
its degeneration formula, depends on
shifted Lagrangian classes in Remark \ref{rmk-KKPS-lagrangian-classes}. 
The quasimap invariants and their
gluing formulae in Appendix \ref{appendix-GLSM} use its $F_0$-equivariant version.

By contrast, the zero-dimensional theory in Section \ref{sec-zero-dim-DT4} and its
applications in Section \ref{sec-applications} are constructed in Chow groups and are
independent of Conjecture \ref{conj-Joyce-conj}. Shifted Lagrangian
classes are used there only to provide parallel formulations in
Borel-Moore homology.

\section{Logarithmic Donaldson-Thomas theory for log Calabi-Yau \texorpdfstring{$4$}{4}-pairs}
\label{sec-logDT4-curves}
In this section, we construct the shifted-symplectic geometry of logarithmic Hilbert
schemes in both the relative and family settings by combining the frameworks of
\cite{CZZ24,MR24}, and define the logarithmic Donaldson-Thomas theory.
\subsection{About logarithmic Hilbert schemes}
Let $(X,D)$ be a simple normal crossings pair with divisor components $D_1,\ldots,D_r$.
\begin{itemize}
	\item For simplicity, we assume that every finite intersection of the
components $D_i$ is either connected or empty; see also \cite{MR24}.
This assumption can be removed by allowing disconnected tropicalizations.
\end{itemize}
One canonically associates a \textit{cone complex} $\Sigma_{X|D}$ to $(X|D)$ \cite[Def.~2.3.1]{MR25}:
it is the subfan of $\bb R^r_{\geq0}$ consisting of cones that correspond to nonempty intersections $D_I=\bigcap_{i\in I}D_i$.
A \textit{$1$-complex} in $\Sigma_{X|D}$ is an embedded polyhedral complex of dimension at most $1$;
a \textit{$0$-complex} is an embedded polyhedral complex of dimension $0$.
These combinatorial objects parameterize the combinatorial types of expansions.

\begin{nota}
Thus the strata of
codimension at least $2$ are removed.
If we only consider $0$-complexes, then all the codimension $\geq1$ loci will be removed in the expansions.
\end{nota}

The main construction of \cite{MR24} outputs, for each $(X,D)$, a moduli stack $\mathsf{Exp}(X|D)$ of \textit{expansions} of $X$ along $D$.
Its $\bb C$-points parameterize expansions in the sense of \cite[Def.~2.1.2]{MR25}:
unions of strata in the rank $1$ locus of reduced rough expansions of $(X,D)$.
The construction uses the dictionary between Artin fans and cone spaces \cite[Thm.~3]{MR24}.
Let $|T(\Sigma_{X|D})|$ be the topological moduli space of $1$-complexes in $\Sigma$.
The set $|T(\Sigma_{X|D})|$ admits a family of cone space structures $T(\Sigma_{X|D})_\lambda$, 
organized by a directed partially ordered set $(\Lambda,\prec)$, together with a universal 
$1$-complex $\Upsilon_\lambda\to T(\Sigma_{X|D})_\lambda$ that is combinatorially flat.
Passing to associated Artin fans, one obtains
\[
\mathsf{Exp}(X|D)_{\lambda}:=\varinjlim_{\sigma}A_\sigma=\varinjlim_{\sigma}\left[\spec\bb C[\mathsf{S}_{\sigma}]/\spec\bb C[M]\right],
\]
where $T(\Sigma_{X|D})_\lambda=\varinjlim_\sigma(\sigma,M)$ is the colimit of rational polyhedral cones
and $\mathsf{S}_{\sigma}:=\sigma^{\vee}\cap M$ is the dual monoid of $\sigma$.
By \cite[Prop.~2.3.2]{MR25}, the points of $\mathsf{Exp}(X|D)_{\lambda}$ are in natural bijection with cones of $T(\Sigma_{X|D})_\lambda$.

\begin{nota}[Compare with Li-Wu]
If $D$ is a smooth divisor, then the tropicalization is given by $\Sigma_{X|D}=\bb R_{\geq0}$.
Then Li-Wu's construction described in \cite{LW15,Zhou18} is given by the $1$-complexes of form 
$\{\{q_1<...<q_L\}\subset\bb R_{\geq0}\}$. The corresponding set of $1$-complexes $|T^{\mathrm{LW}}(\Sigma_{X|D})|\subset|T(\Sigma_{X|D})|$
has the canonical cone structure $T^{\mathrm{LW}}(\Sigma_{X|D})$, where two points lie in the same cone if the 
associated $1$-complexes have the same number of vertices. 
Then we have the corresponding stack of expansions $\mathsf{Exp}^{\mathrm{LW}}(X|D)\cong\mcal A$
and its universal family. See \cite[\S~6.7]{MR24} for further descriptions.

For the point-counting case, we can choose the subcone consisting of $0$-complexes
$\{q_1<...<q_L\}$. The resulting expanded pair is given by the disjoint union of $X\backslash D$
and $\{\widetilde{D}\}_{1,...,L}$ for $\bb C^*$-bundle $\widetilde{D}=\Tot(N_{D/X})\setminus D$ on $D$.
\end{nota}

\begin{exam}
For example, consider an snc pair $(\overline{X},D)$ such that $D=D_1\cup_{\ell}D_2$ has two components with connected 
intersection $\ell$. Then $\Sigma_{\overline{X}|D}=\bb R_{\geq0}^2$. Consider the following $1$-complex
on the left. The corresponding expansion is described on the right:
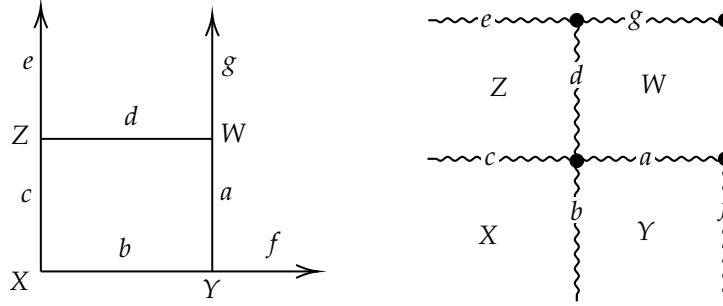
\begin{center}
\tikzset{every picture/.style={line width=0.75pt}}

\begin{tikzpicture}[x=0.75pt,y=0.75pt,yscale=-1,xscale=1]
\tikzset{
  geometric edge/.style={
    line width=0.8pt,
    decorate,
    decoration={
      snake,
      amplitude=0.75pt,
      segment length=5.5pt,
      pre length=1pt,
      post length=1pt
    }
  },
  deleted locus/.style={
    line width=0.75pt,
    dash pattern={on 0.84pt off 2.51pt}
  },
  vertex/.style={
    circle,
    fill=black,
    inner sep=0pt,
    minimum size=5.5pt
  },
  right edge label/.style={
    fill=white,
    fill opacity=1,
    text opacity=1,
    inner sep=1.1pt,
    outer sep=0pt
  }
}

\draw (42.5,203) -- (42.5,72);
\draw [shift={(42.5,70)},rotate=90,line width=0.75pt]
  (10.93,-3.29) .. controls (6.95,-1.4) and (3.31,-0.3) ..
  (0,0) .. controls (3.31,0.3) and (6.95,1.4) .. (10.93,3.29);

\draw (42.5,203) -- (179.5,203);
\draw [shift={(181.5,203)},rotate=180,line width=0.75pt]
  (10.93,-3.29) .. controls (6.95,-1.4) and (3.31,-0.3) ..
  (0,0) .. controls (3.31,0.3) and (6.95,1.4) .. (10.93,3.29);

\draw (42.5,136.5) -- (128.5,136.5);
\draw (128.5,136.5) -- (128.5,203);
\draw (128.5,136.5) -- (128.5,76);
\draw [shift={(128.5,74)},rotate=90,line width=0.75pt]
  (10.93,-3.29) .. controls (6.95,-1.4) and (3.31,-0.3) ..
  (0,0) .. controls (3.31,0.3) and (6.95,1.4) .. (10.93,3.29);

\draw[geometric edge] (385.89,77) -- (236.89,77);
\draw[geometric edge] (385.89,146.5) -- (236.5,146.5);
\draw[geometric edge] (311.39,77) -- (311.39,218);
\draw[geometric edge] (385.89,146.5) -- (385.89,219);
\draw[deleted locus] (385.89,77) -- (385.89,146.5);

\node[vertex] at (311.39,77) {};
\node[vertex] at (311.39,147.5) {};
\node[vertex] at (385.89,77) {};
\node[vertex] at (385.89,146.5) {};

\draw (25,201.4) node [anchor=north west][inner sep=0.75pt] {$X$};
\draw (122,205.4) node [anchor=north west][inner sep=0.75pt] {$Y$};
\draw (26,128.4) node [anchor=north west][inner sep=0.75pt] {$Z$};
\draw (131,127.4) node [anchor=north west][inner sep=0.75pt] {$W$};
\draw (131,159.4) node [anchor=north west][inner sep=0.75pt] {$a$};
\draw (80,184.4) node [anchor=north west][inner sep=0.75pt] {$b$};
\draw (31,160.4) node [anchor=north west][inner sep=0.75pt] {$c$};
\draw (83,118.4) node [anchor=north west][inner sep=0.75pt] {$d$};
\draw (32,93.4) node [anchor=north west][inner sep=0.75pt] {$e$};
\draw (154,181.4) node [anchor=north west][inner sep=0.75pt] {$f$};
\draw (132,93.4) node [anchor=north west][inner sep=0.75pt] {$g$};

\draw (260,178.4) node [anchor=north west][inner sep=0.75pt] {$X$};
\draw (340,176.4) node [anchor=north west][inner sep=0.75pt] {$Y$};
\draw (266,104.4) node [anchor=north west][inner sep=0.75pt] {$Z$};
\draw (342,102.4) node [anchor=north west][inner sep=0.75pt] {$W$};

\node[right edge label] at (346,146.5) {$a$};
\node[right edge label] at (311.39,172.4) {$b$};
\node[right edge label] at (268,146.5) {$c$};
\node[right edge label] at (311.39,104.4) {$d$};
\node[right edge label] at (266,77) {$e$};
\node[right edge label] at (385.89,171.4) {$f$};
\node[right edge label] at (341,77) {$g$};

\end{tikzpicture}
\end{center}
Note that the direction of $b$ (resp. $c$) corresponds to $D_1$ (resp. $D_2$)
and the black bullets indicate the loci of codimension $\geq2$, which are removed.
The component $X$ is the main component and is isomorphic to $\overline X\backslash\ell$.
Moreover, we have $Y\subset\overline{Y}\cong\bb P_{D_1}(N_{D_1/\overline{X}}\oplus\scr O)$, 
$Z\subset\overline{Z}\cong\bb P_{D_2}(N_{D_2/\overline{X}}\oplus\scr O)$,
and 
\[
W\subset\overline W\cong
\bb P_{\ell}(N_{D_1/\overline X}|_{\ell}\oplus\scr O_{\ell})
\times_{\ell}
\bb P_{\ell}(N_{D_2/\overline X}|_{\ell}\oplus\scr O_{\ell}).
\]
by removing the boundary divisor shown as a
dashed line, together with all strata of codimension at least two.
\end{exam}
A cone structure on $T(\Sigma_{X|D})$ determines the expansion
stack, whereas the construction of its universal family requires
an additional compatible combinatorial choice.
Precisely, there is another partially ordered set $L$ and a system of spaces 
$\{\mathfrak X_{\ell}\}_{\ell\in L}$. 
For each choice of $\ell$ and $\lambda$,
we say they are \textit{compatible} if there exists a natural projection 
$\mathfrak X_{\ell}\to\mathsf{Exp}(X|D)_{\lambda}$ which is flat with reduced fibres.

A component of $X_p$ is called \textit{linear $2$-valent} if its vertex
is $2$-valent and, near that vertex, its two incident edges or rays lie
on a line in the relative interior of a single target cone.
A \textit{tube subscheme} on such a component is the schematic inverse
image of a zero-dimensional subscheme under its $\bb P^1$-bundle projection.

Another notion is \textit{tube components}: certain distinguished $\bb P^1$-bundle components in each expansion fibre,
depending on the choice of $\ell\in L$.
Here we note that a linear $2$-valent component need not be a tube component.

The locus of rank $0$ expansions is an open substack
$\mathsf{Exp}_0(X|D)\subset\mathsf{Exp}(X|D)$,
corresponding to the cone subspace $P(\Sigma_{X|D})\subset T(\Sigma_{X|D})$
parameterizing $0$-complexes.

We will drop $\lambda$ and $\ell$ if we have chosen such structures. We always 
choose so that $\mathsf{Exp}(X|D)_{\lambda}$ is smooth.
With the stack $\mathsf{Exp}(X|D)$ established, one can define moduli spaces of sheaves over it \cite[\S~2.4]{MR25}.
\begin{defi}
Let $X_p$ be an expansion over $\spec\bb C$ (the fibre over a point $p\in\mathsf{Exp}(X|D)$).
A properly supported subscheme $Z\subset X_p$ of dimension at most $1$ is \textit{algebraically transverse} if it meets all locally closed strata properly and non-trivial, and for any codimension $1$ stratum $S$, the multiplication map $I_Z\otimes\mcal O_S\to\mcal O_{X_p}\otimes\mcal O_S$ is injective \cite[Def.~2.4.1]{MR25}.
It is \textit{DT-stable} if, for every linear $2$-valent
component $Y$, the restriction $Z\cap Y$ is a tube subscheme if and only if
$Y$ is a tube component; see \cite[\S~4.1, Def.~4.2.1]{MR24}.
Then we can define the \textit{logarithmic Hilbert scheme} 
\[
\Hilb(X|D)\subset\Hilb(\mathfrak X/\mathsf{Exp}(X|D))
\]
which parameterizes maps $S\to\mathsf{Exp}(X|D)$ together with properly supported algebraically transverse and DT-stable families of subschemes;
see also \cite[Def.~2.4.3]{MR25} for details.
Let $\Hilb_{\beta,\chi}(X|D)$ be the open and closed substack with fixed pushforward curve class $\beta$ and holomorphic 
Euler characteristic $\chi$.
\end{defi}

The following is a fundamental result on logarithmic Hilbert stacks 
due to Maulik-Ranganathan.

\begin{thm}[{\cite{MR24},\cite[Thm.~2.4.5]{MR25}}]
The stack $\Hilb_{\beta,\chi}(X|D)$ is a separated Deligne--Mumford stack of finite type over $\bb C$.
If $X$ is proper, then $\Hilb_{\beta,\chi}(X|D)$
is a proper Deligne--Mumford stack.
The analogous statements hold for $\PP_{\beta,\chi}(X|D)$.
\end{thm}

In order to construct the logarithmic Donaldson-Thomas theory,
we need the following result, which is a natural generalization 
of \cite[Prop.~2.5]{CZZ24} and shows that the universal pair
preserves the log Calabi-Yau property after fixing a compatible
cone structure and universal family satisfying the flatness and
reducedness conditions below.

\begin{lem}\label{lem:log-crepant-MR-pair}
Let $(X,D)$ be a smooth quasi-projective variety with $D$ a reduced simple normal crossing
divisor. Fix a cone structure $\lambda$ such that $\Exp(X|D)_\lambda$
is smooth and the universal cone family
$\Upsilon_\lambda\to T_\lambda$ is combinatorially flat with reduced fibres, with
the compatible universal expansion
\[
        f:\mathfrak X\to\Exp(X|D)_\lambda.
\] 
Let $\mathfrak D\subset \mathfrak X$
be the universal strict transform of boundary divisor, and let
$p:\mathfrak X\to X\times \Exp(X|D)_\lambda$
be the contraction map. Then there is a canonical isomorphism
\[
        \omega_{\mathfrak X/\Exp(X|D)_\lambda}(\mathfrak D)
        \cong
        p^*\pr_X^*\omega_X(D).
\]
In particular, if $\omega_X(D)\cong \mathcal O_X$, any choice of
trivialization of $\omega_X(D)$ induces an isomorphism
\[
        \omega_{\mathfrak X/\Exp(X|D)_\lambda}(\mathfrak D)
        \cong
        \mathcal O_{\mathfrak X}.
\]
In this case, for each original component $D_i\subset D$ with corresponding 
universal boundary
$\mathfrak D_i\subset\mathfrak D$,
we have
$\omega_{\mathfrak D_{i}/\Exp(X|D)_\lambda}
        \cong
        \mathcal O_{\mathfrak D_{i}}$.
\end{lem}

\begin{proof}
Consider the canonical logarithmic structures
on $\Exp(X|D)_\lambda$, $(X,D)$ and $X\times\Exp(X|D)_\lambda$.
Checking locally in smooth toroidal charts, we know that 
$f$ is flat and Gorenstein by the fact that we consider the rank one locus.
In this case 
the horizontal boundary $\mathfrak D$ is a disjoint union of
relative effective Cartier divisors $\mathfrak D_i$ over
$\Exp(X|D)_\lambda$.

Let $E_{\Exp(X|D)_\lambda}$ be the toroidal boundary of
$\Exp(X|D)_\lambda$, then the toroidal boundary of
$\mathfrak X$ is $\partial\mathfrak X=E_{\mathfrak X}+\mathfrak D$, where
$E_{\mathfrak X}$ is the union of the boundary divisors mapping into
$E_{\Exp(X|D)_\lambda}$. We first claim that
\[
        f^*E_{\Exp(X|D)_\lambda}=E_{\mathfrak X}.
\]
This is local and can be checked at the generic point
of a codimension-one toroidal boundary divisor of $\mathfrak X$ lying over
$E_{\Exp(X|D)_\lambda}$. Here the characteristic monoids have rank one, and
on the local chart it is given by $\mathbb N\to\mathbb N$, $1\mapsto m$, that is,
a primitive local equation of the corresponding component of
$E_{\Exp(X|D)_\lambda}$ pulls back to the $m$-th power of a primitive local
equation on $\mathfrak X$. The reduced-fibre condition imply that the induced map on the
corresponding rank-one lattices has primitive image. Hence $m=1$, and the
claim follows.

The contraction
$p:\mathfrak X\to X\times\Exp(X|D)_\lambda$ is the open restriction of the
logarithmic modification obtained from the cone subdivision defining the
universal expansion; see also \cite[\S3.7]{MR24}. Since open immersion 
and logarithmic modifications
are logarithmically étale, we have
\[
        \omega^{\log}_{\mathfrak X/\Exp(X|D)_\lambda}
        \cong
        p^*\omega^{\log}_{X\times\Exp(X|D)_\lambda/\Exp(X|D)_\lambda}
		\cong
        p^*\pr_X^*\omega_X(D).
\]
On the other hand, we have
\[
\begin{aligned}
        \omega^{\log}_{\mathfrak X/\Exp(X|D)_\lambda}
        &=
        \omega_{\mathfrak X}(E_{\mathfrak X}+\mathfrak D)
        \otimes
        f^*\omega_{\Exp(X|D)_\lambda}(E_{\Exp(X|D)_\lambda})^{-1}  \\
        &\cong
        \omega_{\mathfrak X}\otimes f^*\omega_{\Exp(X|D)_\lambda}^{-1}
        \otimes\mathcal O_{\mathfrak X}(\mathfrak D)=
        \omega_{\mathfrak X/\Exp(X|D)_\lambda}(\mathfrak D),
\end{aligned}
\]
where we used $f^*E_{\Exp(X|D)_\lambda}=E_{\mathfrak X}$. Combining the two
identifications gives
\[
        \omega_{\mathfrak X/\Exp(X|D)_\lambda}(\mathfrak D)
        \cong
        p^*\pr_X^*\omega_X(D).
\]
If $\omega_X(D)\cong\mathcal O_X$, then any chosen trivialization gives
$\omega_{\mathfrak X/\Exp(X|D)_\lambda}(\mathfrak D)\cong
\mathcal O_{\mathfrak X}$.

Finally, for each original component $D_i\subset D$, let $\mathfrak D_i$ be
its universal strict transform in $\mathfrak X$. This is the induced expansion
of $(D_i,\partial D_i)$. Now $\mathfrak D_i$ is flat over $\Exp(X|D)_\lambda$. Relative adjunction gives
\[
        \omega_{\mathfrak D_i/\Exp(X|D)_\lambda}
        \cong
        \left(
        \omega_{\mathfrak X/\Exp(X|D)_\lambda}(\mathfrak D_i)
        \right)\big|_{\mathfrak D_i}.
\]
Since the codimension $\geq2$ strata of $X$ have
been removed, the different horizontal boundary pieces $\mathfrak D_j$
are disjoint from $\mathfrak D_i$, and
$\mathcal O_{\mathfrak X}(\mathfrak D_i)|_{\mathfrak D_i}\cong
\mathcal O_{\mathfrak X}(\mathfrak D)|_{\mathfrak D_i}$. Hence
\[
        \omega_{\mathfrak D_i/\Exp(X|D)_\lambda}
        \cong
        (p^*\pr_X^*\omega_X(D))\big|_{\mathfrak D_i} \cong\scr O_{\mathfrak D_i}.
\]
This finishes the proof.
\end{proof}

\subsection{Derived enhancements of logarithmic Hilbert schemes}
Fix an snc pair $(X,D)$ with universal family $\mathfrak X\to\mathsf{Exp}(X|D)$.
We can equip $\Hilb(X|D)$ and $\Hilb_{\beta,\chi}(X|D)$ with canonical derived enhancements as follows. 
Here we need an elementary lemma.

\begin{lem}\label{lem-snc-compactification}
There exists an open immersion
$X\hookrightarrow\overline X$ into a smooth projective variety
such that, writing $\overline D_i$ for the closure of $D_i$,
$\overline D=\bigcup_i\overline D_i$, 
for every nonempty subset $I\subseteq\{1,\ldots,r\}$,
we have $\bigcap_{i\in I}\overline D_i=\overline{\bigcap_{i\in I}D_i}$.
Hence taking closures induces bijections between the irreducible
components of the corresponding intersections, compatible
with their incidence relations.
\end{lem}

\begin{proof}
Choose a projective compactification of $X$.
By log resolution, we obtain
a smooth projective compactification $\overline X$ such that
$E=\overline X\setminus X$ is a divisor and $\overline D+E$ is snc.

Fix $I$ and let $W$ be an irreducible component of
$\bigcap_{i\in I}\overline D_i$.
Then $W$ has codimension $|I|$ in $\overline X$.
If $W\cap X=\emptyset$, then $W$ is contained in an irreducible
component $E_a$ of $E$. This is impossible, since the snc
condition implies that
$\bigcap_{i\in I}\overline D_i\cap E_a$, if nonempty, is pure
of codimension $|I|+1$.
Thus every such $W$ meets $X$.

Since $\bigcap_{i\in I}\overline D_i$ is smooth, its irreducible
components are disjoint. Intersecting them with $X$ therefore
gives precisely the irreducible components of
$\bigcap_{i\in I}D_i$, proving the asserted equality and
bijections.
\end{proof}

Fix a compactification $(X,D)\subset(\overline X,\overline D)$
as in Lemma~\ref{lem-snc-compactification}, and 
a choice of cone structure on $\Exp(X|D)_{\lambda}$.
Following \cite[Rmk.~3.5.2]{MR24}, we work over finite-type open
substacks of $\Exp(X|D)_{\lambda}$ and suppress this restriction from the notation.
After compatible subdivisions of the universal cone complex and the
base, we may complete the universal embedded $1$-complex to a \textit{complete}
projective subdivision over $\Sigma_{\overline X|\overline D}=\Sigma_{X|D}$.
By further subdivisions and, if necessary, root constructions,
we can construct the resulting map of cone stacks of 
the full universal family of expansions of
$(\overline X,\overline D)$:
\[
    \pi^c:\mathfrak X^c\longrightarrow\Exp(X|D)
\]
which is projective, combinatorially flat with reduced fibres and the base
is smooth; see \cite[Prop.~1.7.2]{MR25}.
We continue to denote the resulting expansion stack by $\Exp(X|D)$.

Define $\mathfrak X\subset\mathfrak X^c$ to be the 
open substack correspond to the rank-one universal family of $(X,D)$.
The composite
\[
    \pi:\mathfrak X\hookrightarrow\mathfrak X^c
    \xrightarrow{\pi^c}\Exp(X|D)
\]
is still combinatorially flat with reduced fibres.
In particular, $\pi$ is flat and its geometric fibres are reduced,
being open subschemes of the corresponding fibres of $\pi^c$.
All expansion stacks and universal families below are understood
with these refined cone and integral structures.

Consider the derived stack of perfect complexes $\mathsf{dPerf}$ and the derived 
mapping stack
\[
\mathsf{dPerf}(\mathfrak X^c/\mathsf{Exp}(X|D))
:=\mathsf{dMap}_{\mathsf{Exp}(X|D)}(\mathfrak X^c,\mathsf{dPerf}\times\mathsf{Exp}(X|D)).
\]
Similarly, define $\mathsf{dPic}(\mathfrak X^c/\mathsf{Exp}(X|D)):=\mathsf{dMap}_{\mathsf{Exp}(X|D)}(\mathfrak X^c,\mathsf{dPic}\times\mathsf{Exp}(X|D))$.
Then we have the following homotopical pullback diagram 
of derived stacks:
\[\begin{tikzcd}
	{\mathsf{dPerf}(\mathfrak X^c/\mathsf{Exp}(X|D))_0} & {\{\scr O\}} \\
	{\mathsf{dPerf}(\mathfrak X^c/\mathsf{Exp}(X|D))} & {\mathsf{dPic}(\mathfrak X^c/\mathsf{Exp}(X|D)).}
	\arrow[from=1-1, to=1-2]
	\arrow[from=1-1, to=2-1]
	\arrow["\square"{description}, draw=none, from=1-1, to=2-2]
	\arrow[from=1-2, to=2-2]
	\arrow["\det", from=2-1, to=2-2]
\end{tikzcd}\]

\begin{lem}\label{lem-derived-hilb}
With the compactification chosen above, there exists an open substack
\[
    \dHilb(X|D)\subset\mathsf{dPerf}(\mathfrak X^c/\Exp(X|D))_0
\]
such that $\left(\dHilb(X|D)\right)_{\mathrm{cl}}\cong\Hilb(X|D)$.
Moreover, this derived enhancement is independent
of the compactification up to a common refinement of the stacks of expansions
and the universal families. 
The same statements hold for the open and closed substack $\dHilb_{\beta,\chi}(X|D)$.
\end{lem}
\begin{proof}
We first claim that the morphism
\[
    \Hilb(X|D)\longrightarrow
    \mathsf{Perf}^{\geq0}(\mathfrak X^c/\Exp(X|D))_{0,\mathrm{rk}=1},
    \qquad \mathfrak Z\longmapsto I_{\mathfrak Z/\mathfrak X^c},
\]
is an open immersion, where the target consists of universally gluable
perfect complexes of virtual rank one with trivialized determinant.
A family $\mathfrak Z\subset\mathfrak X_S$ proper over $S$ is also closed
in $\mathfrak X^c_S$ by definition.
The morphism is well-defined for the following reasons:
\begin{itemize}
\item On $\mathfrak X_S$, algebraic transversality and the local Li-Wu model
imply that the ideal sheaf is perfect; see \cite[Cor.~5.9]{Wu07}.
Away from $\mathfrak Z$ it is the structure sheaf, so it is perfect on $\mathfrak X^c_S$.
\item It has a canonical trivialization of its determinant, since
$\mathfrak Z$ has fibrewise codimension at least two.
\end{itemize}

By the local argument of \cite[Lem.~2.27]{CZZ24}, the ideal sheaves of subschemes
contained in $\mathfrak X$ and normal to its codimension-one strata
are characterized by the following conditions on $F$:
\begin{itemize}
\item $F$ is a flat family of torsion-free rank-one coherent sheaves,
normal to the codimension-one strata of $\mathfrak X$.
\item $F$ is a structure sheaf near $\mathfrak X^c\setminus\mathfrak X$.
\item We also require the corresponding subschemes to have fibrewise dimension
at most one; this is an open condition on the resulting proper Hilbert functor.
\end{itemize}
These conditions are open on the proper family. In particular, the failure locus
of the second condition has closed image in the base, since the complement
$\mathfrak X^c\setminus\mathfrak X$ is proper over $\Exp(X|D)$;
see also \cite[\href{https://stacks.math.columbia.edu/tag/0DPE}{Tag~0DPE}]{SP26}.
Moreover, DT-stability gives a further open substack by \cite[Lem.~4.3.2]{MR24},
proving the claim. The existence of derived enhancement follows from \cite[Prop.~2.1]{STV15}.

Finally, we consider the independence of the compactification.
Consider two compactifications with associated cone structures $\lambda$ and $\mu$.
By \cite[Prop.~3.6.1]{MR24}, after further subdivisions and root constructions,
we obtain a common refinement satisfying the preceding conditions:
\[
    \Exp(X|D)_\lambda\longleftarrow\Exp(X|D)_\varphi
    \longrightarrow\Exp(X|D)_\mu.
\]
Pullback their two universal families, then
they contain the same rank-one universal family.
We denote them as
\[
    \mathfrak X\hookrightarrow\mathfrak X_a^c
    \longrightarrow\Exp(X|D)_{\varphi},\qquad a=1,2.
\]
Let $S\to\Exp(X|D)$ be affine derived, and let $\mathcal I$ be a perfect complex
on $\mathfrak X_S$ with trivialized determinant, whose classical restriction is
$I_Z$ for $Z\in\Hilb(X|D)(S_{\mathrm{cl}})$, with its canonical determinant trivialization.
Now $Z$ closed in $(\mathfrak X_a^c)_{S_{\mathrm{cl}}}$ for each $a$,
consider their complements $V_a\subset(\mathfrak X_a^c)_S$.
On $\mathfrak X_S\cap V_a$, the complex $\mathcal I$ is a line bundle
by \cite[Rmk.~2.3]{PS23}, and its determinant trivialization gives
$\mathcal I|_{\mathfrak X_S\cap V_a}\cong\mathcal O_{\mathfrak X_S\cap V_a}$.
Descent glues $\mathcal I$ and $\mathcal O_{V_a}$ to a perfect complex
with trivialized determinant on $(\mathfrak X_a^c)_S$,
whose classical restriction is the ideal sheaf of $Z$.
Conversely, every point of the corresponding enhancement is canonically trivial on $V_a$.
Thus extension and restriction are inverse equivalences of spaces, naturally in $S$.
Both compactifications therefore represent the same functor on the common
refined cone structure, proving the asserted independence.
\end{proof}

\begin{defi}\label{def-compactifiable-cone-structures}
A pair of cone structures
on the stack of expansions and its universal family is called
\emph{compactifiable} if it admits a full projective containing family
as in the construction preceding Lemma \ref{lem-derived-hilb}.
\end{defi}
\begin{nota}\label{rmk-intrinsic-derived-hilb} 
By the proof of Lemma \ref{lem-derived-hilb}, the derived enhancement
has the following intrinsic description. For an affine derived scheme $S$ over the expansion stack, 
its $S$-points are perfect complexes on the original universal family over $S$, 
with trivialized determinant, whose restrictions to $S_{\mathrm{cl}}$ are ideal sheaves 
of families in the classical logarithmic Hilbert stack, with their canonical 
trivialized determinant. Once the original universal family is fixed, the same 
extension-and-restriction argument applies to any flat projective family containing 
it as an open substack: the complexes are extended by the structure sheaf 
away from their classical support.
\end{nota}

For $(X,D)$, all subsequent choices of cone structures and common refinements
are taken to be compactifiable.

\begin{nota}[Derived logarithmic moduli space for stable pairs]
	\label{rmk-stable-pairs}
Here we consider the case of stable pairs in \cite[Def.~2.4.4]{MR25}.
We expect that this can be generalized to $Z_t$-stable objects in \cite{CT22}.
We use compactifiable cone structures as in
Definition \ref{def-compactifiable-cone-structures} and the compactification
$\mathfrak X\hookrightarrow\mathfrak X^c$ above.
Let $p:\spec\bb C\longrightarrow\mathsf{Exp}(X|D)$ 
be a geometric point, and let $X_p$ be the corresponding expansion of
$(X,D)$. A \emph{logarithmic PT-stable pair} on $X_p$ is a pair
\[
(F,s),
\qquad
s:\mcal O_{X_p}\longrightarrow F,
\]
where $F$ is a properly supported pure one-dimensional coherent sheaf
on $X_p$, satisfying the following conditions:
\begin{enumerate}
\item The cokernel $Q:=\coker(s)$
is zero-dimensional, and its support is disjoint from all rank $1$
strata and all tube components of $X_p$.
\item The scheme-theoretic support of $F$ is algebraically transverse to all strata of $X_p$.
\item For every linear $2$-valent component $Y\subset X_p$, the restricted
pair $(F|_Y,s|_Y)$ is invariant, up to isomorphism of pairs, under the
fibrewise $\bb G_m$-action if and only if $Y$ is a tube component.
\end{enumerate}
One defines flat families similarly, requiring $\operatorname{Supp}(F)\to S$
to be proper for a family over $S$.
Therefore, we can define $\mathsf{P}(X|D)$ and Deligne-Mumford stack $\mathsf{P}_{\beta,\chi}(X|D)$
with canonical structure morphism
\[
\mathsf{P}_{\beta,\chi}(X|D)\subset\mathsf{P}(X|D)\to\mathsf{Exp}(X|D)
\]
as above; see also \cite[Rmk.~2.4.6]{MR25}.

Now we claim that we have an open immersion of derived stacks
\[
\mathsf{dP}_{\beta,\chi}(X|D)\subset\mathsf{dP}(X|D)\subset\mathsf{dPerf}(\mathfrak X^c/\mathsf{Exp}(X|D))_0
\]
such that $\left(\mathsf{dP}(X|D)\right)_{\mathrm{cl}}\cong\mathsf{P}(X|D)$
and $\left(\mathsf{dP}_{\beta,\chi}(X|D)\right)_{\mathrm{cl}}\cong\mathsf{P}_{\beta,\chi}(X|D)$
over $\mathsf{Exp}(X|D)$. As in Lemma \ref{lem-derived-hilb},
we only need to prove that 
\[
\mathsf{P}(X|D)\subset\mathsf{Perf}^{\geq0}(\mathfrak X^c/\mathsf{Exp}(X|D))_{0,\mathrm{rk=1}}
\]
is an open immersion. The morphism sends a family $(F,s)$ over $S$ to
$(\mcal O_{\mathfrak X^c_S}\to F^c)$, where $F^c$ is the extension of $F$
by zero. Since this assertion
is local for the smooth topology on $\mathsf{Exp}(X|D)$, we can
pullback to a smooth atlas of $\mathsf{Exp}(X|D)$ and work locally near a
fixed expansion.

On the complement of the rank $1$ strata and the tube components in
$\mathfrak X$, the ambient space is smooth, where we can use the usual
local characterization and openness conditions for fixed-determinant
stable-pair complexes; see also \cite[Thm.~0.1]{CT22}.
On a neighborhood of the rank $1$ strata and the tube components,
logarithmic stable pairs are surjective and correspond to ideal sheaves. On this
neighborhood, the required openness is therefore exactly the openness
of the logarithmic Hilbert stack proved in
Lemma~\ref{lem-derived-hilb}. We apply these local descriptions on the
proper family $\mathfrak X^c\to\Exp(X|D)$, with the trivial
condition near $\mathfrak X^c\setminus\mathfrak X$.
The openness argument of Lemma \ref{lem-derived-hilb}, together with
DT-stability, therefore proves the claim.
After a compactifiable common refinement, the independence argument in
the proof of Lemma \ref{lem-derived-hilb} applies as well.
\end{nota}

\begin{nota}
As in Li-Wu's construction, we also have 
the weighted stack of expansions $\Exp_{\beta,\chi}(X|D)$ equipped with 
an étale morphism 
\[
q_{\beta,\chi}:\Exp_{\beta,\chi}(X|D)\to\Exp(X|D)
\]
such that $\Hilb_{\beta,\chi}(X|D),\mathsf{P}_{\beta,\chi}(X|D)\to\Exp(X|D)$ factor through 
this morphism.
The same holds for the zero-dimensional case and the family version, which will be described
later.
Let
\[
\begin{aligned}
\mathfrak X_{\beta,\chi}
&:=
\mathfrak X\times_{\Exp(X|D)}\Exp_{\beta,\chi}(X|D),\\
\mathfrak X^c_{\beta,\chi}
&:=
\mathfrak X^c\times_{\Exp(X|D)}\Exp_{\beta,\chi}(X|D).
\end{aligned}
\]
We use the induced cone and integral structures on the weighted expansion
stack. The open immersion $\mathfrak X_{\beta,\chi}\hookrightarrow
\mathfrak X^c_{\beta,\chi}$ is the base change of the chosen compactification,
so these structures are again compactifiable.
Using compatibility of the compactified relative mapping stacks with
base change and Lemma~\ref{lem-derived-hilb}, the same derived stack is obtained as the derived
open substack
\[
\dHilb_{\beta,\chi}(X|D)
\subset
\mathsf{dPerf}
\bigl(
\mathfrak X^c_{\beta,\chi}/
\Exp_{\beta,\chi}(X|D)
\bigr)_0
\]
whose classical truncation is $\Hilb_{\beta,\chi}(X|D)$ and whose
weights agree with the topological data of the universal subscheme;
see \cite[Cor.~2.2.2.9]{TV-HAGII} for more details.
The same statement holds for stable pairs.

Since $q_{\beta,\chi}$ is étale, the transitivity triangle gives
\[
\bb L_{\dHilb_{\beta,\chi}(X|D)/
\Exp_{\beta,\chi}(X|D)}
\simeq
\bb L_{\dHilb_{\beta,\chi}(X|D)/
\Exp(X|D)}.
\]
Thus the relative shifted symplectic and Lagrangian structures
constructed over $\Exp(X|D)$ restrict canonically to the fixed
topological locus and may equivalently be regarded as structures
relative to $\Exp_{\beta,\chi}(X|D)$. The same convention applies to
the boundary Hilbert stacks of points, family Hilbert stacks and zero-dimensional 
Hilbert stacks.

Henceforth, we do not distinguish between these two canonically
equivalent derived presentations and suppress the subscript
$(\beta,\chi)$ from the base-changed universal expansions and their compactifications, while
keeping the base of each relative construction explicit.
\end{nota}

\subsection{Relative logarithmic \texorpdfstring{$\mathsf{DT}_4$}{DT4}-theory}\label{sec-rel-log-DT4}
Next, we will consider relative logarithmic curve counting. 
Fix a compactifiable cone structure with universal family 
$\mathfrak X\subset\mathfrak X^c\to\Exp_{\beta,\chi}(X|D)$.
Let $\mathfrak D$ be the universal boundary
expansion induced by the universal expansion of $(X|D)$
and $\mathfrak D^c\subset\mathfrak X^c$ be the full horizontal boundary with branches 
$\mathfrak D_i^c$. Note that $\mathfrak D^c$ is projective and flat over 
$\Exp_{\beta,\chi}(X|D)$. 
We have 
the following derived restriction map which is compatible with determinant maps, i.e. we have 
\[\begin{tikzcd}
	{\mathsf{dPerf}(\mathfrak X^c/\mathsf{Exp}_{\beta,\chi}(X| D))} & {\mathsf{dPerf}(\mathfrak D^c/\mathsf{Exp}_{\beta,\chi}(X| D))} \\
	{\mathsf{dPic}(\mathfrak X^c/\mathsf{Exp}_{\beta,\chi}(X| D))} & {\mathsf{dPic}(\mathfrak D^c/\mathsf{Exp}_{\beta,\chi}(X| D))}
	\arrow["{r_{X\to D}}", from=1-1, to=1-2]
	\arrow["\det", from=1-1, to=2-1]
	\arrow["\det", from=1-2, to=2-2]
	\arrow["{r_{X\to D}}", from=2-1, to=2-2]
\end{tikzcd}\]
By Lemma \ref{lem-derived-hilb},
this diagram induces
\[\begin{tikzcd}
	{\dHilb_{\beta,\chi}(X| D)} & {\dHilb^{\mathbf n}(D|\partial D)} \\
	{\mathsf{dPerf}(\mathfrak X^c/\mathsf{Exp}_{\beta,\chi}(X| D))_0} & {\mathsf{dPerf}(\mathfrak D^c/\mathsf{Exp}_{\beta,\chi}(X| D))_0} \\
	{\mathsf{Exp}_{\beta,\chi}(X| D)}
	\arrow["{\mathsf{r}_{X\to D}}", from=1-1, to=1-2]
	\arrow[hook, from=1-1, to=2-1]
	\arrow[hook, from=1-2, to=2-2]
	\arrow["{r_{X\to D}}", from=2-1, to=2-2]
	\arrow[from=2-1, to=3-1]
	\arrow[from=2-2, to=3-1]
\end{tikzcd}\]
where the classical truncation $\Hilb^{\mathbf n}(D|\partial D)$ parameterizes the 
prescribed DT-stable families of points supported in $\mathfrak D$,
with $\mathbf n=(n_1,\ldots,n_r)$ and of total length $n_i=\beta\cdot D_i$ 
on the transform of $D_i$. 
Note that our boundary target is defined over $\Exp_{\beta,\chi}(X|D)$, 
unlike the target in \cite[\S~3.6-3.8]{MR25}\footnote{For compatible cone and 
integral structures, MR construct $\Exp_{\beta,\chi}(X|D)\to\Exp_0(D|\partial D)$ 
with $\mathfrak D\cong\mathfrak D_{\mathrm{MR}}\times_{\Exp_0(D|\partial D)}\Exp_{\beta,\chi}(X|D)$. 
Base change therefore give $\Hilb^{\mathbf n}(D|\partial D)\cong\Hilb^{\mathbf n}(D|\partial D)_{\mathrm{MR}}\times_{\Exp_0(D|\partial D)}\Exp_{\beta,\chi}(X|D)$.
The Maulik-Ranganathan target is Deligne-Mumford, whereas its displayed base change need not be.}.

\begin{thm}\label{thm-curves-shifted-lag}
Let $(X,D)$ be a log Calabi-Yau snc pair with $X$ an $n$-dimensional smooth quasi-projective variety.
Then we have the following properties.
\begin{enumerate}
	\item The canonical morphism \[\dHilb^{\mathbf n}(D|\partial D)\to\mathsf{Exp}_{\beta,\chi}(X| D)\]
admits a $(3-n)$-shifted symplectic structure.
\item The restriction morphism
\[\begin{tikzcd}
	{\dHilb_{\beta,\chi}(X| D)} & {\dHilb^{\mathbf n}(D|\partial D)} \\
	{\mathsf{Exp}_{\beta,\chi}(X| D)}
	\arrow["{\mathsf{r}_{X\to D}}", from=1-1, to=1-2]
	\arrow["{F_{\beta,\chi}}", from=1-1, to=2-1]
	\arrow["{\mathrm{ev}_{\mathbf n,D}}", from=1-2, to=2-1]
\end{tikzcd}\]
is a $(3-n)$-shifted Lagrangian relative to $\mathsf{Exp}_{\beta,\chi}(X| D)$.
\item Assume $n=4$ and that there exists an snc pair $(Y,S)$ with $\dim Y=3$ such that $X=\mathrm{Tot}(\omega_Y(S))$ and $D=\mathrm{Tot}(\omega_S)$.
    Then 
    \[\mathsf{r}_{X\to D}:\dHilb_{\beta,\chi}(X| D)\to\dHilb^{\mathbf n}(D|\partial D)\]
	is an exact $(-1)$-shifted Lagrangian relative to $\mathsf{Exp}_{\beta,\chi}(X| D)$
    with a canonical orientation relative to $\mathsf{Exp}_{\beta,\chi}(X| D)$.
\end{enumerate}
\end{thm}
\begin{proof}
Let $\bb I^c$ be the universal perfect complex on
$\mathfrak X^c\times_{\Exp(X|D)}\dHilb(X|D)$,
and let $\bb I$ be its restriction to
$\mathfrak X\times_{\Exp(X|D)}\dHilb(X|D)$.
Write $\pi^c$ and $\pi$ for the respective projections to $\dHilb(X|D)$.
Away from the universal substack, $\bb I^c$
is canonically trivialized by its determinant, therefore
\[
\mathbf R\pi^c_*\mathbf R\scr Hom(\bb I^c,\bb I^c)_0
\cong
\mathbf R\pi_*\mathbf R\scr Hom(\bb I,\bb I)_0.
\]
The same comparison holds for the boundary.
By excision, the supported Toën-Vezzosi Chern characters and trace constructions
are computed on the original families.
We therefore apply Lemma \ref{lem:log-crepant-MR-pair}
to $\mathfrak X$ and $\mathfrak D$ throughout this proof.

For (1), the second statement of Lemma \ref{lem:log-crepant-MR-pair},
the proof of \cite[Thm.~3.3]{CZZ24}
and the fact that we only consider boundary divisors and all codimension $\geq2$ strata 
have been removed (hence $\partial(\partial\mathfrak D)=\emptyset$), we know that 
\[\dHilb^{\mathbf n}(D|\partial D)\to\mathsf{Exp}_{\beta,\chi}(X| D)\]
is a $(4-n)$-shifted Lagrangian relative to $\mathsf{Exp}_{\beta,\chi}(X| D)$
which induces the $(3-n)$-shifted symplectic structure
of \[\dHilb^{\mathbf n}(D|\partial D)\to\mathsf{Exp}_{\beta,\chi}(X| D).\]
See \cite[Lem.~3.7]{CZZ24} or \cite[Lem.~1.9]{PY26} for details.
Note that the shifted symplectic form agrees with the form obtained directly from the AKSZ construction 
(since by the second statement of Lemma \ref{lem:log-crepant-MR-pair} again,
$\mathfrak D/\mathsf{Exp}_{\beta,\chi}(X| D)$ is Calabi-Yau), which is given by 
\[
\int_{\mathfrak D/\Exp(X|D)}\mathrm{vol}_{\mathfrak D/\Exp(X|D)}
\wedge\widetilde{\mathrm{ev}_D^*\Omega_{\mathsf{dPerf}}},
\]
where $\Omega_{\mathsf{dPerf}}$ is the canonical $2$-shifted form on $\mathsf{dPerf}$,
with the evaluation map
$\mathrm{ev}_D:\mathfrak D\times_{\Exp(X|D)}\dHilb(D|\partial D)\to\mathsf{dPerf}\times\Exp(X|D)$,
and $\widetilde{\mathrm{ev}_D^*\Omega_{\mathsf{dPerf}}}$ is a lift of ${\mathrm{ev}_D^*\Omega_{\mathsf{dPerf}}}$ 
to a compactly supported form on the universal substack $\mathfrak Z_D$.
Moreover, the trivialization $\mathrm{vol}_{\mathfrak D/\Exp(X|D)}:\scr O\cong\omega_{\mathfrak D/\Exp(X|D)}$ 
and the trace map described in \cite[Eqn.~(3.4)]{CZZ24} induce the integration
$\int_{\mathfrak D/\Exp(X|D)}\mathrm{vol}_{\mathfrak D/\Exp(X|D)}
\wedge(-)$.

For (2), it follows from Lemma \ref{lem:log-crepant-MR-pair}
and the same arguments in \cite[Thm.~3.3]{CZZ24}. Indeed,
it suffices to construct a null-homotopy of
\[
\mathsf{r}_{X\to D}^*\int_{\mathfrak D/\Exp(X|D)}\mathrm{vol}_{\mathfrak D/\Exp(X|D)}
\wedge\widetilde{\mathrm{ev}_D^*\Omega_{\mathsf{dPerf}}}.
\]
By the arguments in \cite[Thm.~3.3]{CZZ24}, we have:
\[\begin{tikzcd}
	{\dHilb(X|D)} & {\mathfrak D\times_{\Exp(X|D)}\dHilb(X|D)} & {\mathfrak X\times_{\Exp(X|D)}\dHilb(X|D)} \\
	{\dHilb(D|\partial D)} & {\mathfrak D\times_{\Exp(X|D)}\dHilb(D|\partial D)} & {\mathsf{dPerf}\times\Exp(X|D).}
	\arrow["{\mathsf{r}_{X\to D}}", from=1-1, to=2-1]
	\arrow["\square"{description}, draw=none, from=1-1, to=2-2]
	\arrow[from=1-2, to=1-1]
	\arrow["{i_{D\subset X}}", from=1-2, to=1-3]
	\arrow["{\tilde{r}_{X\to D}}", from=1-2, to=2-2]
	\arrow["{\circlearrowright }"{description}, draw=none, from=1-2, to=2-3]
	\arrow["{\mathrm{ev}_{X,D}}", from=1-3, to=2-3]
	\arrow[from=2-2, to=2-1]
	\arrow["{\mathrm{ev}_D}", from=2-2, to=2-3]
\end{tikzcd}\]
Then we have 
\begin{align*}
&\mathsf{r}_{X\to D}^*\int_{\mathfrak D/\Exp(X|D)}\mathrm{vol}_{\mathfrak D/\Exp(X|D)}
\wedge\widetilde{\mathrm{ev}_D^*\Omega_{\mathsf{dPerf}}}\\ 
&=
\int_{\mathfrak D/\Exp(X|D)}^{\dHilb(X|D)}(1\times\tilde{r}_{X\to D})^*\mathrm{vol}_{\mathfrak D/\Exp(X|D)}
\wedge\widetilde{\mathrm{ev}_D^*\Omega_{\mathsf{dPerf}}}\\
&=\int_{\mathfrak D/\Exp(X|D)}^{\dHilb(X|D)}\mathrm{vol}_{\mathfrak D/\Exp(X|D)}
\wedge i_{D\subset X}^*\widetilde{\mathrm{ev}_{X,D}^*\Omega_{\mathsf{dPerf}}}.
\end{align*}
Here $\int_{\mathfrak D/\Exp(X|D)}$ is the integration from 
$\mathfrak D\times_{\Exp(X|D)}\dHilb(D|\partial D)$ to $\dHilb(D|\partial D)$
and $\int_{\mathfrak D/\Exp(X|D)}^{\dHilb(X|D)}$ is the integration from 
$\mathfrak D\times_{\Exp(X|D)}\dHilb(X|D)$ to $\dHilb(X|D)$.
This follows from the compatibility results in \cite[Thm.~A.0.10(i),~Cor.~A.0.12]{Pre}.
Then by arguments similar to those in \cite[Eqn.~(3.12)(3.13)(3.14)]{CZZ24} and by the family version 
of \cite[\S~3.2.1,~\S~2.2.3]{Cal15}, we construct a null-homotopy of the map
\begin{align*}
&\int_{\mathfrak D/\Exp(X|D)}^{\dHilb(X|D)}\mathrm{vol}_{\mathfrak D/\Exp(X|D)}
\wedge i_{D\subset X}^*(-):\\
&\mathrm{NC}_{\mathfrak Z_{X,D}}((\mathfrak X\times_{\Exp(X|D)}\dHilb(X|D))/\Exp(X|D))(2)
\to\mathrm{NC}(\dHilb(X|D)/\Exp(X|D))[-n+1](2),
\end{align*}
for universal substack $\mathfrak Z_{X,D}\subset\mathfrak X\times_{\Exp(X|D)}\dHilb(X|D)$.
The resulting isotropic structure is non-degenerate by supported
relative Serre duality and Lemma \ref{lem:log-crepant-MR-pair}.

For (3), inspired from \cite[proof of Prop.~2.34]{Saf23}, we
consider the fibre $\bb C^*$-action on
$X=\mathrm{Tot}(\omega_Y(S))$. It lifts to the original universal pair
$(\mathfrak X,\mathfrak D)$ with trivial action on
$\mathsf{Exp}_{\beta,\chi}(X|D)$. By
Remark \ref{rmk-intrinsic-derived-hilb}, it induces actions on the two
derived Hilbert stacks for which $\mathsf r_{X\to D}$ is $\bb C^*$-equivariant.
The canonical relative log Calabi-Yau volume form and its boundary residue
have $\bb C^*$-weight one. 
Consequently, the resulting symplectic form $\Omega$ and this specified
isotropic homotopy $h$ in (1) and (2) together have weight one.

Write $\mathrm{DR}$ for the global Hodge-completed relative de Rham complex
with differential
$d_{\mathrm{tot}}=d_{\mathrm{int}}+d_{\mathrm{dR}}$.
Equip the following homotopy fibre with its induced Hodge filtration:
\[
\mathcal C:=\operatorname{fib}\bigl(
	\mathrm{DR}(\dHilb^{\mathbf n}(D|\partial D)/\mathsf{Exp}_{\beta,\chi}(X|D))
	\xrightarrow{\mathsf r_{X\to D}^*}
	\mathrm{DR}(\dHilb_{\beta,\chi}(X|D)/\mathsf{Exp}_{\beta,\chi}(X|D))\bigr).
\]
In the standard fibre model, a degree-$k$ element is a pair $(\alpha,\beta)$,
where $\alpha$ has degree $k$ in the boundary complex and $\beta$ has
degree $k-1$ in the source complex, with total differential
\[
d_{\mathcal C}(\alpha,\beta)
=(d_{\mathrm{tot}}\alpha,\mathsf r_{X\to D}^*\alpha-d_{\mathrm{tot}}\beta).
\]
Choose equivariant representatives so that
$\xi=(\Omega,h)\in Z^1(\mathrm{Fil}^2\mathcal C)$ has $\bb C^*$-weight one.
Thus $d_{\mathrm{tot}}\Omega=0$ and
$d_{\mathrm{tot}}h=\mathsf r_{X\to D}^*\Omega$.
Let $E_X,E_D$ be the induced relative Euler vector fields and
equivariance gives $\mathsf r_{X\to D}^*\iota_{E_D}
=\iota_{E_X}\mathsf r_{X\to D}^*$. Consider 
$\iota_E(\alpha,\beta)=(\iota_{E_D}\alpha,-\iota_{E_X}\beta)$, Cartan's formula implies
\[
d_{\mathcal C}\iota_E+\iota_Ed_{\mathcal C}=\mathcal L_E.
\]
Since $\xi$ is of weight one, one has 
$\mathcal L_E\xi=\xi$. Therefore, if we define
$\eta:=\iota_E\xi\in \mathrm{Fil}^1\mathcal C^0$, then using 
$d_{\mathcal C}\xi=0$, we have
\[
d_{\mathcal C}\eta=\xi.
\]
Thus $\bar\eta\in Z^0(\mathrm{Gr}^1\mathcal C)$ has image
represented by $\xi$ under the connecting morphism for
\[
\mathrm{Fil}^2\mathcal C\longrightarrow \mathrm{Fil}^1\mathcal C
\longrightarrow\mathrm{Gr}^1\mathcal C.
\]
Write $\eta=\eta_1+\eta_{\geq2}$ according to Hodge degree.
The inner part $d_{\mathrm{int},\mathcal C}$ annihilates $\eta_1$.
The de-Rham part of differential is
$d_{\mathrm{dR},\mathcal C}(\alpha,\beta)
=(d_{\mathrm{dR}}\alpha,-d_{\mathrm{dR}}\beta)$, we therefore have
\[
\xi=d_{\mcal C}\eta=d_{\mathrm{int},\mcal C}\eta_{\geq2}+d_{\mathrm{dR},\mathcal C}\eta
\quad\Rightarrow\quad
\xi-d_{\mathrm{dR},\mathcal C}\eta_1=d_{\mathcal C}\eta_{\geq2}.
\]
Here $\bar\eta$ gives a $(-1)$-shifted $1$-form on the boundary Hilbert
stack together with a null-homotopy of its pullback, and $\eta_{\geq2}$
specifies a homotopy identifying its de Rham differential with the
original pair $(\Omega,h)$. This gives an exact lift of the entire
isotropic structure; see also \cite[Def.~3.10(4)]{BCS24}.
Together with the non-degeneracy proved in (2), this proves that
$\mathsf r_{X\to D}$ is an exact $(-1)$-shifted Lagrangian relative to
$\mathsf{Exp}_{\beta,\chi}(X|D)$.

Finally, we will consider the existence of orientations.
Indeed, this follows from the same arguments in \cite[Thm.~4.6]{CZZ24}
and the choice in \cite[Rmk.~4.10]{CZZ24}.
Here we give a sketch. Consider the forgetful morphisms $(\scr O\twoheadrightarrow\scr F)\mapsto\scr F$:
\[\begin{tikzcd}
	{\Hilb(X| D)} & {M_{\mathrm{cpt}}(\mathfrak X/\mathsf{Exp}(X| D))} \\
	{\Hilb(D|\partial D)} & {M_{\mathrm{cpt}}(\mathfrak D/\mathsf{Exp}(X| D))}
	\arrow["{F_X}", from=1-1, to=1-2]
	\arrow["{\mathsf{r}_{X\to D}}"', from=1-1, to=2-1]
	\arrow["r", from=1-2, to=2-2]
	\arrow["{F_D}", from=2-1, to=2-2]
\end{tikzcd}\]
where $M_{\mathrm{cpt}}$ denotes
the locus of flat families of coherent sheaves that are perfect on
the corresponding base-changed universal family, have support proper
over their parameter scheme, and are fibrewise normal to the relevant
codimension-one strata. 
Let $\bb F,\bb F_D,\bb F_S$ denote
the universal sheaves lying on 
\[M_{\mathrm{cpt}}(\bullet/\mathsf{Exp}(X| D))
\times_{\mathsf{Exp}(X| D)}\bullet\] for $\bullet=\mathfrak X,\mathfrak D,\mathfrak S$,
respectively. Consider the following diagram
\begin{footnotesize}
\[\begin{tikzcd}
	{M_{\mathrm{cpt}}(\mathfrak X/\mathsf{Exp}(X| D)) \times_{\mathsf{Exp}(X| D)}\mathfrak X} & {M_{\mathrm{cpt}}(\mathfrak X/\mathsf{Exp}(X| D)) \times_{\mathsf{Exp}(X| D)}\mathfrak D} & {M_{\mathrm{cpt}}(\mathfrak D/\mathsf{Exp}(X| D)) \times_{\mathsf{Exp}(X| D)}\mathfrak D} & {} \\
	& {M_{\mathrm{cpt}}(\mathfrak X/\mathsf{Exp}(X| D))} & {M_{\mathrm{cpt}}(\mathfrak D/\mathsf{Exp}(X| D))} & {}
	\arrow["{\pi_{M_X}}"', from=1-1, to=2-2]
	\arrow[from=1-2, to=1-1]
	\arrow[from=1-2, to=1-3]
	\arrow[from=1-2, to=2-2]
	\arrow["\square"{description}, draw=none, from=1-2, to=2-3]
	\arrow[from=1-3, to=1-4]
	\arrow["{\pi_{M_D}}"', from=1-3, to=2-3]
	\arrow["r", from=2-2, to=2-3]
	\arrow["{=}", from=2-3, to=2-4]
\end{tikzcd}\]
\end{footnotesize}
\begin{footnotesize}
\[\begin{tikzcd}
	{} & {M_{\mathrm{cpt}}(\mathfrak D/\mathsf{Exp}(X| D)) \times_{\mathsf{Exp}(X| D)}\mathfrak S} & {M_{\mathrm{cpt}}(\mathfrak S/\mathsf{Exp}(X| D)) \times_{\mathsf{Exp}(X| D)}\mathfrak S} \\
	{} & {M_{\mathrm{cpt}}(\mathfrak D/\mathsf{Exp}(X| D))} & {M_{\mathrm{cpt}}(\mathfrak S/\mathsf{Exp}(X| D)).}
	\arrow[from=1-1, to=1-2]
	\arrow[from=1-2, to=1-3]
	\arrow[from=1-2, to=2-2]
	\arrow["\square"{description}, draw=none, from=1-2, to=2-3]
	\arrow["{\pi_{M_S}}", from=1-3, to=2-3]
	\arrow["{=}", from=2-1, to=2-2]
	\arrow["{p_{\dagger}}", from=2-2, to=2-3]
\end{tikzcd}\]
\end{footnotesize}
Then we can prove that 
\begin{equation}\label{eqn-001}
	\begin{aligned}
&\det\bb T_{\dHilb(X| D)/\mathsf{Exp}(X| D)}|_{\Hilb(X| D)}\\
&\cong F_X^*\left(\det\pi_{M_X,*}\mathbf R\scr Hom(\bb F,\bb F)[1]\right)
\otimes(\mathsf{r}_{X\to D}^*F_D^*\det(\pi_{M_D,*}\bb F_D))
\end{aligned}
\end{equation}
and 
\begin{align*}
&\det\bb T_{\dHilb(D|\partial D)/\mathsf{Exp}(X| D)}|_{\Hilb(D|\partial D)}
\cong F_D^*\frac{\det(\pi_{M_D,*}\bb F_D)^{\otimes2}}{\det(p_{\dagger}^*\pi_{M_S,*}
\mathbf R\scr Hom(\bb F_S,\bb F_S))^{\otimes2}}.
\end{align*}
Then we can define 
\[
K^{-1/2}_{\Hilb(D|\partial D)}:=F_D^*\frac{\det(\pi_{M_D,*}\bb F_D)}{\det(p_{\dagger}^*\pi_{M_S,*}
\mathbf R\scr Hom(\bb F_S,\bb F_S))}
\]
which forces 
\[
\det\bb T_{\dHilb(D|\partial D)/\mathsf{Exp}(X| D)}|_{\Hilb(D|\partial D)}\cong
\left(K^{-1/2}_{\Hilb(D|\partial D)}\right)^{\otimes2}.
\]
Moreover, to compare $\mathsf{r}_{X\to D}^*K^{-1/2}_{\Hilb(D|\partial D)}$
and Eqn. (\ref{eqn-001}), we need to compare 
\[\det\pi_{M_X,*}\mathbf R\scr Hom(\bb F,\bb F)[1]\text{ and }
r^*p_{\dagger}^*\det\pi_{M_S,*}
\mathbf R\scr Hom(\bb F_S,\bb F_S)[1].\]
Consider the following projection maps 
\[\begin{tikzcd}
	{\mathsf{d}M_{\mathrm{cpt}}(\mathfrak X/\mathsf{Exp}(X| D))} & {\mathsf{d}M_{\mathrm{cpt}}(\mathfrak Y/\mathsf{Exp}(X| D))} \\
	{\mathsf{d}M_{\mathrm{cpt}}(\mathfrak D/\mathsf{Exp}(X| D))} & {\mathsf{d}M_{\mathrm{cpt}}(\mathfrak S/\mathsf{Exp}(X| D))}
	\arrow["{\pi_{\dagger}}", from=1-1, to=1-2]
	\arrow["r", from=1-1, to=2-1]
	\arrow[from=1-2, to=2-2]
	\arrow["{p_{\dagger}}", from=2-1, to=2-2]
\end{tikzcd}\]
where $r$ admits a $(-1)$-shifted Lagrangian structure relative to $\mathsf{Exp}(X| D)$
with Lagrangian fibration $p_{\dagger}$ relative to $\mathsf{Exp}(X| D)$.
This forces $p_{\dagger}\circ r$ to have a $(-2)$-shifted symplectic structure.
The spectral construction gives us a distinguished triangle
\begin{small}
	\begin{align*}
		&\pi_{\dagger}^*\bb L_{\mathsf{d}M_{\mathrm{cpt}}(\mathfrak Y/\mathsf{Exp}(X| D))/\mathsf{d}M_{\mathrm{cpt}}(\mathfrak S/\mathsf{Exp}(X| D))}[-2]
\to\bb T_{\mathsf{d}M_{\mathrm{cpt}}(\mathfrak X/\mathsf{Exp}(X| D))/\mathsf{d}M_{\mathrm{cpt}}(\mathfrak S/\mathsf{Exp}(X| D))}\\
&\to\pi_{\dagger}^*\bb T_{\mathsf{d}M_{\mathrm{cpt}}(\mathfrak Y/\mathsf{Exp}(X| D))/\mathsf{d}M_{\mathrm{cpt}}(\mathfrak S/\mathsf{Exp}(X| D))}
\to
	\end{align*}
\end{small}
which is self-dual via the isomorphism induced by the shifted symplectic form.
Let $\bb E:=\bb L_{\mathsf{d}M_{\mathrm{cpt}}(\mathfrak X/\mathsf{Exp}(X| D))/\mathsf{d}M_{\mathrm{cpt}}(\mathfrak S/\mathsf{Exp}(X| D))}|_{M_{\mathrm{cpt}}(\mathfrak X/\mathsf{Exp}(X| D))}$.
This gives $\bb E=(\bb V\to\bb V^{\vee})$, with orientation
$\pm(-\sqrt{-1})^{\rank\bb V}:\scr O\to\scr O$.
We always choose the \textit{positive sign}.
The positive choice determines a canonical choice of isomorphism 
\[
\det(\pi_{M_X,*}\mathbf R\scr Hom(\bb F,\bb F)[1])\cong
r^*p_{\dagger}^*\det(\pi_{M_S,*}
\mathbf R\scr Hom(\bb F_S,\bb F_S)[1])
\]
which concludes the argument.
\end{proof}

\begin{nota}
Here we note that, unlike in the Li-Wu and Cao-Zhao-Zhou's settings,
the universal divisor in our setting is not constant over
$\mathsf{Exp}_{\beta,\chi}(X|D)$. However, the proofs of
\cite[Thm.~3.2, Thm.~3.3, Thm.~4.6]{CZZ24} only use the relative
AKSZ construction, relative Serre duality, and the determinant-line
identities; hence they apply verbatim to our universal families.
\end{nota}

Now we will use this and shifted Lagrangian classes to deduce the 
logarithmic relative Donaldson-Thomas theory. Let $n=4$ and assume that 
all the pairs are total spaces as described in Theorem \ref{thm-curves-shifted-lag}(3).

By Theorem \ref{thm-curves-shifted-lag}(3), the map $\mathrm{ev}_{\mathbf n,D}:\dHilb^{\mathbf n}(D|\partial D)\to \mathsf{Exp}_{\beta,\chi}(X| D)$
is an exact $(-1)$-shifted symplectic fibration and $\mathsf r_{X\to D}$ is an exact Lagrangian.
Hence Theorem \ref{thm-perverse-pullback} gives the t-exact
perverse pullback $\mathrm{ev}_{\mathbf n,D}^{\varphi}$.

Then Remark \ref{rmk-KKPS-lagrangian-classes}
yields the canonical morphism
	\[
	\mathsf r_{X\to D}^*\mathrm{ev}_{\mathbf n,D}^{\varphi}[\vdim F_{\beta,\chi}]
	\to F_{\beta,\chi}^!.
	\]
	Since $\mathsf{Exp}_{\beta,\chi}(X| D)$ is smooth of dimension $0$, we apply $\mathbb Q_{\mathsf{Exp}_{\beta,\chi}(X| D)}[0]$
	to this morphism of functors. Then we have 
	\[
	\mathsf r_{X\to D}^*\varphi_{\mathrm{ev}_{\mathbf n,D}}[\vdim \dHilb_{\beta,\chi}(X| D)]
	\to \bb D\mathbb Q_{\dHilb_{\beta,\chi}(X| D)}
	\]
for $\varphi_{\mathrm{ev}_{\mathbf n,D}}:=\mathrm{ev}_{\mathbf n,D}^{\varphi}\mathbb Q_{\mathsf{Exp}_{\beta,\chi}(X| D)}[0]$.
Define the \textit{cohomological Donaldson-Thomas theory}
by 
\[
H^*_{\mathrm{crit}}(\dHilb^{\mathbf n}(D|\partial D)):=H^*(\varphi_{\mathrm{ev}_{\mathbf n,D}}).
\]

\begin{defi}
In this case, the \textit{logarithmic relative Donaldson-Thomas theory}
is the canonical morphism 
\[
	[\mathsf r_{X\to D}]^{\mathrm{Lag}}:
	H^*_{\mathrm{crit}}(\dHilb^{\mathbf n}(D|\partial D))\to H^{\BM}_{\vdim \dHilb_{\beta,\chi}(X| D)-*}(\dHilb_{\beta,\chi}(X| D),\bb Q)
	\]
induced by the construction in Remark~\ref{rmk-constructions-BM}(1).
\end{defi}
\begin{nota}
Note that $\vdim\dHilb_{\beta,\chi}(X| D)=2\chi-D\cdot\beta$ by the Riemann-Roch theorem.
The same holds for $\vdim\dP_{\beta,\chi}(X| D)$.
\end{nota}

Finally we will consider how the choice of cone structure can affect
the logarithmic relative Donaldson-Thomas theory.
We always consider the total-space assumptions and orientations now.
By \cite[Prop.~3.6.1]{MR24} and the common-refinement construction in Lemma~\ref{lem-derived-hilb}, 
it suffices to compare two such models related by a compactifiable refinement $\mu\to\lambda$.
Set
\[
(\overline{\mathfrak X}^{\mu}_{\lambda},\overline{\mathfrak D}^{\mu}_{\lambda})
:=(\mathfrak X^\lambda,\mathfrak D^\lambda)\times_{\mathsf{Exp}_{\beta,\chi}(X|D)_\lambda}\mathsf{Exp}_{\beta,\chi}(X|D)_\mu.
\]
By the proof of \cite[Thm.~5.2.1]{MR24}, the canonical morphism $\mathfrak X^\mu\to\overline{\mathfrak X}^{\mu}_{\lambda}$ contracts the additional tube components. These insertions do not change the induced boundary $0$-complex, and restriction of the contraction gives
\[
\mathfrak D^\mu\cong\overline{\mathfrak D}^{\mu}_{\lambda}.
\]
For the base-changed family, choose the flat projective compactification
\[
\bigl((\overline{\mathfrak X}^{\mu}_{\lambda})^c,(\overline{\mathfrak D}^{\mu}_{\lambda})^c\bigr)
:=\bigl((\mathfrak X^\lambda)^c,(\mathfrak D^\lambda)^c\bigr)\times_{\mathsf{Exp}_{\beta,\chi}(X|D)_\lambda}\mathsf{Exp}_{\beta,\chi}(X|D)_\mu.
\]
Define
\[
\overline{\dHilb}_{\beta,\chi}(X|D)_{\mu;\lambda}
:=\dHilb_{\beta,\chi}(X|D)_\lambda\times_{\mathsf{Exp}_{\beta,\chi}(X|D)_\lambda}\mathsf{Exp}_{\beta,\chi}(X|D)_\mu.
\]
By base change for derived mapping stacks, this is the open substack of 
$\mathsf{dPerf}((\overline{\mathfrak X}^{\mu}_{\lambda})^c/\mathsf{Exp}_{\beta,\chi}(X|D)_\mu)_0$.
The identification of the original boundary families and Remark~\ref{rmk-intrinsic-derived-hilb} similarly give
\[
\dHilb^{\mathbf n}(D|\partial D)_\mu
\cong\dHilb^{\mathbf n}(D|\partial D)_\lambda\times_{\mathsf{Exp}_{\beta,\chi}(X|D)_\lambda}\mathsf{Exp}_{\beta,\chi}(X|D)_\mu.
\]
Consequently, restriction gives the following diagram, with the right vertical open immersion understood via this identification:
\[\begin{tikzcd}
{\overline{\dHilb}_{\beta,\chi}(X|D)_{\mu;\lambda}} & {\dHilb^{\mathbf n}(D|\partial D)_\mu} \\
{\mathsf{dPerf}((\overline{\mathfrak X}^{\mu}_{\lambda})^c/\mathsf{Exp}_{\beta,\chi}(X|D)_\mu)_0} & {\mathsf{dPerf}((\overline{\mathfrak D}^{\mu}_{\lambda})^c/\mathsf{Exp}_{\beta,\chi}(X|D)_\mu)_0}
\arrow["{\overline{\mathsf r}_{X\to D}^{\mu,\lambda}}",from=1-1,to=1-2]
\arrow[hook,from=1-1,to=2-1] \arrow[hook,from=1-2,to=2-2]
\arrow["{\overline r_{X\to D}^{\mu,\lambda}}",from=2-1,to=2-2]
\end{tikzcd}\]
The upper horizontal morphism is the base change of $\mathsf r_{X\to D}^{\lambda}$, hence is an oriented exact $(-1)$-shifted 
Lagrangian relative to $\mathsf{Exp}_{\beta,\chi}(X|D)_\mu$. 
Thus we obtain two oriented exact $(-1)$-shifted Lagrangians over $\mathsf{Exp}_{\beta,\chi}(X|D)_\mu$:
\[\begin{tikzcd}
{\dHilb_{\beta,\chi}(X|D)_\mu} && {\overline{\dHilb}_{\beta,\chi}(X|D)_{\mu;\lambda}} \\
& {\dHilb^{\mathbf n}(D|\partial D)_\mu.}
\arrow["{\mathsf r_{X\to D}^{\mu}}"',from=1-1,to=2-2]
\arrow["{\overline{\mathsf r}_{X\to D}^{\mu,\lambda}}",from=1-3,to=2-2]
\end{tikzcd}\]
DT-stability requires the subschemes on the additional tube components to be tube subschemes. Pullback and contraction therefore give inverse operations on classical families, preserving the boundary restrictions; see the proof of \cite[Thm.~5.2.1]{MR24}. We obtain a canonical isomorphism
\[
p_{\mu,\lambda}:\Hilb_{\beta,\chi}(X|D)_\mu\xrightarrow{\ \cong\ }\overline{\Hilb}_{\beta,\chi}(X|D)_{\mu;\lambda},
\]
where the right-hand side is the classical truncation of $\overline{\dHilb}_{\beta,\chi}(X|D)_{\mu;\lambda}$. 
By construction, we have the following homotopical pullback diagram with proper vertical morphisms:
\[\begin{tikzcd}
{\overline{\dHilb}_{\beta,\chi}(X|D)_{\mu;\lambda}} & {\dHilb^{\mathbf n}(D|\partial D)_\mu} & {\mathsf{Exp}_{\beta,\chi}(X|D)_\mu} \\
{\dHilb_{\beta,\chi}(X|D)_\lambda} & {\dHilb^{\mathbf n}(D|\partial D)_\lambda} & {\mathsf{Exp}_{\beta,\chi}(X|D)_\lambda.}
\arrow["{\overline{\mathsf r}_{X\to D}^{\mu,\lambda}}",from=1-1,to=1-2] \arrow[from=1-2,to=1-3]
\arrow["{q_{X|D,\mu,\lambda}}",from=1-1,to=2-1] \arrow["{q_{D,\mu,\lambda}}",from=1-2,to=2-2] \arrow[from=1-3,to=2-3]
\arrow["\square"{description},draw=none,from=1-1,to=2-2] \arrow["\square"{description},draw=none,from=1-2,to=2-3]
\arrow["{\mathsf r_{X\to D}^{\lambda}}",from=2-1,to=2-2] \arrow[from=2-2,to=2-3]
\end{tikzcd}\]
\begin{conj}\label{conj-cone-structure-relative-log}
There are canonical pushforwards $q_{D,\mu,\lambda,*}$ on critical cohomology, compatible with composition of refinements, 
such that the following diagram commutes:
\[\begin{tikzcd}
& {H^*_{\mathrm{crit}}(\Hilb^{\mathbf n}(D|\partial D)_\mu)} & {H^*_{\mathrm{crit}}(\Hilb^{\mathbf n}(D|\partial D)_\lambda)} \\
{H^{\BM}_*(\Hilb_{\beta,\chi}(X|D)_\mu,\bb Q)} & {H^{\BM}_*(\overline{\Hilb}_{\beta,\chi}(X|D)_{\mu;\lambda},\bb Q)} & {H^{\BM}_*(\Hilb_{\beta,\chi}(X|D)_\lambda,\bb Q)}
\arrow["{q_{D,\mu,\lambda,*}}",from=1-2,to=1-3]
\arrow["{[\mathsf r_{X\to D}^{\mu}]^{\mathrm{Lag}}}"',from=1-2,to=2-1]
\arrow["{[\overline{\mathsf r}_{X\to D}^{\mu,\lambda}]^{\mathrm{Lag}}}",from=1-2,to=2-2]
\arrow["{[\mathsf r_{X\to D}^{\lambda}]^{\mathrm{Lag}}}",from=1-3,to=2-3]
\arrow["{p_{\mu,\lambda,*}}",from=2-1,to=2-2] \arrow["{q_{X|D,\mu,\lambda,*}}",from=2-2,to=2-3]
\end{tikzcd}\]
\end{conj}

\subsection{Tautological logarithmic DT/PT correspondence}\label{sec-DT-PT-corr}
Here we fix a compactifiable cone structure and only consider local pairs.
By Remark \ref{rmk-stable-pairs-results} and Remark \ref{rmk-stable-pairs},
we can consider both logarithmic DT and PT invariants.

Let $(X,D)$ be a log local Calabi-Yau snc pair of dimension $4$, and let
$\beta\neq 0$ be an effective curve class. Write
\[
D=\bigcup_{i=1}^rD_i,
\qquad
\mathbf n=(n_1,\ldots,n_r),
\qquad
n_i=\beta\cdot D_i.
\]
Suppose that a Calabi-Yau torus $\bb T$ (could be trivial) acts on $(X,D)$.
Assume that the fixed loci $\Hilb_{\beta,\chi}(X|D)^{\bb T}$, 
$\mathsf{P}_{\beta,\chi}(X|D)^{\bb T}$ and $\Hilb^n(X|D)^{\bb T}$
are proper for every $\chi$ and every $n>0$.
Let $\scr L$ be a $\bb T$-equivariant line bundle on $X$,
and let $\bb C_m^*$ be an auxiliary trivial torus with equivariant
parameter $m$.

Since we assume that our spaces are all total spaces of snc pairs of dimension
$3$, we have relative logarithmic theories:
\begin{align*}
	&[\mathsf r_{X\to D}]^{\mathrm{Lag}}:
H^*_{\mathrm{crit},\bb T}
\left(
\dHilb^{\mathbf n}(D|\partial D)_{\beta,\chi}
\right)
\longrightarrow
H^{\BM,\bb T}_{
\vdim\dHilb_{\beta,\chi}(X|D)-*}
\left(
\Hilb_{\beta,\chi}(X|D),
\bb Q
\right)_{\loc},\\ 
	&
[\mathsf r^{\PT}_{X\to D}]^{\mathrm{Lag}}:
H^*_{\mathrm{crit},\bb T}
\left(
\dHilb^{\mathbf n}(D|\partial D)_{\beta,\chi}
\right)
\longrightarrow
H^{\BM,\bb T}_{
\vdim\mathsf{dP}_{\beta,\chi}(X|D)-*}
\left(
\mathsf P_{\beta,\chi}(X|D),
\bb Q
\right)_{\loc}.
\end{align*}

Suppressing $\beta,\chi$ from the notation, let
\[
q_{X|D}:\mathfrak X\times_{\mathsf{Exp}_{\beta,\chi}(X|D)}\Hilb_{\beta,\chi}(X|D)
\longrightarrow\Hilb_{\beta,\chi}(X|D)
\]
and
\[
p_{X|D}:\mathfrak X\times_{\mathsf{Exp}_{\beta,\chi}(X|D)}\Hilb_{\beta,\chi}(X|D)
\longrightarrow\mathfrak X\longrightarrow X
\]
be the natural morphisms, and let
\[
\mathfrak Z\subset\mathfrak X
\times_{\mathsf{Exp}_{\beta,\chi}(X|D)}\Hilb_{\beta,\chi}(X|D)
\]
be the universal subscheme.
Define the tautological perfect complex
\[
\scr L_{\DT}^{[\beta,\chi]}:=\RR q_{X|D,*}\left(p_{X|D}^*\scr L
\otimes^{\LL}\mcal O_{\mathfrak Z}\right)\in K^0_{\bb T}
\left(\Hilb_{\beta,\chi}(X|D)\right).
\]
Its virtual rank is 
$\rank\scr L_{\DT}^{[\beta,\chi]}=\chi\left(\mcal O_{\mathfrak Z}\otimes\scr L\right)
=\chi+c_1(\scr L)\cdot\beta$.
Set
\[
\mathsf e_{\beta,\chi}^{\DT}(\scr L,m):=e^{\bb T\times\bb C_m^*}
\left(\left(\scr L_{\DT}^{[\beta,\chi]}\right)^\vee\otimes e^m\right).
\]

Let $\dHilb^{\mathbf n}(D|\partial D)_{\mathrm{prime}}\subset\mathsf{dPerf}(\mathfrak D^c/\Exp(X|D))_0$ 
be the derived open substack whose classical truncation consists of the prescribed DT-stable 
points supported in $\mathfrak D$, proper over their base, with numerical datum 
$\mathbf n$. Its $(\beta,\chi)$-weighted base change is $\dHilb^{\mathbf n}(D|\partial D)_{\beta,\chi}$, 
by base change and Remark~\ref{rmk-intrinsic-derived-hilb}.
By \cite[Thm.~5.24(1)]{KKPS1}, we know that $\varphi_{\mathrm{ev}_{\mathbf n,D;\beta,\chi}}$
is the pullback of $\varphi_{\mathrm{ev}_{\mathbf n,D;\mathrm{prime}}}$.

For $\alpha\in H^{k}_{\mathrm{crit},\bb T}
\left(\dHilb^{\mathbf n}(D|\partial D)_{\mathrm{prime}}\right)$,
we define a sequence of compatible classes 
$\alpha_{\beta,\chi}\in H^{k}_{\mathrm{crit},\bb T}
\left(\dHilb^{\mathbf n}(D|\partial D)_{\beta,\chi}\right)$
pulled back from $\alpha$.
Then we
define the \textit{relative logarithmic tautological DT invariant
associated to $\alpha$}:
\[
\DT_{\beta}(X|D;\alpha)_{\scr L,m}:=\sum_{\chi\in\bb Z}q^\chi
\int_{[\mathsf r_{X\to D}]^{\mathrm{Lag}}(\alpha_{\beta,\chi}),\bb T}
\mathsf e_{\beta,\chi}^{\DT}(\scr L,m)\in\operatorname{Frac}
H^*_{\bb T\times\bb C_m^*}(\mathrm{pt})((q)).
\]
Similarly, we can define the \textit{relative logarithmic tautological PT invariant
associated to $\alpha$}:
\[
\PT_{\beta}(X|D;\alpha)_{\scr L,m}:=\sum_{\chi\in\bb Z}q^\chi
\int_{[\mathsf r^{\PT}_{X\to D}]^{\mathrm{Lag}}(\alpha_{\beta,\chi}),\bb T}
\mathsf e_{\beta,\chi}^{\PT}(\scr L,m)\in\operatorname{Frac}H^*_{\bb T\times\bb C_m^*}(\mathrm{pt})((q)).
\]
\begin{nota}
For $\alpha_{\beta,\chi}\in
H^{k}_{\mathrm{crit},\bb T}
\left(
\dHilb^{\mathbf n}(D|\partial D)_{\beta,\chi}
\right)$
pulled back from $\alpha$, the degrees of 
$[\mathsf r_{X\to D}]^{\mathrm{Lag}}(\alpha_{\beta,\chi})$
and $[\mathsf r^{\PT}_{X\to D}]^{\mathrm{Lag}}(\alpha_{\beta,\chi})$ are
$2\chi-D\cdot\beta-k$. Hence the degrees of 
$[\mathsf r_{X\to D}]^{\mathrm{Lag}}(\alpha_{\beta,\chi})\cap\mathsf e_{\beta,\chi}^{\DT}(\scr L,m)$
and $[\mathsf r^{\PT}_{X\to D}]^{\mathrm{Lag}}(\alpha_{\beta,\chi})\cap \mathsf e_{\beta,\chi}^{\PT}(\scr L,m)$ are
\[2\chi-D\cdot\beta-k-2\left(
\chi+c_1(\scr L)\cdot\beta
\right)=-k-(D+2c_1(\scr L))\cdot\beta\]
which is independent of $\chi$. Hence to match the dimension,
we can take some \[\alpha\in H^{-(D+2c_1(\scr L))\cdot\beta}_{\mathrm{crit},\bb T}
\left(
\dHilb^{\mathbf n}(D|\partial D)_{\mathrm{prime}}
\right)\] uniformly.
\end{nota}

Recall that the degree-zero logarithmic DT partition function
with tautological insertion is
\[
\DT_0(X|D)_{\scr L,m}:=1+\sum_{n>0}q^n\int_{[\Hilb^n(X|D)]^{\vir}_{\bb T}}e^{\bb T\times\bb C_m^*}
\left((\scr L^{[n]})^\vee\otimes e^m\right).
\]
Note that one can define this using the classical theory 
in \S~\ref{sec-zero-dim-DT4}.

\begin{conj}[Tautological logarithmic DT/PT correspondence]
\label{conj-tautological-log-DT-PT}
There exist compatible choices of orientations such that, for every
\[
\alpha\in H^{k}_{\mathrm{crit},\bb T}
\left(\dHilb^{\mathbf n}(D|\partial D)_{\mathrm{prime}}\right),
\]
one has
\[
\frac{\DT_{\beta}(X|D;\alpha)_{\scr L,m}}{\DT_0(X|D)_{\scr L,m}}
=\PT_{\beta}(X|D;\alpha)_{\scr L,m}\in\operatorname{Frac}H^*_{\bb T\times\bb C_m^*}
(\mathrm{pt})((q)).
\]
\end{conj}
If $D=\emptyset$ and we choose $\alpha=1$, then 
this reduces to the absolute tautological DT/PT correspondence
in \cite[Conj.~0.13]{CKM22}.

\subsection{Family theory of Calabi-Yau snc degenerations}\label{sec-family-snc-dege}
First of all, we point out that the construction of the virtual classes 
of Calabi-Yau snc degenerations can be stated in Chow groups using 
virtual pullback theories \cite{park21,park24}. However, 
to be compatible with the relative theory, here we will use shifted 
Lagrangian classes.

\begin{defi}\label{defi-snc-degeneration}
Let \[\pi:\mcal X\to\bb A^1\] 
be a flat family  of quasi-projective varieties such that $\mcal X$ is a smooth quasi-projective variety.
We say it is a \textit{simple normal crossing degeneration} if the following conditions hold.
\begin{itemize}
	\item If $0\neq b\in\bb A^1$, then $\mcal X_b=\pi^{-1}(b)$ is smooth.
	\item The central fibre $\mcal X_0:=\pi^{-1}(0)$ is an snc divisor of $\mcal X$.
\end{itemize}
We say that such a degeneration is \textit{Calabi-Yau} if $\omega_{\mcal X/\bb A^1}\cong\scr O_{\mcal X}$.
\end{defi}
\begin{nota}\label{rmk-zariski-open-replace}
In this paper, we only consider the degeneration over $(\bb A^1,0)$.
However, all the theories in this paper hold after replacing $\bb A^1$ by $(U,0)$ for any 
Zariski open curve $U\subset\bb A^1$ containing $0$ since $U$ is also Calabi-Yau.
\end{nota}
In this case, let $\mcal D:=X_0:=\pi^{-1}(0)$. Then
$(\mcal X,\mcal D)$ is also a log Calabi-Yau snc pair since $X_t\sim0$ for any $t\in\bb A^1$.
We will consider logarithmic morphism $(\mcal X,\mcal D)\to(\bb A^1,0)$ which is log smooth.
Then this induces 
\[
\pi^{\mathrm{trop}}:\Sigma_{\mcal X}:=\Sigma_{\mcal X|\mcal D}\to\bb R_{\geq0}.
\]
Let $\Delta_{X_0}=(\pi^{\mathrm{trop}})^{-1}(1)$. Then 
\[
X_0=\bigcup_{\nu\in\mathrm{vertex}(\Delta_{X_0})}Y_{\nu}
\]
gives the decomposition of irreducible components.

Consider moduli space $T(\Sigma_{\mcal X})$ of embedded $1$-complexes in $\Sigma_{\mcal X}$
equipped with a suitable cone structure.
We restrict to its subspace $T(\Sigma_{\mcal X}/\bb R_{\geq0})$ of 
vertical $1$-complexes, that is, the $1$-complexes contracted by $\pi^{\mathrm{trop}}$,
with the cone structure induced by that of $T(\Sigma_{\mcal X})$.
See \cite[Lem.~6.1.1]{MR25} for details. Then we have a canonical map 
\[
T(\Sigma_{\mcal X}/\bb R_{\geq0})\to\bb R_{\geq0}.
\]
Then we have the \textit{stack of expansions of degeneration}:
\[
\mathsf{Exp}_{\beta_t,\chi_t}(\mcal X/\bb A^1)\to\bb A^1
\]
of relative dimension zero.

\begin{ass}\label{ass-compactifiable-families}
Let $\mathfrak X\to\Exp(\mcal X/\bb A^1)$ be the original universal vertical expansion.
On the finite-type expansion loci under consideration, we assume that the vertical expansion 
data are compactifiable: there exists a refined cone structure such that 
$\mathfrak X$ is an open substack of a full projective universal 
family $\mathfrak X^c\to\Exp(\mcal X/\bb A^1)$ whose cone map is combinatorially flat with 
reduced fibres, with smooth base. 
\end{ass}

The following local case satisfies the Assumption \ref{ass-compactifiable-families}.
\begin{lem}
\label{lem-compactifiable-local-families}
Let $\mcal Y\to\bb A^1$ be a projective snc degeneration,
and let $\mcal S$ be a horizontal snc divisor such that
$(\mcal Y,\mcal Y_0\cup\mcal S)\to(\bb A^1,0)$ is logarithmically smooth.
Let $\mcal E$ be a vector bundle on $\mcal Y$, and set
\[
\eta:\mcal X:=\Tot_{\mcal Y}(\mcal E)\to\mcal Y,\qquad
\mcal D':=\eta^{-1}(\mcal S),\qquad
\mcal D:=\mcal X_0\cup\mcal D'.
\]
Then the compatible refinements
of the base and universal cone and integral structures can be chosen
to satisfy Assumption \ref{ass-compactifiable-families}.
\end{lem}
\begin{proof}
Consider the projective bundle parametrizing lines
\[
\overline{\mcal X}:=
\bb P_{\mcal Y}(\scr O_{\mcal Y}\oplus\mcal E)
\longrightarrow\mcal Y.
\]
Its standard affine chart is $\mcal X$, and it is projective over
$\bb A^1$. Equip it with the logarithmic boundary pulled back from
$\mcal Y_0\cup\mcal S$, the projections identify
\[
\Sigma_{\mcal X|\mcal D}
\cong\Sigma_{\mcal Y|\mcal Y_0\cup\mcal S}
\cong\Sigma_{\overline{\mcal X}|
\overline{\mcal X}_0\cup\overline{\mcal D'}},
\]
compatibly with their maps to $\bb R_{\geq0}$.

On a finite-type locus, the universal vertical complex is embedded in
\[
\Sigma_{\mcal Y|\mcal Y_0\cup\mcal S}\times_{\bb R_{\geq0}}T(\Sigma_{\mcal X|\mcal D}/\bb R_{\geq0}).
\]
This fibre product is the equal-height locus in the corresponding product of cone spaces.
Apply the projective completion construction of \cite[Rmk.~3.5.2]{MR24} in that product,
and refine by the equality of the two height functions. 
Restricting the resulting subdivision to the equal-height locus gives a complete projective 
subdivision extending the universal vertical complex. We then apply combinatorial 
flattening to its projection to the vertical expansion space. 
Combining \cite[Prop.~1.7.2]{MR25}, we have the full family
\[
\mathfrak Y^c\longrightarrow
\Exp((\mcal X|\mcal D)/\bb A^1)
\]
whose cone map is combinatorially flat with reduced fibres,
with smooth base cone space. The completion is chosen projective.
Since $\mcal Y\to\bb A^1$ is projective, this full family is
projective over the expansion stack.
Now define
\[
\mathfrak X^c:=
\bb P_{\mathfrak Y^c}(\scr O_{\mathfrak Y^c}\oplus p^*\mcal E)
\cong\mathfrak Y^c\times_{\mcal Y}\overline{\mcal X}.
\]
This is the full expansion of $\overline{\mcal X}$ obtained by
strict smooth base change. It is projective and flat with reduced
fibres over the expansion stack, and its cone map is the same as
that of $\mathfrak Y^c$.
\end{proof}

Following Assumption \ref{ass-compactifiable-families}, e.g.,
the case in Lemma \ref{lem-compactifiable-local-families},
we get the family version of logarithmic (derived-) Hilbert stacks 
\[
\Hilb_{\beta_t,\chi_t}(\mcal X/\bb A^1)\subset
\dHilb_{\beta_t,\chi_t}(\mcal X/\bb A^1)\to\mathsf{Exp}_{\beta_t,\chi_t}(\mcal X/\bb A^1)\to\bb A^1
\]
such that $\Hilb_{\beta_t,\chi_t}(\mcal X/\bb A^1)$ is separated and of finite type
and is proper if $\pi:\mcal X\to\bb A^1$ is proper; see also \cite[Thm.~6.3.1]{MR25}.

Similar to Lemma \ref{lem:log-crepant-MR-pair}, we have the following 
generalization of \cite[Prop.~2.11]{CZZ24}.

\begin{lem}\label{lem:crepant-MR-degeneration}
Let $\pi:\mathcal X\to\mathbb A^1$ be a simple normal crossing degeneration. 
Equip $\mathcal X$ with the divisorial logarithmic
structure induced only by the central fibre $\mathcal X_0$, and similarly for $(\bb A^1,0)$. 
Fix a cone structure $\lambda$ for the vertical expansion
problem such that $\Exp(\mathcal X/\mathbb A^1)_\lambda$ is smooth and the universal cone family
is combinatorially flat with reduced fibres. Consider
the corresponding universal vertical expansion
\[
        f:\mathfrak X\to\Exp(\mathcal X/\mathbb A^1)_\lambda
\]
with the contraction map $p:\mathfrak X\to
        Z:=\mathcal X\times_{\mathbb A^1}\Exp(\mathcal X/\mathbb A^1)_\lambda$.
Then there is a canonical isomorphism
\[
        \omega_{\mathfrak X/\Exp(\mathcal X/\mathbb A^1)_\lambda}
        \cong
        p^*\pr_{\mathcal X}^*
        \omega_{\mathcal X/\mathbb A^1}.
\]
In particular, any choice of trivialization of
$\omega_{\mathcal X/\mathbb A^1}$ induces an isomorphism
\[
        \omega_{\mathfrak X/\Exp(\mathcal X/\mathbb A^1)_\lambda}
        \cong
        \mathcal O_{\mathfrak X}.
\]
\end{lem}

\begin{proof}
The proof is similar to that of Lemma \ref{lem:log-crepant-MR-pair} with some 
modifications.

First we note that, in these logarithmic structures,
the resulting logarithmic morphism $(\mathcal X,\mathcal X_0)\longrightarrow(\mathbb A^1,0)$
is saturated and logarithmically smooth.
All expansion spaces are equipped with their
canonical logarithmic structures. By the construction, after passing
to smooth toroidal charts, the universal vertical expansion
$f:\mathfrak X\to\Exp(\mathcal X/\mathbb A^1)_\lambda$ is flat and Gorenstein.

The map of cone complexes $\Sigma_{\mathcal X}\to\mathbb R_{\ge0}$ is
combinatorially flat with reduced fibres: locally the characteristic monoid
map is $\mathbb N\to\mathbb N^r$, $1\mapsto(1,\ldots,1)$, and the induced
lattice map is primitive on every relevant rank-one face. Thus, by
\cite[Prop.~1.7.3]{MR25}, after pulling back to any smooth toroidal chart of
$\Exp(\mathcal X/\mathbb A^1)_\lambda$, the fs logarithmic fibre product agrees
with the ordinary fibre product:
\[
        \left(
        \mathcal X
        \times_{\mathbb A^1}^{\log}
        \Exp(\mathcal X/\mathbb A^1)_\lambda
        \right)_{\mathrm{und}}
        \cong
        \mathcal X\times_{\mathbb A^1}\Exp(\mathcal X/\mathbb A^1)_\lambda
        =:Z.
\]
Logarithmic differentials commute with this logarithmic base change, so
\[
        \omega^{\log}_{Z/\Exp(\mathcal X/\mathbb A^1)_\lambda}
        \cong
        \pr_{\mathcal X}^*
        \omega^{\log}_{\mathcal X/\mathbb A^1}
        \cong
        \pr_{\mathcal X}^*\omega_{\mathcal X/\mathbb A^1},
\]
where the last equation follows from 
$\omega^{\log}_{\mathcal X/\mathbb A^1}
        =
        \omega_{\mathcal X}(\mathcal X_0)
        \otimes
        \pi^*\omega_{\mathbb A^1}(\{0\})^{-1}\cong\omega_{\mathcal X/\mathbb A^1}$.

By the vertical expansion construction \cite[\S6.3]{MR25}, together with the
construction of universal expansions in \cite[\S3.7]{MR24}, the contraction
$p:\mathfrak X\to Z$ is the open restriction of a logarithmic modification over
$\Exp(\mathcal X/\mathbb A^1)_\lambda$. Hence $p$ is logarithmically étale over
$\Exp(\mathcal X/\mathbb A^1)_\lambda$, and therefore
\[
        \omega^{\log}_{\mathfrak X/\Exp(\mathcal X/\mathbb A^1)_\lambda}
        \cong
        p^*\omega^{\log}_{Z/\Exp(\mathcal X/\mathbb A^1)_\lambda}
        \cong
        p^*\pr_{\mathcal X}^*\omega_{\mathcal X/\mathbb A^1}.
\]

Moreover, let $E_{\Exp(\mathcal X/\mathbb A^1)_\lambda}$ be the
boundary of $\Exp(\mathcal X/\mathbb A^1)_\lambda$, and let
$E_{\mathfrak X}$ be the corresponding boundary of
$\mathfrak X$.
By the same reason as in Lemma \ref{lem:log-crepant-MR-pair},
we have $f^*E_{\Exp(\mathcal X/\mathbb A^1)_\lambda}=E_{\mathfrak X}$.
Hence we have
\begin{align*}
\omega^{\log}_{\mathfrak X/\Exp(\mathcal X/\mathbb A^1)_\lambda}
&=\omega_{\mathfrak X}(E_{\mathfrak X})\otimes
    f^*\omega_{\Exp(\mathcal X/\mathbb A^1)_\lambda}
    (E_{\Exp(\mathcal X/\mathbb A^1)_\lambda})^{-1}  \\
&\cong\omega_{\mathfrak X}\otimes
    f^*\omega_{\Exp(\mathcal X/\mathbb A^1)_\lambda}^{-1} =
    \omega_{\mathfrak X/\Exp(\mathcal X/\mathbb A^1)_\lambda}.
\end{align*}
Combining this with the logarithmic crepant identity above, we obtain
\[
        \omega_{\mathfrak X/\Exp(\mathcal X/\mathbb A^1)_\lambda}
        \cong
        p^*\pr_{\mathcal X}^*\omega_{\mathcal X/\mathbb A^1}.
\]
A trivialization of $\omega_{\mathcal X/\mathbb A^1}$ pulls back to a
trivialization of the left-hand side. This proves the lemma.
\end{proof}

The following is the main result in this section.

\begin{prp}\label{prp-family-shifted}
Let $\pi:\mcal X\to\bb A^1$ be a Calabi-Yau simple normal crossing degeneration
with $\dim\mcal X=n+1$.
We work with relative compactifiable expansion data as in 
Assumption \ref{ass-compactifiable-families}, e.g., Lemma \ref{lem-compactifiable-local-families}.
\begin{enumerate}
    \item Then \[\varpi_{\beta_t,\chi_t}:
\dHilb_{\beta_t,\chi_t}(\mcal X/\bb A^1)\to\mathsf{Exp}_{\beta_t,\chi_t}(\mcal X/\bb A^1)
\]
admits
a $(2-n)$-shifted symplectic structure.
Moreover, such a symplectic form is exact if $n=4$.
\item If $n=4$, and there exists an snc degeneration $\mcal Y\to\bb A^1$ of relative dimension $3$ such that $\mcal X=\mathrm{Tot}(\omega_{\mcal Y/\bb A^1})$, then 
    \[\varpi_{\beta_t,\chi_t}:
\dHilb_{\beta_t,\chi_t}(\mcal X/\bb A^1)\to\mathsf{Exp}_{\beta_t,\chi_t}(\mcal X/\bb A^1)
\]  admits a canonical orientation.
\end{enumerate}
\end{prp}
\begin{proof}
As in the proof of Theorem \ref{thm-curves-shifted-lag}, excision identifies the trace-free 
deformation complexes and the supported trace constructions with those computed on the 
original universal family $\mathfrak X/\mathsf{Exp}(\mcal X/\bb A^1)$. 
The projective containing family is used only to construct the derived enhancement. 
Thus Lemma \ref{lem:crepant-MR-degeneration} applies to the original universal family 
throughout the proof.

For (1), the existence of a shifted symplectic structure follows from 
Lemma \ref{lem:crepant-MR-degeneration} and
a proof similar to that of \cite[Thm.~3.2]{CZZ24}, which is a family version of the AKSZ-type construction.
The shifted symplectic form is given by 
\[
\int_{\mathfrak X/\mathsf{Exp}(\mcal X/\bb A^1)}\mathrm{vol}_{\mathfrak X/\mathsf{Exp}(\mcal X/\bb A^1)}
\wedge\widetilde{\mathrm{ev}_X^*\Omega_{\mathsf{dPerf}}},
\]
where $\Omega_{\mathsf{dPerf}}$ is the canonical $2$-shifted form on $\mathsf{dPerf}$,
with the evaluation map
$\mathrm{ev}_X:\mathfrak X\times_{\mathsf{Exp}(\mcal X/\bb A^1)}\dHilb(\mcal X/\bb A^1)\to\mathsf{dPerf}\times\mathsf{Exp}(\mcal X/\bb A^1)$,
and $\widetilde{\mathrm{ev}_X^*\Omega_{\mathsf{dPerf}}}$ is a lift of $\mathrm{ev}_X^*\Omega_{\mathsf{dPerf}}$ 
to a compactly supported form on the universal substack $\mathfrak Z$.
Moreover, the trivialization $\mathrm{vol}_{\mathfrak X/\mathsf{Exp}(\mcal X/\bb A^1)}:\scr O\cong\omega_{\mathfrak X/\mathsf{Exp}(\mcal X/\bb A^1)}$ 
(induced from Lemma \ref{lem:crepant-MR-degeneration})
and the trace map described in \cite[Eqn.~(3.4)]{CZZ24} induce the integration
$\int_{\mathfrak X/\mathsf{Exp}(\mcal X/\bb A^1)}\mathrm{vol}_{\mathfrak X/\mathsf{Exp}(\mcal X/\bb A^1)}
\wedge(-)$.

Now we assume $n=4$. By the proof of \cite[Prop.~5.11,~Appendix~C]{CZZ24}
and the results in \cite[\S~6]{park24},
to prove the exactness, we only need to prove that for any expanded object $\mathfrak X_b$
and a subscheme $Z$, we have $H^4_{\mathrm{sing}}(\mathfrak X_b,\mathfrak X_b\backslash Z)=0$.
Indeed, this follows from the induction on the number of components and the construction of expanded pairs (in which 
we remove all the $\mathrm{codim}\geq2$ strata)\footnote{See Lemma C.1 in the
\href{https://arxiv.org/abs/2402.16103v1}{version 1} of \cite{CZZ24} for the proof
of Li-Wu's case. Our case is similar.}.

For (2), this follows directly from the proof of \cite[Thm.~4.3]{CZZ24}; see also \cite{CMT21}.
Here we give a sketch. Consider the forgetful map
\[
F:\Hilb(\mcal X/\bb A^1)\to M_{\mathrm{cpt}}(\mathfrak X/\mathsf{Exp}(\mcal X/\bb A^1)),\quad 
(\scr O\twoheadrightarrow\scr F)\mapsto \scr F,
\]
to the locus of properly supported perfect coherent sheaves normal
to the codimension-one strata on on $\mathfrak X$ over $\mathsf{Exp}(\mcal X/\bb A^1)$. 
Consider the following diagram of natural projections:
\[\begin{tikzcd}
	{M_{\mathrm{cpt}}(\mathfrak X/\mathsf{Exp}(\mcal X/\bb A^1))\times_{\mathsf{Exp}(\mcal X/\bb A^1)}\mathfrak X} & \\
	{M_{\mathrm{cpt}}(\mathfrak X/\mathsf{Exp}(\mcal X/\bb A^1))\times_{\mathsf{Exp}(\mcal X/\bb A^1)}\mathfrak Y} & {M_{\mathrm{cpt}}(\mathfrak X/\mathsf{Exp}(\mcal X/\bb A^1))}
	\arrow["\pi", from=1-1, to=2-1]
	\arrow["{\pi_{M_X}}", from=1-1, to=2-2]
	\arrow["{\pi_{M_Y}}", from=2-1, to=2-2]
\end{tikzcd}\]
Let $\bb F$ and $\bb I$ denote the universal objects. Then 
one can prove that 
\begin{align*}
\det\bb T_{\dHilb(\mcal X/\bb A^1)/\mathsf{Exp}(\mcal X/\bb A^1)}|_{\Hilb(\mcal X/\bb A^1)}
&\cong F^*\det(\pi_{M_X,*}\mathbf R\scr Hom(\bb F,\bb F)[1])\\ 
&\cong F^*\det\bb T_{\mathsf d M_{\mathrm{cpt}}(\mathfrak X/\mathsf{Exp}(\mcal X/\bb A^1))/\mathsf{Exp}(\mcal X/\bb A^1)}|_{M_{\mathrm{cpt}}(\mathfrak X/\mathsf{Exp}(\mcal X/\bb A^1))}
\end{align*}
for the natural derived enhancement $\mathsf d M_{\mathrm{cpt}}(\mathfrak X/\mathsf{Exp}(\mcal X/\bb A^1))$,
which is compatible with the trivialization of the squares.
Hence we only need to find the compatible trivialization of the latter one.
The spectral construction, or equivalently the identification 
$\mathsf d M_{\mathrm{cpt}}(\mathfrak X/\mathsf{Exp}(\mcal X/\bb A^1))
\cong\TT^*[-2](\mathsf d M_{\mathrm{cpt}}(\mathfrak Y/\mathsf{Exp}(\mcal X/\bb A^1))/\mathsf{Exp}(\mcal X/\bb A^1))$
with natural projection $p$ to $\mathsf d M_{\mathrm{cpt}}(\mathfrak Y/\mathsf{Exp}(\mcal X/\bb A^1))$
induces the distinguished triangle
\begin{align*}
	&p^*\bb L_{\mathsf d M_{\mathrm{cpt}}(\mathfrak Y/\mathsf{Exp}(\mcal X/\bb A^1))/\mathsf{Exp}(\mcal X/\bb A^1)}[-2]\to \bb T_{\mathsf d M_{\mathrm{cpt}}(\mathfrak X/\mathsf{Exp}(\mcal X/\bb A^1))/\mathsf{Exp}(\mcal X/\bb A^1)}\\
&\to p^*\bb T_{\mathsf d M_{\mathrm{cpt}}(\mathfrak Y/\mathsf{Exp}(\mcal X/\bb A^1))/\mathsf{Exp}(\mcal X/\bb A^1)}\to.
\end{align*}
Then restricting to the classical truncation and taking determinant, we deduce 
the existence of an orientation by Grothendieck-Serre duality.
Note that we always choose the canonical positive 
one as in \cite[Rmk.~4.4]{CZZ24}; that is, 
by spectral triangle as above, the symmetric obstruction complex on
$M_{\mathrm{cpt}}(\mathfrak X/\mathsf{Exp}(\mcal X/\bb A^1))$
is $\bb E=(\bb V\to\bb V^{\vee})$ and
the orientation given by 
$\pm(-\sqrt{-1})^{\rank\bb V}:\scr O\to\scr O$.
We always choose the \textit{plus sign} above as the canonical choice of orientation.
\end{proof}

\begin{nota}
Note that the exactness of $(-2)$-shifted symplectic fibration 
implies the isotropy condition described in \cite[Def.~1.10]{park21};
see \cite[Rmk.~5.2.5]{park24} for details.
\end{nota}

Now assume $n=4$ and that an orientation exists; for example, assume that there exists an snc degeneration $\mcal Y\to\bb A^1$ of relative dimension $3$
with form as in Lemma \ref{lem-compactifiable-local-families} such that $\mcal X=\mathrm{Tot}(\omega_{\mcal Y/\bb A^1})$.
Then by \cite[Cor.~3.1.3]{park24}, we have an induced Lagrangian
\[
\varpi_{\beta_t,\chi_t}:\dHilb_{\beta_t,\chi_t}(\mcal X/\bb A^1)\to\TT^*[-1]\mathsf{Exp}_{\beta_t,\chi_t}(\mcal X/\bb A^1)
\]
which induces the following morphism using shifted Lagrangian classes:
\begin{align*}
&\varpi_{\beta_t,\chi_t}^*\varphi_{\TT^*[-1]\mathsf{Exp}_{\beta_t,\chi_t}(\mcal X/\bb A^1)}[\vdim\dHilb_{\beta_t,\chi_t}(\mcal X/\bb A^1)]\\
&\cong\bb Q_{\dHilb_{\beta_t,\chi_t}(\mcal X/\bb A^1)}[\vdim\dHilb_{\beta_t,\chi_t}(\mcal X/\bb A^1)+1]
\to\bb D\bb Q_{\dHilb_{\beta_t,\chi_t}(\mcal X/\bb A^1)}.
\end{align*}

\begin{defi}
In this case, the \textit{family virtual Lagrangian class} 
is defined to be the following class induced by the previous morphism
via Remark~\ref{rmk-constructions-BM}(1):
\[
[\dHilb_{\beta_t,\chi_t}(\mcal X/\bb A^1)]^{\vir}\in H^{\BM}_{\vdim\dHilb_{\beta_t,\chi_t}(\mcal X/\bb A^1)+1}(\dHilb_{\beta_t,\chi_t}(\mcal X/\bb A^1),\bb Q).
\]
\end{defi}

If there exists a Calabi-Yau torus $\bb T$ acting on the space,
then the virtual class is $\bb T$-equivariant.
Note that the virtual classes of the fibres of family theory 
are independent of the cone structures; see Remark \ref{rmk-cone-independence}.

\section{Degeneration formula for logarithmic \texorpdfstring{$\mathsf{DT}_4$}{DT4}-invariants}
\label{sec-general-degeneration}
Using shifted Lagrangian classes, we prove a $\mathsf{DT}_4$ degeneration formula for 
local Calabi-Yau snc degenerations. In the smooth-pair case, Li-Wu's construction yields a 
Borel-Moore-homological formula that does not require the target to admit a critical-locus 
presentation, in contrast to the Chow-valued formula of \cite{CZZ24}.

\subsection{Decomposition of the central fibre}\label{sec-cutting-maps}
Fix a Calabi-Yau snc degeneration $\pi:\mcal X\to\bb A^1$
with $4$-dimensional fibres satisfying
Assumption~\ref{ass-compactifiable-families}.

Choose a compactifiable whole-family cone structure. After a barycentric
subdivision and a smooth refinement of the expansion base, we may assume
that its boundary is simple normal crossings.
Pull back the original and full universal families along this refinement.
Thus we have the original universal family and flat projective containing family
\[
\mathfrak X\hookrightarrow\mathfrak X^c
\longrightarrow\mathsf{Exp}_{\beta_t,\chi_t}(\mcal X/\bb A^1)
\]
with reduced fibres and smooth expansion base with simple normal crossings boundary.
Write
\[
\varpi_{\beta_t,\chi_t}:
\dHilb_{\beta_t,\chi_t}(\mcal X/\bb A^1)
\longrightarrow
\mathsf{Exp}_{\beta_t,\chi_t}(\mcal X/\bb A^1)
\]
for the structure morphism. The map from the stack of expansion to
$\bb A^1$ need not have reduced fibres.

Let $\mathsf{Exp}_\gamma(X_0)$ be the smooth reduced central components,
indexed by rigid Hilbert $1$-complexes $\gamma$. Set
\[
\mathfrak X_\gamma
:=\left.\mathfrak X\right|_{\mathsf{Exp}_\gamma(X_0)},
\qquad
\mathfrak X_\gamma^c
:=\left.\mathfrak X^c\right|_{\mathsf{Exp}_\gamma(X_0)}.
\]
Define $m_\gamma:=\operatorname{ord}_{\mathsf{Exp}_\gamma(X_0)}(t)$
where $t$ is the fixed coordinate on $\bb A^1$, and $m_\gamma$ is
the height of the primitive generator of the corresponding vertical
ray, with its chosen integral structure.
Now we know that
$\dHilb_\gamma(X_0)$ is the derived Hilbert stack of
$\mathfrak X_\gamma/\mathsf{Exp}_\gamma(X_0)$ with the inherited
log Hilbert conditions, constructed using $\mathfrak X_\gamma^c$.
That is, we have the Cartesian diagram
\[
\begin{tikzcd}
\bigsqcup_\gamma\dHilb_\gamma(X_0)
\arrow[r,"n"]\arrow[d,"\bigsqcup_\gamma\varpi_\gamma"']
\arrow[dr,phantom,"\square"]&
\dHilb_{\beta_t,\chi_t}(\mcal X/\bb A^1)_0\arrow[d,"\varpi_0"]\\
\bigsqcup_\gamma\mathsf{Exp}_\gamma(X_0)\arrow[r,"n"]&
\mathsf{Exp}_{\beta_t,\chi_t}(\mcal X/\bb A^1)_0
\end{tikzcd}
\]
with
\[
\left[\mathsf{Exp}_{\beta_t,\chi_t}(\mcal X/\bb A^1)_0\right]
=n_*\sum_\gamma m_\gamma[\mathsf{Exp}_\gamma(X_0)].
\]
Therefore, Conjecture \ref{conj-Joyce-conj}(3)(4) and
Remark \ref{rmk-KKPS-lagrangian-classes} imply
\begin{equation}\label{eqn-central-virtual-multiplicities}
	i_0^![\dHilb_{\beta_t,\chi_t}(\mcal X/\bb A^1)]^{\vir}
=n_*\sum_\gamma m_\gamma[\dHilb_\gamma(X_0)]^{\vir}.
\end{equation}

Now fix $\gamma$ and we will consider the decomposition of $\dHilb_\gamma(X_0)$ over
$\mathsf{Exp}_\gamma(X_0)$.
Let $\nu,e$ denote the vertices and edges of the universal expansion
over the generic stratum indexed by $\gamma$, including all tube
vertices. Write $\gamma(\nu)$ for the induced numerical decorations
and $n_e$ for the boundary lengths. Partially normalize $\mathfrak X_\gamma$ along the 
distinguished
double loci indexed by the edges. 
The two flags of each edge are copies of the same double locus, so
\[
\mathfrak D_{u,e}\cong\mathfrak D_e\cong\mathfrak D_{\nu,e}
\quad\text{if $e$ joins $u$ and $\nu$}
\]
for $\mathfrak D_e\subset\mathfrak X_\gamma$
and $\mathfrak D_{u/\nu,e}\subset\mathfrak Y_{u/\nu}$.
Hence we have 
\begin{equation}\label{eqn-gluing-diagram}
\mathfrak X_\gamma\cong
\varinjlim\left(
\begin{tikzcd}
&\mathfrak D_e\arrow[dl]\arrow[dr]&\\
\mathfrak Y_u&&\mathfrak Y_\nu
\end{tikzcd},
\quad\text{if $e$ joins $u$ and $\nu$}
\right).
\end{equation}

Taking the normalizations of the corresponding stratum closures
in $\mathfrak X_\gamma^c$ gives containing families
$\mathfrak Y_\nu\subset\mathfrak Y_\nu^c$ and $\mathfrak D_e\subset\mathfrak D_e^c$,
which are projective and flat with reduced fibre.

Let $Y_\nu$ label the component at $\nu$, and set $D_\nu=\partial Y_\nu$.
Then the original families satisfy $\mathfrak D_\nu=\coprod_{e\in E(\nu)}\mathfrak D_{\nu,e}$.
We consider open substack by the similar arguments in Lemma\ref{lem-derived-hilb}:
\[
\dHilb_{\Delta,\gamma(\nu)}(Y_\nu|\partial Y_\nu)
\hookrightarrow
\mathsf{dPerf}(\mathfrak Y_\nu^c/\mathsf{Exp}_\gamma(X_0))_0
\]
which is the derived Hilbert stack of the family
$\mathfrak Y_\nu/\mathsf{Exp}_\gamma(X_0)$ parameterizes the algebraically transverse, 
``DT''-stable subschemes (inherited from the ambient space) with support proper over the base and contained in
$\mathfrak Y_\nu$.
Similarly we consider the boundary Hilbert stacks parameterize points with the prescribed
lengths $n_e$, using one containing family $\mathfrak D_e^c$. Thus
\[
\begin{aligned}
\dHilb_{\Delta,\gamma(\nu)}(D_\nu|\partial D_\nu)
&\cong{\prod_{e\in E(\nu)}}_{/\mathsf{Exp}_\gamma(X_0)}
\dHilb_{\Delta,\gamma(\nu)}(D_{\nu,e}|\partial D_{\nu,e}),\\
\dHilb_{\Delta,\gamma(u)}(D_{u,e}|\partial D_{u,e})
&\cong
\dHilb_{\Delta,\gamma(\nu)}(D_{\nu,e}|\partial D_{\nu,e}).
\end{aligned}
\]
This holds since we have Remark \ref{rmk-intrinsic-derived-hilb} and
\[
\mathfrak D_\nu=\coprod_{e\in E(\nu)}\mathfrak D_{\nu,e}
\subset\coprod_{e\in E(\nu)}\mathfrak D_{\nu,e}^c,
\]
although $\mathfrak D_{\nu,e}^c$ will intersect to each other in the original ambient space.

Restriction gives
\[
\mathsf r_{Y_\nu\to D_\nu}:
\dHilb_{\Delta,\gamma(\nu)}(Y_\nu|\partial Y_\nu)
\longrightarrow
\dHilb_{\Delta,\gamma(\nu)}(D_\nu|\partial D_\nu).
\]
Taking the relative diagonal for each edge defines
\[
\Delta_\gamma:=
{\prod_e}_{/\mathsf{Exp}_\gamma(X_0)}
\dHilb_{\Delta,\gamma(\nu)}(D_{\nu,e}|\partial D_{\nu,e}),
\qquad
\delta_\gamma:
\Delta_\gamma\longrightarrow
{\prod_\nu}_{/\mathsf{Exp}_\gamma(X_0)}
\dHilb_{\Delta,\gamma(\nu)}(D_\nu|\partial D_\nu).
\]
In the first product either flag of $e$ gives the same factor.

\begin{prp}\label{prp-fiberprod-gluing-general}
We have a homotopical pullback diagram of derived stacks
\[
\begin{tikzcd}
{\dHilb_\gamma(X_0)}&&{\Delta_\gamma}\\
{\displaystyle{\prod_\nu}_{/\mathsf{Exp}_\gamma(X_0)}
\dHilb_{\Delta,\gamma(\nu)}(Y_\nu|\partial Y_\nu)}
&&
{\displaystyle{\prod_\nu}_{/\mathsf{Exp}_\gamma(X_0)}
\dHilb_{\Delta,\gamma(\nu)}(D_\nu|\partial D_\nu).}
\arrow[from=1-1,to=1-3]
\arrow["{i_{\Delta,\gamma}}",from=1-1,to=2-1]
\arrow["\square"{description},draw=none,from=1-1,to=2-3]
\arrow[from=1-3,to=2-3]
\arrow["{{\prod_\nu}_{/\mathsf{Exp}_\gamma(X_0)}\mathsf r_{Y_\nu\to D_\nu}}",
from=2-1,to=2-3]
\end{tikzcd}
\]
The diagram is a shifted
Lagrangian intersection over $\mathsf{Exp}_\gamma(X_0)$.
In the total-space case, the restriction maps and diagonals are
exact, with the compatible total-space orientations.
\end{prp}
\begin{proof}
All constructions below are relative to $\mathsf{Exp}_\gamma(X_0)$,
with the expansion data fixed above.
By the decomposition of the original boundary families, each edge
contributes two copies of the same boundary Hilbert stack;
their relative diagonals define $\Delta_\gamma$.

Testing on derived affine schemes and checking \'etale locally,
\cite[Thm.~1.3, Rmk.~1.7]{Cho25} and descent for perfect complexes
give the following homotopical fibre product for
Fig.~(\ref{eqn-gluing-diagram}):
\begin{footnotesize}
\[
\begin{tikzcd}
{\dMap_{\mathsf{Exp}_\gamma(X_0)}
(\mathfrak X_\gamma,\mathsf{dPerf}\times\mathsf{Exp}_\gamma(X_0))}
&&{\Delta_\gamma^{\mathsf{dMap}}}\\
{\displaystyle{\prod_\nu}_{/\mathsf{Exp}_\gamma(X_0)}
\dMap_{\mathsf{Exp}_\gamma(X_0)}
(\mathfrak Y_\nu,\mathsf{dPerf}\times\mathsf{Exp}_\gamma(X_0))}
&&
{\displaystyle{\prod_\nu}_{/\mathsf{Exp}_\gamma(X_0)}
{\prod_{e\in E(\nu)}}_{/\mathsf{Exp}_\gamma(X_0)}\dMap_{\mathsf{Exp}_\gamma(X_0)}
(\mathfrak D_{\nu,e},\mathsf{dPerf}\times\mathsf{Exp}_\gamma(X_0))}
\arrow[from=1-1,to=1-3]\arrow[from=1-1,to=2-1]
\arrow["\square"{description},draw=none,from=1-1,to=2-3]
\arrow[from=1-3,to=2-3]
\arrow["{{\prod_\nu}_{/\mathsf{Exp}_\gamma(X_0)}\mathsf r_{Y_\nu\to D_\nu}}",from=2-1,to=2-3]
\end{tikzcd}
\]
\end{footnotesize}
Here $\Delta_\gamma^{\mathsf{dMap}}$ is the product of the relative
diagonals for the two mapping-stack factors at each edge.
The same gluing statement holds for $\mathsf{dPic}$, and the
determinant morphisms commute with restriction. Taking the fibres
over the structure sheaves, with their trivializations, therefore
gives the analogous pullback diagram for perfect complexes with
trivialized determinant.
Then our pullback diagram follows from Remark \ref{rmk-intrinsic-derived-hilb},
testing via derived affine schemes and the geometric meaning of 
the fiber products. The other statements are easy to see.
\end{proof}

\subsection{General degeneration formula}
Define the \textit{DT-perverse sheaf} by
\[\varphi_{\dHilb_{\gamma}(D)}:=\pi_{\gamma,D}^{\varphi}(\bb Q_{\mathsf{Exp}_\gamma(X_0)}[0])\] 
for the structure morphism 
\[\pi_{\gamma,D}:{\prod_{\nu}}_{/\mathsf{Exp}_\gamma(X_0)}\dHilb_{\Delta,\gamma(\nu)}(D_{\nu}|\partial D_{\nu})
\to\mathsf{Exp}_\gamma(X_0).\] Define its
critical cohomology as
\[H^*_{\mathrm{crit}}\left({\prod_{\nu}}_{/\mathsf{Exp}_\gamma(X_0)}\dHilb_{\Delta,\gamma(\nu)}(D_{\nu}|\partial D_{\nu})\right)
=H^*\left(\varphi_{\dHilb_{\gamma}(D)}\right).\]

\begin{thm}[General degeneration formula for cycles]\label{thm-degeneration-general-curves}
Let \[\pi:\mcal X\to\bb A^1\] 
be a Calabi-Yau simple normal crossing degeneration of $4$-folds
satisfying Assumption \ref{ass-compactifiable-families}, e.g.,
Lemma \ref{lem-compactifiable-local-families}.
Assume there exists an snc degeneration $\mcal Y\to\bb A^1$ such that
$\mcal X=\mathrm{Tot}(\omega_{\mcal Y/\bb A^1})$ as described in Proposition
\ref{prp-family-shifted}.
Then we have the following results.
\begin{enumerate}
	\item If $0\neq b\in\bb A^1$, then 
	\[
	i_b^![\dHilb_{\beta_t,\chi_t}(\mcal X/\bb A^1)]^{\vir}=[\dHilb_{\beta_b,\chi_b}(\mcal X_b)]^{\vir}.
	\]
	\item If $b=0\in\bb A^1$, then
	\[
	i_0^![\dHilb_{\beta_t,\chi_t}(\mcal X/\bb A^1)]^{\vir}=n_*\sum_{\gamma}m_\gamma[\dHilb_{\gamma}(X_0)]^{\vir}.
	\]
	Moreover, we have
	\[
	i_{\Delta,\gamma,*}[\dHilb_{\gamma}(X_0)]^{\vir}=
	\left[{\prod_{\nu}}_{/\mathsf{Exp}_\gamma(X_0)}\mathsf r_{Y_{\nu}\to D_{\nu}}\right]^{\mathrm{Lag}}\left(\Delta_{\gamma}^{\mathrm{Lag}}\right),
	\]
Here $m_\gamma$ and the component Hilbert stacks are as in
\S~\ref{sec-cutting-maps}, over
$\mathsf{Exp}_\gamma(X_0)$.
The operator in the second identity is the Lagrangian pullback
of the fibre product of restriction maps
	\[
	\mathsf r_{Y_{\nu}\to D_{\nu}}:\dHilb_{\Delta,\gamma(\nu)}(Y_{\nu}|\partial Y_{\nu})\to\dHilb_{\Delta,\gamma(\nu)}(D_{\nu}|\partial D_{\nu})
	\]
	and $\Delta_{\gamma}^{\mathrm{Lag}}\in H^*_{\mathrm{crit}}\left({\prod_{\nu}}_{/\mathsf{Exp}_\gamma(X_0)}\dHilb_{\Delta,\gamma(\nu)}(D_{\nu}|\partial D_{\nu})\right)$ is the Lagrangian diagonal class
	determined by 
	\[
	\bb Q_{{\prod_{\nu}}_{/\mathsf{Exp}_\gamma(X_0)}\dHilb_{\Delta,\gamma(\nu)}(D_{\nu}|\partial D_{\nu})}
	\to\delta_{\gamma,*}\bb Q_{\Delta_{\gamma}}\xrightarrow{[\delta_{\gamma}]^{\mathrm{Lag}}}
	\delta_{\gamma,*}\delta_{\gamma}^!\varphi_{\dHilb_{\gamma}(D)}\to\varphi_{\dHilb_{\gamma}(D)}.
	\]
\end{enumerate}
\end{thm}
\begin{proof}
Statement (1) follows from base change for the family virtual class.
The first identity in (2) is Eqn. \eqref{eqn-central-virtual-multiplicities}.
For the second, we have two types of oriented exact $(-1)$-shifted Lagrangians related over $\mathsf{Exp}_\gamma(X_0)$:
\begin{align*}
	\mathsf r_{Y_{\nu}\to D_{\nu}}&:\dHilb_{\Delta,\gamma(\nu)}(Y_{\nu}|\partial Y_{\nu})\to\dHilb_{\Delta,\gamma(\nu)}(D_{\nu}|\partial D_{\nu}),\\ 
	\delta_{\gamma,e}&:\dHilb_{\Delta,\gamma(\nu)}\left(D_{\nu,e}|\partial D_{\nu,e}\right)
	\to\dHilb_{\Delta,\gamma(\nu)}\left(D_{\nu,e}|\partial D_{\nu,e}\right)\times_{\mathsf{Exp}_\gamma(X_0)}\dHilb_{\Delta,\gamma(\nu)}\left(D_{\nu,e}|\partial D_{\nu,e}\right).
\end{align*}
For the diagonal, the product shifted symplectic structure on the two boundary factors 
adjacent to an edge is taken with the opposite signs.
Taking products gives oriented exact $(-1)$-shifted Lagrangians
relative to $\mathsf{Exp}_\gamma(X_0)$:
\begin{align*}
	{\prod_{\nu}}_{/\mathsf{Exp}_\gamma(X_0)}\mathsf r_{Y_{\nu}\to D_{\nu}}&:
	{\prod_{\nu}}_{/\mathsf{Exp}_\gamma(X_0)}\dHilb_{\Delta,\gamma(\nu)}(Y_{\nu}|\partial Y_{\nu})\to{\prod_{\nu}}_{/\mathsf{Exp}_\gamma(X_0)}\dHilb_{\Delta,\gamma(\nu)}(D_{\nu}|\partial D_{\nu}),\\ 
	\delta_{\gamma}&:\Delta_{\gamma}\to\displaystyle{\prod_{\nu}}_{/\mathsf{Exp}_\gamma(X_0)}\dHilb_{\Delta,\gamma(\nu)}(D_{\nu}|\partial D_{\nu}).
\end{align*}
By Proposition \ref{prp-fiberprod-gluing-general}, we have the following pullback diagram 
of derived stacks in $\mathsf{dSt}_{\mathsf{Exp}_\gamma(X_0)}$:
\[\begin{tikzcd}
	{\dHilb_{\gamma}(X_0)} && {\Delta_{\gamma}} \\
	{\displaystyle{\prod_{\nu}}_{/\mathsf{Exp}_\gamma(X_0)}\dHilb_{\Delta,\gamma(\nu)}(Y_{\nu}|\partial Y_{\nu})} && {\displaystyle{\prod_{\nu}}_{/\mathsf{Exp}_\gamma(X_0)}\dHilb_{\Delta,\gamma(\nu)}(D_{\nu}|\partial D_{\nu}).}
	\arrow[from=1-1, to=1-3]
	\arrow[from=1-1, to=2-1]
	\arrow["\square"{description}, draw=none, from=1-1, to=2-3]
	\arrow[from=1-3, to=2-3]
	\arrow["{\prod_{\nu}}_{/\mathsf{Exp}_\gamma(X_0)}\mathsf r_{Y_{\nu}\to D_{\nu}}", from=2-1, to=2-3]
\end{tikzcd}\]
This forms a composition of Lagrangian correspondences in $\mathsf{dSt}_{\mathsf{Exp}_\gamma(X_0)}$:
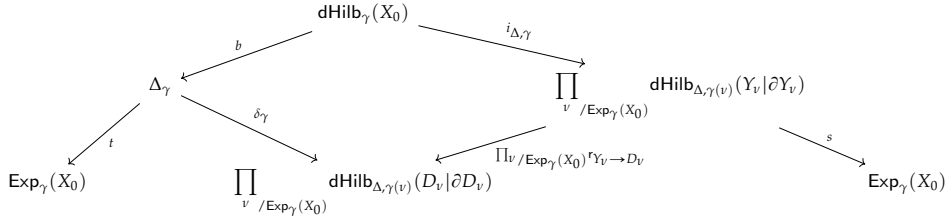
\begin{scriptsize}
\[\begin{tikzcd}
	&& {\dHilb_{\gamma}(X_0)} && \\
	& {\Delta_{\gamma}} && {\displaystyle{\prod_{\nu}}_{/\mathsf{Exp}_\gamma(X_0)}\dHilb_{\Delta,\gamma(\nu)}(Y_{\nu}|\partial Y_{\nu})} \\
	{\mathsf{Exp}_\gamma(X_0)} && {\displaystyle{\prod_{\nu}}_{/\mathsf{Exp}_\gamma(X_0)}\dHilb_{\Delta,\gamma(\nu)}(D_{\nu}|\partial D_{\nu})} && {\mathsf{Exp}_\gamma(X_0)}
	\arrow["b"', from=1-3, to=2-2]
	\arrow["{i_{\Delta,\gamma}}", from=1-3, to=2-4]
	\arrow["t", from=2-2, to=3-1]
	\arrow["{\delta_{\gamma}}", from=2-2, to=3-3]
	\arrow["{{\prod_{\nu}}_{/\mathsf{Exp}_\gamma(X_0)}\mathsf r_{Y_{\nu}\to D_{\nu}}}", from=2-4, to=3-3]
	\arrow["s", from=2-4, to=3-5]
\end{tikzcd}\]
\end{scriptsize}
Then by the associativity of shifted Lagrangian classes in Conjecture \ref{conj-Joyce-conj}(2)
and Remark \ref{rmk-KKPS-lagrangian-classes}, the class
\[
i_{\Delta,\gamma,*}[\dHilb_{\gamma}(X_0)]^{\vir}\in H^{\BM}_*\left({\prod_{\nu}}_{/\mathsf{Exp}_\gamma(X_0)}\dHilb_{\Delta,\gamma(\nu)}(Y_{\nu}|\partial Y_{\nu}),\bb Q\right)
\]
is induced by 
\begin{equation}\label{eqn-Lag-01}
	\begin{aligned}
		&\bb Q_{{\prod_{\nu}}_{/\mathsf{Exp}_\gamma(X_0)}\dHilb_{\Delta,\gamma(\nu)}(Y_{\nu}|\partial Y_{\nu})}[\vdim]
	\to i_{\Delta,\gamma,*}\bb Q_{\dHilb_{\gamma}(X_0)}[\vdim]=i_{\Delta,\gamma,*}b^*\bb Q_{\Delta_{\gamma}}[\vdim]\\
	&\xrightarrow{[\delta_{\gamma}]^{\mathrm{Lag}}}i_{\Delta,\gamma,*}b^*\delta_{\gamma}^!\varphi_{\dHilb_{\gamma}(D)}[\vdim]
	\to i_{\Delta,\gamma,*}i_{\Delta,\gamma}^!\left({\prod_{\nu}}_{/\mathsf{Exp}_\gamma(X_0)}\mathsf r_{Y_{\nu}\to D_{\nu}}\right)^*\varphi_{\dHilb_{\gamma}(D)}[\vdim]\\ 
	&\to\left({\prod_{\nu}}_{/\mathsf{Exp}_\gamma(X_0)}\mathsf r_{Y_{\nu}\to D_{\nu}}\right)^*\varphi_{\dHilb_{\gamma}(D)}[\vdim]\xrightarrow{\left[{\prod_{\nu}}_{/\mathsf{Exp}_\gamma(X_0)}\mathsf r_{Y_{\nu}\to D_{\nu}}\right]^{\mathrm{Lag}}}
	\bb D\bb Q_{{\prod_{\nu}}_{/\mathsf{Exp}_\gamma(X_0)}\dHilb_{\Delta,\gamma(\nu)}(Y_{\nu}|\partial Y_{\nu})}.
	\end{aligned}
\end{equation}
The boundary Hilbert stacks are separated over the common base, so
$\delta_\gamma$ and its base change $i_{\Delta,\gamma}$ are closed
immersions; see \cite[\href{https://stacks.math.columbia.edu/tag/0DM7}{Tag 0DM7}]{SP26}.
Hence Theorem~\ref{thm-six-functor}(4) gives
\[
\delta_{\gamma,!}\cong\delta_{\gamma,*},\qquad
i_{\Delta,\gamma,!}\cong i_{\Delta,\gamma,*}.
\]

Moreover, the class
\[
\left[{\prod_{\nu}}_{/\mathsf{Exp}_\gamma(X_0)}\mathsf r_{Y_{\nu}\to D_{\nu}}\right]^{\mathrm{Lag}}\left(\Delta_{\gamma}^{\mathrm{Lag}}\right)
\in H^{\BM}_*\left({\prod_{\nu}}_{/\mathsf{Exp}_\gamma(X_0)}\dHilb_{\Delta,\gamma(\nu)}(Y_{\nu}|\partial Y_{\nu}),\bb Q\right)
\]
is induced by 
\begin{equation}\label{eqn-Lag-02}
	\begin{aligned}
		&\bb Q_{{\prod_{\nu}}_{/\mathsf{Exp}_\gamma(X_0)}\dHilb_{\Delta,\gamma(\nu)}(Y_{\nu}|\partial Y_{\nu})}[\vdim]
	=\left({\prod_{\nu}}_{/\mathsf{Exp}_\gamma(X_0)}\mathsf r_{Y_{\nu}\to D_{\nu}}\right)^*\bb Q_{{\prod_{\nu}}_{/\mathsf{Exp}_\gamma(X_0)}\dHilb_{\Delta,\gamma(\nu)}(D_{\nu}|\partial D_{\nu})}[\vdim]\\
	&\to\left({\prod_{\nu}}_{/\mathsf{Exp}_\gamma(X_0)}\mathsf r_{Y_{\nu}\to D_{\nu}}\right)^*\delta_{\gamma,*}\bb Q_{\Delta_{\gamma}}[\vdim]
	\xrightarrow{[\delta_{\gamma}]^{\mathrm{Lag}}}\left({\prod_{\nu}}_{/\mathsf{Exp}_\gamma(X_0)}\mathsf r_{Y_{\nu}\to D_{\nu}}\right)^*\delta_{\gamma,*}\delta_{\gamma}^!\varphi_{\dHilb_{\gamma}(D)}[\vdim]\\
	&\to\left({\prod_{\nu}}_{/\mathsf{Exp}_\gamma(X_0)}\mathsf r_{Y_{\nu}\to D_{\nu}}\right)^*\varphi_{\dHilb_{\gamma}(D)}[\vdim]
	\xrightarrow{\left[{\prod_{\nu}}_{/\mathsf{Exp}_\gamma(X_0)}\mathsf r_{Y_{\nu}\to D_{\nu}}\right]^{\mathrm{Lag}}}\bb D\bb Q_{{\prod_{\nu}}_{/\mathsf{Exp}_\gamma(X_0)}\dHilb_{\Delta,\gamma(\nu)}(Y_{\nu}|\partial Y_{\nu})}.
	\end{aligned}
\end{equation}
Here we have used the constructions in Remark \ref{rmk-constructions-BM}.

To compare Eqn. \eqref{eqn-Lag-01} and Eqn. \eqref{eqn-Lag-02}, use the
exchange isomorphisms of Theorem~\ref{thm-six-functor}(3).

The first diagram:
\[\begin{tikzcd}
	{\bb Q_{{\prod_{\nu}}_{/\mathsf{Exp}_\gamma(X_0)}\dHilb_{\Delta,\gamma(\nu)}(Y_{\nu}|\partial Y_{\nu})}[\vdim]} & {\left({\prod_{\nu}}_{/\mathsf{Exp}_\gamma(X_0)}\mathsf r_{Y_{\nu}\to D_{\nu}}\right)^*\delta_{\gamma,*=!}\bb Q_{\Delta_{\gamma}}[\vdim]} \\
	{i_{\Delta,\gamma,*=!}b^*\bb Q_{\Delta_{\gamma}}[\vdim]}
	\arrow[from=1-1, to=1-2]
	\arrow[from=1-1, to=2-1]
	\arrow["{\cong,\mathrm{Ex}_!^*}"', from=2-1, to=1-2]
\end{tikzcd}\]
The horizontal arrow and the composite of the other two
have the same adjoint $\mathrm{id}_{\bb Q_{\dHilb_{\gamma}(X_0)}[\vdim]}$.
Hence this diagram is commutative up to isomorphisms; see also \cite[Exam.~3.2.1]{FYZ23}.

The second diagram:
\[\begin{tikzcd}[column sep = small]
	{\left({\prod_{\nu}}_{/\mathsf{Exp}_\gamma(X_0)}\mathsf r_{Y_{\nu}\to D_{\nu}}\right)^*\delta_{\gamma,*=!}\bb Q_{\Delta_{\gamma}}[\vdim]} & {\left({\prod_{\nu}}_{/\mathsf{Exp}_\gamma(X_0)}\mathsf r_{Y_{\nu}\to D_{\nu}}\right)^*\delta_{\gamma,*=!}\delta_{\gamma}^!\varphi_{\dHilb_{\gamma}(D)}[\vdim]} \\
	{i_{\Delta,\gamma,*=!}b^*\bb Q_{\Delta_{\gamma}}[\vdim]} & {i_{\Delta,\gamma,*=!}b^*\delta_{\gamma}^!\varphi_{\dHilb_{\gamma}(D)}[\vdim]}
	\arrow["{[\delta_{\gamma}]^{\mathrm{Lag}}}", from=1-1, to=1-2]
	\arrow["{\cong,\mathrm{Ex}_!^*}", from=2-1, to=1-1]
	\arrow["{[\delta_{\gamma}]^{\mathrm{Lag}}}", from=2-1, to=2-2]
	\arrow["{\cong,\mathrm{Ex}_!^*}", from=2-2, to=1-2]
\end{tikzcd}\]
which is commutative by the naturality of $\mathrm{Ex}_!^*$.

The third diagram:
\begin{small}
\[\begin{tikzcd}[column sep=small]
	{\left({\prod_{\nu}}_{/\mathsf{Exp}_\gamma(X_0)}\mathsf r_{Y_{\nu}\to D_{\nu}}\right)^*\delta_{\gamma,*=!}\delta_{\gamma}^!\varphi_{\dHilb_{\gamma}(D)}[\vdim]} & {\left({\prod_{\nu}}_{/\mathsf{Exp}_\gamma(X_0)}\mathsf r_{Y_{\nu}\to D_{\nu}}\right)^*\varphi_{\dHilb_{\gamma}(D)}[\vdim]} \\
	{i_{\Delta,\gamma,*=!}b^*\delta_{\gamma}^!\varphi_{\dHilb_{\gamma}(D)}[\vdim]} & {i_{\Delta,\gamma,*}i_{\Delta,\gamma}^!\left({\prod_{\nu}}_{/\mathsf{Exp}_\gamma(X_0)}\mathsf r_{Y_{\nu}\to D_{\nu}}\right)^*\varphi_{\dHilb_{\gamma}(D)}[\vdim]}
	\arrow[from=1-1, to=1-2]
	\arrow[from=2-1, to=1-1]
	\arrow[from=2-1, to=2-2]
	\arrow[from=2-2, to=1-2]
\end{tikzcd}\]
\end{small}
Both compositions are left adjoint to
$b^*\delta_{\gamma}^!\longrightarrow
i_{\Delta,\gamma}^!
\left({\prod_{\nu}}_{/\mathsf{Exp}_\gamma(X_0)}
\mathsf r_{Y_{\nu}\to D_{\nu}}\right)^*$,
so the diagram commutes;
see \cite[\S~2]{KKPS1} for more details.

These diagrams identify \eqref{eqn-Lag-01} and \eqref{eqn-Lag-02}.
They therefore determine the same Borel-Moore homology class: 
\[
	i_{\Delta,\gamma,*}[\dHilb_{\gamma}(X_0)]^{\vir}=
	\left[{\prod_{\nu}}_{/\mathsf{Exp}_\gamma(X_0)}\mathsf r_{Y_{\nu}\to D_{\nu}}\right]^{\mathrm{Lag}}\left(\Delta_{\gamma}^{\mathrm{Lag}}\right).
\]
This proves the Theorem.
\end{proof}

\begin{nota}
The formula holds for stable pairs.
\end{nota}

\begin{nota}
The component stacks
$\dHilb_{\Delta,\gamma(\nu)}(Y_\nu|\partial Y_\nu)$
are defined over $\mathsf{Exp}_\gamma(X_0)$ with the inherited
expansion data. Their classical truncations need not be
Deligne-Mumford: automorphisms of the ambient expansion may
preserve the subscheme on one component without preserving
the subschemes on the other components.
This does not obstruct Theorem~\ref{thm-degeneration-general-curves},
whose Lagrangian correspondences are taken in derived Artin stacks.
However, these component stacks are not identified here with
the relative logarithmic Hilbert stacks defined independently
for the component pairs. See also Question \ref{Q-dege-log-rel}.
\end{nota}

\subsection{Degeneration formula for smooth pairs via Li-Wu's construction}\label{sec-Li-Wu-deg}
We now specialize to simple degenerations and use Li-Wu's construction.

\begin{nota}
In the setting of Li-Wu, we don't need the modification/root construction of cone structures
and the Li-Wu's stack of expanded pairs are compactifiable automatically.
\end{nota}

Let $X\to\bb A^1$ be a simple degeneration of local Calabi-Yau $4$-folds 
satisfies Assumption \ref{ass-compactifiable-families}, e.g.,
Lemma \ref{lem-compactifiable-local-families}, and $(\beta_t,\chi_t)\in K_{c,\leq1}^{\text{num}}(X_t)$.
Assume there exists a simple degeneration $U\to\bb A^1$ of $3$-folds such that 
\[X=\Tot(\omega_{U/\bb A^1}).\]
Then $U_0=U_+\cup_SU_-$ and $X_0=Y_+\cup_D Y_-$
such that $Y_{\pm}=\Tot(\omega_{U_{\pm}}(S))$ and
$D=\Tot(\omega_S)$. For a splitting datum 
$(P_-,P_+)\in\Lambda_{\text{spl}}^{\beta_0,\chi_0}$, we write $P_D=P_-|_D=P_+|_D$.
Then we have restriction maps
$r_{\pm}^{P_{\pm}}:\dHilb^{P_{\pm}}(Y_{\pm},D)\to\dHilb^{P_D}(D)\times\mcal A^{P_{\pm}}$,
and the following diagram as a composition of Lagrangian correspondences
related to $\mcal C_0^{\dagger ,(P_-,P_+)}$:
\begin{footnotesize}
\[\begin{tikzcd}
	{\dHilb^{(P_-,P_+)}(X/\mathbb A^1)_0^{\dagger}\cong\dHilb^{P_-}(Y_-,D)\times_{\dHilb^{P_D}(D)}\dHilb^{P_+}(Y_+,D)} & {\dHilb^{P_D}(D)\times\TT^*[-1]\mcal C_0^{\dagger ,(P_-,P_+)}} \\
	{\dHilb^{P_-}(Y_-,D)\times\dHilb^{P_+}(Y_+,D)} & {\dHilb^{P_D}(D)\times\dHilb^{P_D}(D)\times \TT^*[-1]\mcal C_0^{\dagger ,(P_-,P_+)}}
	\arrow["{r_{\Delta}^{(P_-,P_+)}}", from=1-1, to=1-2]
	\arrow["{i_{\Delta}^{(P_-,P_+)}}", from=1-1, to=2-1]
	\arrow["\square"{description}, draw=none, from=1-1, to=2-2]
	\arrow["{\Delta\times\text{id}}", from=1-2, to=2-2]
	\arrow["{r_{-}^{P_-}\times r_{+}^{P_+}}", from=2-1, to=2-2]
\end{tikzcd}\]
\end{footnotesize}
where we use $\mcal C_0^{\dagger ,(P_-,P_+)}\cong\mcal A^{P_-}\times\mcal A^{P_+}$.
The same associativity argument as in
Theorem \ref{thm-degeneration-general-curves} gives the following corollary.

\begin{coro}[Degeneration formula for smooth pairs]\label{coro-degeneration-formula-smooth-pair}
In this situation, let $D$ be a smooth divisor and consider the Li-Wu stack of 
expanded pairs. Then we have
\begin{itemize}
	\item If $b\neq0$, then 
	\[
	i_b^![\dHilb_{\beta_t,\chi_t}(X/\bb A^1)]^{\vir}
	=[\dHilb_{\beta_b,\chi_b}(X_b)]^{\vir}.
	\]
	\item If $b=0$, then
	\[
		i_0^![\dHilb_{\beta_t,\chi_t}(X/\bb A^1)]^{\vir}
		=n_*\sum_{(P_-,P_+)\in\Lambda_{\text{spl}}^{\beta_0,\chi_0}}[\dHilb^{(P_-,P_+)}(X/\bb A^1)_0^{\dagger}]^{\vir}.
	\]
    Moreover, we have
    \begin{align*}
    &i_{\Delta,*}^{(P_-,P_+)}[\dHilb^{(P_-,P_+)}(X/\bb A^1)_0^{\dagger}]^{\vir}=[r_{-}^{P_-}\times r_{+}^{P_+}]^{\text{Lag}}_{\text{coh}}
    \left(\left[\Delta^{(P_-,P_+)}_{\text{Lag}}\right]\right),
    \end{align*}
    for the product of relative $\textsf{DT}_4$-theories $[r_{-}^{P_-}\times r_{+}^{P_+}]^{\text{Lag}}_{\text{coh}}=[r_{-}^{P_-}]^{\text{Lag}}_{\text{coh}}\boxtimes[r_{+}^{P_+}]^{\text{Lag}}_{\text{coh}}$  
	of pairs $(Y_{\pm},D)$ and Lagrangian diagonal 
	$\left[\Delta^{(P_-,P_+)}_{\text{Lag}}\right]\in H_{\mathrm{crit}}\left(\dHilb^{P_D}(D)^{\times 2}\times\TT^*[-1]\mcal C_0^{\dagger ,(P_-,P_+)}\right)$.
\end{itemize}
\end{coro}
\begin{nota}
Note that $\left[\Delta^{(P_-,P_+)}_{\text{Lag}}\right]$
is given by
    \begin{align*}
    \bb Q\to\Delta_{*}\bb Q &\xrightarrow{\text{pairing}}\Delta_*\Delta^!\left(\varphi_{\dHilb^{P_D}(D)}^{\boxtimes2}\boxtimes\varphi_{\TT^*[-1]\mcal C_0^{\dagger ,(P_-,P_+)}}\right)\\ 
        &\xrightarrow{\text{can}}
        \varphi_{\dHilb^{P_D}(D)}^{\boxtimes2}\boxtimes\varphi_{\TT^*[-1]\mcal C_0^{\dagger ,(P_-,P_+)}}.
    \end{align*}
Here, the pairing map $\bb Q\to\Delta^!\left(\varphi_{\dHilb^{P_D}(D)}^{\boxtimes2}\right)$ is induced by 
the pairing map
\begin{equation*}
	\begin{aligned}
&\varphi_{\dHilb^{P_D}(D)}\otimes\varphi_{\dHilb^{P_D}(D)}\cong\varphi_{\dHilb^{P_D}(D)}\otimes\bb D\varphi_{\dHilb^{P_D}(D)}\\ 
&\cong\varphi_{\dHilb^{P_D}(D)}\otimes\scr Hom(\varphi_{\dHilb^{P_D}(D)},\bb D\bb Q_{\dHilb^{P_D}(D)})\to\bb D\bb Q_{\dHilb^{P_D}(D)},
\end{aligned}
\end{equation*}
under the identification
\begin{align*}
&\Delta^!\left(\varphi_{\dHilb^{P_D}(D)}^{\boxtimes2}\right)
\cong\bb D\Delta^*(\varphi_{\dHilb^{P_D}(D)}\boxtimes\bb D\varphi_{\dHilb^{P_D}(D)})\\ 
&\cong\bb D(\varphi_{\dHilb^{P_D}(D)}\otimes\bb D\varphi_{\dHilb^{P_D}(D)})
\cong\scr Hom(\varphi_{\dHilb^{P_D}(D)}\otimes\bb D\varphi_{\dHilb^{P_D}(D)},\bb D\bb Q).
\end{align*}
Here we use the properties of Verdier dual $\bb D$ and the canonical isomorphism $\varphi_{-}\cong\bb D\varphi_{-}$.
\end{nota}

\section{Zero-dimensional logarithmic \texorpdfstring{$\mathsf{DT}_4$}{DT4}-theory 
and degeneration formula}\label{sec-zero-dim-DT4}
This section develops the zero-dimensional logarithmic $\mathsf{DT}_4$-theory and 
proves its degeneration formula using the stack $\mathsf{Exp}_0(X|D)$ of $0$-complexes. 
Unlike the curve-counting theory, the construction is carried out entirely in Chow groups 
and does not use shifted Lagrangian classes.

\subsection{Construction of the virtual cycles}
Fix compactifiable rank-zero expansion $\mathsf{Exp}_{0,R}(X|D)$ and write
\[
\mathfrak X\hookrightarrow\mathfrak X^c
\longrightarrow\mathsf{Exp}_{0,R}(X|D)
\]
for the original universal family and its flat projective
family. By Lemma~\ref{lem-derived-hilb} and
Remark~\ref{rmk-intrinsic-derived-hilb}, applied to $0$-complexes,
we have an intrinsic derived enhancement
\[
\dHilb^R(X|D)\hookrightarrow
\mathsf{dPerf}
\bigl(\mathfrak X^c/\mathsf{Exp}_{0,R}(X|D)\bigr)_0,
\]
where $R\in K^{\mathrm{num}}_{\leq0,\mathrm{cpt}}(X)$.

\begin{prp}\label{prp-points-exact-sympl}
Let $(X,D)$ be a log Calabi-Yau pair with reduced simple 
normal crossing divisor $D$ and $\dim X=n$ and $R\in K^{\mathrm{num}}_{\leq0,\mathrm{cpt}}(X)$.
\begin{enumerate}
    \item The canonical morphism 
    \[r_R:\dHilb^R(X|D)\to\mathsf{Exp}_{0,R}(X|D)\]
    admits a $(2-n)$-shifted symplectic structure. It is exact if $n=4$.
	    \item If $n=4$ and there exists an snc pair $(Y,S)$ with $\dim Y=3$ such that $X=\mathrm{Tot}(\omega_Y(S))$ and $D=\mathrm{Tot}(\omega_S)$, then 
    \[r_R:\dHilb^R(X|D)\to\mathsf{Exp}_{0,R}(X|D)\]
    admits a canonical orientation.
\end{enumerate}
\end{prp}
\begin{proof}
Similar as those in Theorem \ref{thm-curves-shifted-lag},
we can working over $\mathfrak X,\mathfrak D$ equipped with Calabi-Yau structures.

For (1), by the construction of $\mathsf{Exp}_{0,R}(X|D)$, we have 
$\mathfrak D=\emptyset$ and hence by Lemma \ref{lem:log-crepant-MR-pair},
we have $\omega_{\mathfrak X/\Exp(X|D)}\cong\scr O$. 
By a proof similar to that of Theorem \ref{thm-curves-shifted-lag}(2),
we obtain a $(3-n)$-shifted Lagrangian structure on the restriction map
\[
\mathsf{r}_{X\to D}^R=r_R:\dHilb^R(X|D)\to\mathsf{Exp}_{0,R}(X|D)
\]
relative to 
$\mathsf{Exp}_{0,R}(X|D)$. Then this induces a 
$(2-n)$-shifted symplectic structure on 
\[
r_R:\dHilb^R(X|D)\to\mathsf{Exp}_{0,R}(X|D)
\]
as in Theorem \ref{thm-curves-shifted-lag}(1).
For $n=4$, the exactness follows from similar arguments in Proposition 
\ref{prp-family-shifted}(1).

For (2), this follows from Lemma \ref{lem:log-crepant-MR-pair} and arguments similar to those in \cite[Thm.~4.6]{CZZ24};
see also \cite[Rmk.~4.10,~Eqn.~(4.30)]{CZZ24} and the sketch of Theorem \ref{thm-curves-shifted-lag}(3).
\end{proof}

Let $\dim X=4$.
Using Proposition \ref{prp-points-exact-sympl}(1) together with \cite[Thm.~5.2.2]{park24} or classical virtual pullback theory described in \cite{park21},
we obtain the \textit{virtual class}, assuming the existence of an orientation $o$:
    \[
    [\Hilb^R(X|D)]^{\vir}_{\bb T,o}:=\sqrt{r_R^!}[\mathsf{Exp}_{0,R}(X|D)]\in \CH_*^{\bb T}(\Hilb^R(X|D)),
    \]
with the obstruction theory given by $\mathbf R\scr Hom_{\pi}(\scr I,\scr I)_0[3]$. 
It is $\bb T$-equivariant if $\bb T$ is any torus preserving the Calabi-Yau form; we call it a Calabi-Yau torus.
Finally, using this virtual class, we define the following 
zero-dimensional logarithmic $\mathsf{DT}_4$-invariant with tautological insertion.

\begin{defi}\label{def-0dim-log-DT4}
Let $(X,D)$ be a log Calabi-Yau pair with reduced simple 
normal crossing divisor $D$ and $\dim X=4$, and suppose there exists an orientation $o$
and a Calabi-Yau torus $\bb T$-action such that $\Hilb^n(X|D)^{\bb T}$
is proper for all $n$. Let $\scr L$ be a $\bb T$-equivariant
line bundle.
The tautological bundle associated to $\scr L$
is defined as
\[
\scr L^{[n]}:=\bb R q_{X|D,*}(p_{X|D}^*\scr L\otimes \scr O_{\mathfrak Z})
\]
where $q_{X|D}:\mathfrak X\times_{\mathsf{Exp}_{0,n}(X|D)}\Hilb^n(X|D)\to\Hilb^n(X|D)$
and $p_{X|D}:\mathfrak X\times_{\mathsf{Exp}_{0,n}(X|D)}\Hilb^n(X|D)\to\mathfrak X\to X$
are the natural morphisms, and $\mathfrak Z\subset\mathfrak X\times_{\mathsf{Exp}_{0,n}(X|D)}\Hilb^n(X|D)$ is the universal subscheme.

We then define the \textit{zero-dimensional logarithmic $\mathsf{DT}_4$-invariant $\DT(X|D)$ with tautological insertion}
by
\[
\DT_0(X|D)_{\scr L,m}:=1+\sum_{n>0}q^n\int_{[\Hilb^n(X|D)]^{\vir}_{\bb T}}
e^{\bb T\times\bb C^*_m}\left((\scr L^{[n]})^{\vee}\otimes e^m\right),
\]
where $\bb C^*_m$ is a trivial torus with weight $m$.
Here we identify $R=n[\scr O_{\mathrm{pt}}]\in K^{\mathrm{num}}_{\leq0,\mathrm{cpt}}(X)$.
\end{defi}

By Proposition \ref{prp-points-exact-sympl}(2), we can always define the zero-dimensional logarithmic $\mathsf{DT}_4$-invariant 
$\DT_0(X|D)_{\scr L,m}$
for any reduced snc pair $(X,D)$ such that there exists
a reduced snc pair $(Y,S)$ of $\dim Y=3$
such that $X=\mathrm{Tot}(\omega_Y(S))$ and $D=\mathrm{Tot}(\omega_S)$
with proper Calabi-Yau-torus fixed loci.

\subsection{Independence of the choice of cone structures}
Next, we will prove that the zero-dimensional logarithmic $\mathsf{DT}_4$-invariants 
in the setting of Proposition \ref{prp-points-exact-sympl}(2)
are independent of the choice of compactifiable cone structures of $P(\Sigma_{X|D})$ and its universal family.

\begin{thm}\label{thm-independent-cone-models}
Consider the spaces in Proposition \ref{prp-points-exact-sympl}(2).
For any compactifiable choices $\lambda,\mu\in\Lambda$ of the cone
structures, 
there exists $\varphi\in\Lambda$
and logarithmic modifications 
\[\begin{tikzcd}
	& {\Hilb^R(X|D)_{\varphi}} & \\
	{\Hilb^R(X|D)_{\lambda}} && {\Hilb^R(X|D)_{\mu}}
	\arrow["{\pi^{\lambda}}"', from=1-2, to=2-1]
	\arrow["{\pi^{\mu}}", from=1-2, to=2-3]
\end{tikzcd}\]
such that 
\[
\pi_*^{\lambda}[\Hilb^R(X|D)_{\varphi}]^{\vir}=[\Hilb^R(X|D)_{\lambda}]^{\vir},\quad 
\pi_*^{\mu}[\Hilb^R(X|D)_{\varphi}]^{\vir}=[\Hilb^R(X|D)_{\mu}]^{\vir}.
\]
We will denote this relation by $[\Hilb^R(X|D)_{\lambda}]^{\vir}\leftrightsquigarrow [\Hilb^R(X|D)_{\mu}]^{\vir}$.
\end{thm}
\begin{proof}
By \cite[Prop.~3.6.1]{MR24}, we choose a common refinement of the
base and universal cone complexes, allowing the necessary
finite-index changes of integral structures.
Using \cite[Rmk.~3.5.2]{MR24} and \cite[Prop.~1.7.2]{MR25},
we take this common model to be compactifiable, with smooth expansion
base.

The induced Hilbert comparison morphisms are proper, as in
\cite[Thm.~5.2.1]{MR24}. These maps between stacks of expansions have pure degree
one: they are proper modifications, including generalized root
constructions, and are isomorphisms on the common dense unexpanded
locus. The orientations are the compatible total-space orientations.
Consider 
\[\begin{tikzcd}
	{X'} & X \\
	{Y'} & Y
	\arrow["f", from=1-1, to=1-2]
	\arrow["{p'}", from=1-1, to=2-1]
	\arrow["\square"{description}, draw=none, from=1-1, to=2-2]
	\arrow["p", from=1-2, to=2-2]
	\arrow["g", from=2-1, to=2-2]
\end{tikzcd}\]
under the following assumptions as in \cite{HW23}:
\begin{itemize}
	\item $X,X'$ are separated Deligne-Mumford stacks of finite type,
and $Y,Y'$ are Artin stacks of the same pure dimension such that $g$ is Deligne-Mumford.
	\item $g$ is proper and birational, hence has pure degree one
along $p$ in the sense of \cite[Def.~2.3]{HW23}, and
$p$ and $p'$ have compatible oriented symmetric obstruction theories
$\bb E,\bb E'=f^*\bb E$ with orientations and isotropy conditions;
see \cite{park21} for the notation and definitions.
\end{itemize}
Then we claim that 
\[f_*\circ\sqrt{(p')^!}[Y']=\sqrt{p^!}[Y].\]
Indeed, 
consider the diagram
\[\begin{tikzcd}
	{\mathfrak C_{X'/Y'}} & {\mathfrak C_{X/Y}} \\
	{\mathfrak Q(\bb E')} & {\mathfrak Q(\bb E)} \\
	{X'} & X,
	\arrow["{h_{\mathfrak C}}", from=1-1, to=1-2]
	\arrow["{a'}", hook, from=1-1, to=2-1]
	\arrow["a", hook, from=1-2, to=2-2]
	\arrow["h", from=2-1, to=2-2]
	\arrow["\square"{description}, draw=none, from=2-1, to=3-2]
	\arrow["{0_{\mathfrak Q(\bb E')}}", from=3-1, to=2-1]
	\arrow["f", from=3-1, to=3-2]
	\arrow["{0_{\mathfrak Q(\bb E)}}"', from=3-2, to=2-2]
\end{tikzcd}\]
where $\mathfrak Q(\bb E)=\mathfrak q_{\bb E}^{-1}(0)\subset\mathfrak C_{\bb E}$,
the existence of $a,a'$ follows from the isotropy condition, and
the fact that the lower square is cartesian follows from \cite[Prop.~1.7(2)]{park21}.
Note that the square-root virtual pullback of \cite[App.~A]{park21} commutes with the proper pushforward
in this diagram\footnote{Choose a projective surjection $a_P:P\to X'$ with $P$ a
quasi-projective scheme by Chow's lemma for separated DM stacks.
Since $X$ is separated and $fa_P$ is proper, $fa_P$ is projective.
The induced projective surjection of quadratic cone stacks gives
a surjection on Chow groups by \cite[Lem.~A.1]{park21}.
Lift a cycle on $\mathfrak Q(\bb E')$ to the quadratic cone over $P$.
The projective-pushforward compatibility of the square-root Gysin map
for $a_P$ and for $fa_P$ then proves the required compatibility for $f$,
independently of the lift.}.
Therefore, we have 
\begin{align*}
    f_*\circ\sqrt{(p')^!}[Y']&=f_*\circ\sqrt{0_{\mathfrak Q(\bb E')}^!}\circ a'_*[\mathfrak C_{X'/Y'}]\\ 
    &=\sqrt{0_{\mathfrak Q(\bb E)}^!}\circ h_*\circ a'_*[\mathfrak C_{X'/Y'}]\\
    &=\sqrt{0_{\mathfrak Q(\bb E)}^!}\circ a_*\circ h_{\mathfrak C,*}[\mathfrak C_{X'/Y'}]\\ 
    &=\sqrt{0_{\mathfrak Q(\bb E)}^!}\circ a_*[\mathfrak C_{X/Y}]=\sqrt{p^!}[Y].
\end{align*}
Here we use the fact that the map of relative intrinsic normal cones
$\mathfrak C_{X'/Y'}\to\mathfrak C_{X/Y}$ is of degree one over generic points; see \cite[Prop.~2.5]{HW23}
for the proof.
\end{proof}
\begin{nota}\label{rmk-cone-independence}
The same proper degree-one pushforward argument, combined with
Conjecture \ref{conj-Joyce-conj}(4), applies to the
non-relative curve Hilbert stacks in the family theory.
Here properness of the comparison morphisms is relative over the
coarser Hilbert stack.
\end{nota}

\begin{coro}\label{coro-0dimlogDT4-independence-cone-models}
The zero-dimensional logarithmic $\mathsf{DT}_4$-invariant $\DT(X|D)$ with tautological insertion
is independent of the choice of compactifiable cone structures of $P(\Sigma_{X|D})$ and its universal family.
\end{coro}
\begin{proof}
In this case, we have the pullback diagram of universal substacks:
\[\begin{tikzcd}
	{\mathfrak Z^{\varphi}} & {\mathfrak Z^{\lambda}} \\
	{\Hilb^R(X|D)_{\varphi}} & {\Hilb^R(X|D)_{\lambda},}
	\arrow[from=1-1, to=1-2]
	\arrow[from=1-1, to=2-1]
	\arrow["\square"{description}, draw=none, from=1-1, to=2-2]
	\arrow[from=1-2, to=2-2]
	\arrow[from=2-1, to=2-2]
\end{tikzcd}\]
where the vertical morphisms are flat, and the same holds for $\mu$. 
Hence, by
Tor-independence, the base-change theorem and the pullback diagram above,
the Euler classes
of tautological bundles are compatible with respect to pullbacks.
Combining these with Theorem \ref{thm-independent-cone-models} shows that the zero-dimensional logarithmic $\mathsf{DT}_4$-invariants 
with tautological insertions
are independent of the choice of cone structures of $P(\Sigma_{X|D})$.
\end{proof}

\subsection{Degeneration formula of zero-dimensional logarithmic \texorpdfstring{$\mathsf{DT}_4$}{DT4}-invariants}\label{sec-zero-dim-deg}
Unlike the general case, for the logarithmic Hilbert stacks of points, we consider 
a simple normal crossing degeneration of \textit{log Calabi-Yau pairs}, rather than Calabi-Yau 
varieties. In the point-counting theory, we only need to consider 
the $0$-complexes.
\subsubsection{Logarithmic Hilbert stacks for log Calabi-Yau families}
Let  \[\pi:\mcal X\to\bb A^1\] 
be a simple normal crossing degeneration described in 
Definition \ref{defi-snc-degeneration}.
Let $\mcal D'\subset\mcal X$ be an anti-canonical snc divisor (that is, $K_{\mcal X}+\mcal D'\sim0$) such that $\pi|_{\mcal D'}$ is flat, $\pi^{-1}(t)\cap \mcal D'$ is again snc
for any $0\neq t\in\bb A^1$ and $\mcal D'+\pi^{-1}(0)\subset\mcal X$ is also snc.

Let $\mcal D:=X_0\cup\mcal D'$, where $X_0:=\pi^{-1}(0)$. Then $(\mathcal X,\mathcal D)$ is again
a log Calabi-Yau snc pair since $X_t\sim0$ for any $t\in\bb A^1$.
Then this induces 
\[
\pi^{\mathrm{trop}}:\Sigma_{\mcal X|\mcal D}\to\bb R_{\geq0}.
\]
Let $\Delta_{X_0}=(\pi^{\mathrm{trop}})^{-1}(1)$. Then 
\[
X_0=\bigcup_{\nu\in\mathrm{vertex}(\Delta_{X_0})}X_{0,\nu}
\]
gives the decomposition of irreducible components.

Assuming Assumption \ref{ass-compactifiable-families}, e.g., Lemma \ref{lem-compactifiable-local-families}, 
for the vertical
$0$-complexes, that is, the $0$-complexes contracted by $\pi^{\mathrm{trop}}$,
including the horizontal boundary.
We choose a compactifiable whole-family cone structure of
$P(\Sigma_{\mcal X|\mcal D}/\bb R_{\geq0})$ with smooth expansion
base and simple normal crossings boundary.
Then we have a canonical map 
\[
P(\Sigma_{\mcal X|\mcal D}/\bb R_{\geq0})\to\bb R_{\geq0}.
\]
Fix a topological datum $R_t\in K^{\mathrm{num}}_{\leq0,\mathrm{cpt}}(X_t)$. Then
we have the \textit{stack of expansions of degeneration}:
\[
\mathsf{Exp}_{0,R_t}((\mcal X|\mcal D)/\bb A^1)\to\bb A^1
\]
of relative dimension zero, and the family version of logarithmic (derived-) Hilbert stacks 
\[
\Hilb^{R_t}((\mcal X|\mcal D)/\bb A^1)\subset
\dHilb^{R_t}((\mcal X|\mcal D)/\bb A^1)\to\mathsf{Exp}_{0,R_t}((\mcal X|\mcal D)/\bb A^1)\to\bb A^1
\]
such that $\Hilb^{R_t}((\mcal X|\mcal D)/\bb A^1)$ is separated of finite type
and is proper if $\pi:\mcal X\to\bb A^1$ is proper.

The arguments in Proposition \ref{prp-family-shifted} give the following results.

\begin{prp}\label{prp-0-dim-shifted-main}
Let $\dim\mcal X=n+1$ and assume Assumption \ref{ass-compactifiable-families}, e.g.,
Lemma \ref{lem-compactifiable-local-families}.
\begin{enumerate}
    \item Then \[\varpi^{R_t}:
\dHilb^{R_t}((\mcal X|\mcal D)/\bb A^1)\to\mathsf{Exp}_{0,R_t}((\mcal X|\mcal D)/\bb A^1)
\]
admits
a $(2-n)$-shifted symplectic structure.
Moreover, it is exact if $n=4$.
\item If $n=4$, and there exists an snc degeneration $(\mcal Y,\mcal S)\to\bb A^1$ of relative dimension $3$ such that $\mcal X=\mathrm{Tot}(\omega_{\mcal Y/\bb A^1}(\mcal S))$ and $\mcal D'=\mathrm{Tot}(\omega_{\mcal S/\bb A^1})$, then 
    \[\varpi^{R_t}:
\dHilb^{R_t}((\mcal X|\mcal D)/\bb A^1)\to\mathsf{Exp}_{0,R_t}((\mcal X|\mcal D)/\bb A^1)
\]
    admits a canonical orientation.
\end{enumerate}
\end{prp}
In the case of Proposition \ref{prp-0-dim-shifted-main},
if $\mcal Y_0=\bigcup_{\nu}Y_{0,\nu}$ and
$S_{0,\nu}:=\mcal S|_{Y_{0,\nu}}$, then the adjunction formula gives
\[
X_{0,\nu}\cong \mathrm{Tot}\bigl(\omega_{Y_{0,\nu}}(S_{0,\nu}+\partial Y_{0,\nu})\bigr),
\qquad
\partial X_{0,\nu}\cong
\mathrm{Tot}\bigl(\omega_{S_{0,\nu}\cup\partial Y_{0,\nu}}\bigr)
\]
componentwise. Hence the central fibre also satisfies the total-space
orientation hypothesis. Moreover, we have the square-root virtual pullback
\[
\sqrt{\varpi^{R_t,!}}:\CH_*(\mathsf{Exp}_{0,R_t}((\mcal X|\mcal D)/\bb A^1))
\to \CH_*\left(\dHilb^{R_t}((\mcal X|\mcal D)/\bb A^1)\right),
\]
which induce the virtual class
\[
[\Hilb^{R_t}((\mcal X|\mcal D)/\bb A^1)]^{\vir}:=
\sqrt{\varpi^{R_t,!}}\left[\mathsf{Exp}_{0,R_t}((\mcal X|\mcal D)/\bb A^1)\right].
\]
If there exists a Calabi-Yau torus $\bb T$ acting on the space,
then the virtual class is $\bb T$-equivariant.

\subsubsection{Zero-dimensional logarithmic degeneration formula for \texorpdfstring{$\mathsf{DT}_4$}{DT4}-invariants}
We continue with Assumption \ref{ass-compactifiable-families},
$\dim X_t=4$, and compatible orientations as above.
For any cone structure $\alpha$, we will denote
\[
\mathsf{Exp}_\alpha:=
\mathsf{Exp}_{0,R_t}((\mcal X|\mcal D)/\bb A^1)_\alpha.
\]
For any smooth compactifiable choice $\varphi$ with snc boundary, we have
\[
\dHilb^{R_t}((\mcal X|\mcal D)/\bb A^1)_\varphi
\xrightarrow{\varpi^{R_t}_\varphi}
\mathsf{Exp}_\varphi\longrightarrow\bb A^1.
\]
Let $\mathsf{Exp}_{\varphi,0}$ be the scheme-theoretic central fibre.
Taking the derived fibre at $0$ and pulling back to its reduced
components gives the Cartesian diagram
\begin{equation}\label{Fig-normalization-Hilb}
\begin{tikzcd}
	{\coprod_\delta\dHilb^{R_t}((\mcal X|\mcal D)/\bb A^1)_{\delta,\varphi}} & {\coprod_\delta\mathsf{Exp}_{\delta,\varphi}} \\
	{\dHilb^{R_t}((\mcal X|\mcal D)/\bb A^1)_{0,\varphi}^{\dagger}} & {\mathsf{Exp}_{\varphi,0}^{\dagger}} \\
	{\dHilb^{R_t}((\mcal X|\mcal D)/\bb A^1)_{0,\varphi}} & {\mathsf{Exp}_{\varphi,0}}
	\arrow[from=1-1, to=1-2]
	\arrow["{=}", from=1-1, to=2-1]
	\arrow["\square"{description}, draw=none, from=1-1, to=2-2]
	\arrow["{=}", from=1-2, to=2-2]
	\arrow[from=2-1, to=2-2]
	\arrow["{n_\varphi}", from=2-1, to=3-1]
	\arrow["\square"{description}, draw=none, from=2-1, to=3-2]
	\arrow["{n_\varphi}", from=2-2, to=3-2]
	\arrow[from=3-1, to=3-2]
\end{tikzcd}
\end{equation}
Here $\delta$ ranges over the rigid Hilbert $0$-complexes, and
$\mathsf{Exp}_{\delta,\varphi}$ denotes the corresponding
smooth \textit{reduced} central component.
For the fixed degeneration parameter $t$, put
$m_{\delta,\varphi}:=\operatorname{ord}_{\mathsf{Exp}_{\delta,\varphi}}(t)$ again,
we have
\[
[\mathsf{Exp}_{\varphi,0}]
=n_{\varphi,*}\sum_\delta
m_{\delta,\varphi}[\mathsf{Exp}_{\delta,\varphi}].
\]
To state our degeneration formula, we need some
comparison lemma.

\begin{lem}\label{lem-zero-dimensional-starting-comparison}
There are smooth compatible cone structures $\lambda,\lambda_\nu$
such that
$\mathsf{Exp}_{\lambda,0}=\bigcup_\gamma\mathsf{Exp}_{\gamma,\lambda}$,
$\operatorname{div}_{\mathsf{Exp}_\lambda}(t)
=\sum_\gamma\mathsf{Exp}_{\gamma,\lambda}$
and projective logarithmic modifications of pure degree one
\[
c_\gamma:\mathsf{Exp}_{\gamma,\lambda}\longrightarrow
\prod_{\nu\text{ from }\gamma}
\mathsf{Exp}_{\gamma(\nu)}
(X_{0,\nu}|\partial X_{0,\nu})_{\lambda_\nu}.
\]
Here $\gamma$ distributes $R_t$ among the original components.
\end{lem}

\begin{proof}
Let $n$ be the length of $R_t$. If $n=0$, take the expansion base
$\bb A^1$ and the component expansion bases to be points.
From now on we assume $n>0$.

\smallskip\noindent\textit{Step 1. Construction of $\lambda$.}
We order the original central components as $X_{0,1},\ldots,X_{0,r}$
and label the horizontal boundary components by $j$.
A Hilbert $0$-complex of total length $n$ consists of finitely many
vertices $p_1,\ldots,p_k$ with positive integral decorations
$w_1,\ldots,w_k$ and $\sum_iw_i=n$.

For each $p_i$, choose a target cone $\sigma_i$ of
$\Sigma_{\mcal X}$ containing it. We first describe the tuples
$(p_1,\ldots,p_k)\in\sigma_1\times\cdots\times\sigma_k$
whose vertices have the same height $h$ as we only consider vertical one. Let
$x_{i,a}\geq0$ and $z_{i,j}\geq0$ be
the coordinates of $p_i$ along the rays corresponding to
$X_{0,a}$ and the $j$-th horizontal divisor, respectively,
then $\sum_{a=1}^r x_{i,a}=h$.

In order to give a standard refinement, we consider 
the partial sum $b_{i,a}:=\sum_{d\leq a}x_{i,d}$
for any $1\leq a<r$. After collect them for any $i,a$ and take an order, we get 
$0\leq b_1\leq\cdots\leq b_N\leq h$. This forms a subdivision and 
each resulting parameter cone is a face of
\[
\{0\leq b_1\leq\cdots\leq b_N\leq h\}
\times\prod_j\{0\leq z_{1,j}\leq\cdots\leq z_{k,j}\}.
\]
This is a integral free-orthant coordinates by taking successive differences.
The cones are therefore unimodular and
every primitive ray of positive height has $h=1$, $b_{i,a}\in\{0,1\}$ and $z_{i,j}=0$.
Equivalently, for each $i$, precisely one $x_{i,a}$ is $1$.
Thus every vertex lies at an original vertical vertex on such a ray.
Note that the structure we constructed 
is not destroyed when forgetting the labels $p_i$,
since the loci of type $p_i=p_l$ is given by equality of their corresponding
coordinates on an actual common target face and these loci are unions
of faces of our subdivision. Hence we merge coincident vertices and add their decorations,
and identify the resulting cones under relabelling.

For universal cone map, pick each parameter cone $\tau$, the vertices $p_i$ define integral
linear maps to their target cones. Taking graphs give the universal
$0$-complex over $\tau$. Each graph maps integrally isomorphically
to $\tau$, and two graphs meet exactly on the faces where their
vertices coincide. Thus the universal cone map is combinatorially
flat with reduced fibres.

Now we give the smooth stack $\mathsf{Exp}_\lambda$, after base-change along
$\bb A^1\to[\bb A^1/\bb G_m]$, with reduced central fibre.
Its boundary is simple normal crossings globally as well:
each rigid Hilbert $0$-complexes $\gamma$
is supported at the original
vertices $\nu$; its decoration $\gamma(\nu)$ has length
$n_\nu=\sum_{i:\,p_i\text{ specializes to }\nu}w_i$.

On a positive-height primitive ray, write $p_i=e_{\nu(i)}$. Then
\[
\sum_{\nu=1}^r\nu n_\nu
=\sum_iw_i\nu(i)
=rn-\sum_iw_i\sum_{a=1}^{r-1}b_{i,a}.
\]
Along the ordered positive-height rays, the last sum changes strictly, since all $w_i>0$. 
Thus these rays cannot be identified under relabelling.
For a fixed horizontal component $j$, the primitive horizontal rays
have $z_{i,j}\in\{0,1\}$ and distinct values of $\sum_iw_i z_{i,j}$.
Hence distinct rays of one cone are not identified in the quotient
and each $\mathsf{Exp}_{\gamma,\lambda}$ is smooth.

\smallskip\noindent\textit{Step 2. Construction of $c_\gamma$.}
Fix $\gamma$. For a labelled cone $\tau$ of Step 1 containing
a ray $\rho$ representing $\gamma$, let $I_\nu$ be the labels
specializing to $\nu$ and consider quotient $\pi:\tau\to\overline\tau=\tau/\mathbb R\rho$.
Then the original coordinates $x_{i,a}$ ($a\ne\nu$)
and $z_{i,j}$ vanish on $\rho$ and descend to
$\overline x_{i,a},\overline z_{i,j}$ on $\overline\tau$ if $i\in I_\nu$.
Consider the descending of $b$-coordinates:
\[
\begin{aligned}
u_{i,a}&:=\sum_{d\leq a}\overline x_{i,d},
& b_{i,a}&=\pi^*u_{i,a} &&(a<\nu),\\
v_{i,a}&:=\sum_{d>a}\overline x_{i,d},
& b_{i,a}&=h-\pi^*v_{i,a} &&(\nu\leq a<r).
\end{aligned}
\]
Thus the star retains separate $u$-orders, reversed $v$-orders,
and horizontal orders for each $j$; the remaining relations
$\pi^*(u_{i,a}+v_{l,b})\leq h$ disappear for $h\gg0$.
Ordering the $u$-functions, the $v$-functions and, for each $j$,
the $\overline z_{i,j}$ with labels $i\in I_\nu$ defines
$\lambda_\nu$, which is the smooth star in Step 1 for length
$n_\nu$ where all vertices specialize to $\nu$.

Now one can construct 
\[
c_\gamma:\mathsf{Exp}_{\gamma,\lambda}\longrightarrow
\prod_{\nu\text{ from }\gamma}
\mathsf{Exp}_{\gamma(\nu)}
(X_{0,\nu}|\partial X_{0,\nu})_{\lambda_\nu}.
\]
Indeed, for $i\in I_\nu$, write $\sigma_i$ for the target cone of $p_i$ and $e_\nu$
for the primitive generator of its vertical ray $\nu$.
The induced vertex maps $\overline p_i$ define
\[
\overline c_{\gamma,\tau}=(\overline p_i)_i:
\overline\tau\longrightarrow
\prod_\nu\prod_{i\in I_\nu}(\sigma_i/\mathbb R e_\nu).
\]
Using the subdivision $\lambda_\nu$, this image
lies in one product cone.
When $\tau$ varies, every point of this product lifts via
$x_{i,\nu}=h-\sum_{a\ne\nu}\overline x_{i,a}$ for common $h\gg0$,
uniquely modulo $\rho$ by trivial reason.
This gives an integral subdivision along the different $u,v$ and $\overline{z}$.
Then the morphism $c_{\gamma}$ via Artin-fan construction.

Finally, the complete subdivision above preserves lattices
and induces proper birational toric maps $c_\gamma$.
Moreover it is projective since we have the following function
\[
\psi_\gamma
=\sum_{\nu<\nu'}\sum_{\substack{i\in I_\nu\\l\in I_{\nu'}}}w_iw_l
\left(
\sum_{\substack{1\leq a<\nu\\1\leq b<\nu'}}|u_{i,a}-u_{l,b}|
+\sum_{\substack{\nu\leq a<r\\\nu'\leq b<r}}|v_{i,a}-v_{l,b}|
+\sum_j|\overline z_{i,j}-\overline z_{l,j}|
\right).
\]
Easy to see that $\psi_\gamma$ is a relatively strictly convex integral support function
and is invariant under relabelling.
On a collision face within $I_\nu$, the corresponding coordinates agree, and
$w_iw_l+w_{i'}w_l=(w_i+w_{i'})w_l$.
Thus its restrictions agree with the function after merging the coincident vertices. 
Consequently, $\psi_\gamma$ descends to the unlabelled cone space and proves projectivity.
This proves the lemma.
\end{proof}

\begin{nota}
By Step 2, the original rank-zero universal family over
$\mathsf{Exp}_{\gamma,\lambda}$ is isomorphic to the disjoint union
of the pullbacks of the original universal families over
$\mathsf{Exp}_{\gamma(\nu)}
(X_{0,\nu}|\partial X_{0,\nu})_{\lambda_\nu}$ along $c_\gamma$
and natural projections.
\end{nota}

By Lemma \ref{lem-snc-compactification}, we can choose smooth compactifiable
refinements $\mu_\nu$ of the component models. Put
\[
\mathsf{Exp}_{\gamma,\mu}:=
\prod_{\nu\text{ from }\gamma}
\mathsf{Exp}_{\gamma(\nu)}
(X_{0,\nu}|\partial X_{0,\nu})_{\mu_\nu}.
\]

\begin{lem}\label{lem-zero-dimensional-Cartier-comparison}
There is a smooth compactifiable common refinement $\varphi$ of
$\lambda$ and any previously fixed family cone structure
and proper morphisms $j_\gamma:W_\gamma\hookrightarrow\mathsf{Exp}_{\varphi,0}$
and $\rho_\gamma:W_\gamma\longrightarrow\mathsf{Exp}_{\gamma,\mu}$ such that
\begin{equation}\label{eqn-Cartier-product-cycle-comparison}
\rho_{\gamma,*}[W_\gamma]=[\mathsf{Exp}_{\gamma,\mu}],
\qquad
\sum_\gamma j_{\gamma,*}[W_\gamma]=[\mathsf{Exp}_{\varphi,0}].
\end{equation}
Here $W_\gamma:=
\mathsf{Exp}_\varphi
\times_{\mathsf{Exp}_\lambda}\mathsf{Exp}_{\gamma,\lambda}$.
\end{lem}
\begin{proof}
Fix $\gamma$. Choose a smooth refinement
\[\begin{tikzcd}
	{\widetilde{\mathsf{Exp}}_{\gamma,\lambda}} && \\
	& {\mathsf{Exp}_{\gamma,\lambda}\times^{\mathrm{log}}_{ \prod_{\nu\text{ from }\gamma} \mathsf{Exp}_{\gamma(\nu)} (X_{0,\nu}|\partial X_{0,\nu})_{\lambda_\nu}} \mathsf{Exp}_{\gamma,\mu}} & {\mathsf{Exp}_{\gamma,\mu}} \\
	& {\mathsf{Exp}_{\gamma,\lambda}} & {\prod_{\nu\text{ from }\gamma} \mathsf{Exp}_{\gamma(\nu)} (X_{0,\nu}|\partial X_{0,\nu})_{\lambda_\nu}}
	\arrow[from=1-1, to=2-2]
	\arrow["{b_{\gamma}}"{description}, bend left = 10, from=1-1, to=2-3]
	\arrow["{a_{\gamma}}"{description}, bend right = 19, from=1-1, to=3-2]
	\arrow[from=2-2, to=2-3]
	\arrow[from=2-2, to=3-2]
	\arrow["{\square^{\mathrm{log}}}"{description}, draw=none, from=2-2, to=3-3]
	\arrow[from=2-3, to=3-3]
	\arrow["{c_{\gamma}}", from=3-2, to=3-3]
\end{tikzcd}\]
Both $a_{\gamma},b_{\gamma}$ are proper of degree one. We will modify the whole expansion
base $\mathsf{Exp}_\lambda$ so that the full inverse image of
$\mathsf{Exp}_{\gamma,\lambda}$ is smooth, reduced, and dominates
$\widetilde{\mathsf{Exp}}_{\gamma,\lambda}$ over
$\mathsf{Exp}_{\gamma,\lambda}$.

\smallskip\noindent\textit{Step 1. Extension to the whole expansion base.}
Let $E_1,\ldots,E_s$ be all boundary divisors of
$\mathsf{Exp}_\lambda$ other than $\mathsf{Exp}_{\gamma,\lambda}$,
on the finite-type locus where we consider.
For this step, equip $\mathsf{Exp}_\lambda$ with boundary
$\sum_jE_j$, and $\mathsf{Exp}_{\gamma,\lambda}$ with its restriction.
Then each cone of both stacks can embedded in $\bb R_{\geq0}^s$.
For each such cone $\tau$ on $\mathsf{Exp}_{\gamma,\lambda}$, write
$\widetilde\tau\longrightarrow\tau$ for the 
subdivision induced by $a_\gamma$ and $N_\tau=\bb Z^s\cap\operatorname{span}_{\bb R}\tau$
for the the original lattice.
Note that different strata with the same coordinate face are retained separately.

Extend the equations of all codimension-one walls in these
subdivisions $\widetilde\tau$ to integral linear equations
$\ell_1=0,\ldots,\ell_N=0$ on $\bb R^s$.
Then all these hyperplanes cut the orthant give the subdivision
defined by $\sum_{j=1}^N|\ell_j|$.
Thus, over each coordinate face, it refines every subdivision
$\widetilde\tau$, which is compatible to the restrictions to smaller faces.
This extends at the level of cones.

For integral structures, given any cone $\kappa\subset\widetilde\tau$, let
$\widetilde N_\kappa$ be its prescribed lattice, viewed in $N_\tau$
via $a_\gamma$. Choose one integer $M>0$ satisfying
\[
M\bigl(N_\tau\cap\operatorname{span}_{\bb R}\kappa\bigr)
\subset\widetilde N_\kappa
\quad\text{for every }\tau\text{ and }\kappa\in\widetilde\tau.
\]
Let $\Sigma^{(\gamma)}$ be a smooth refinement of the hyperplane
subdivision obtained by a sequence of star subdivisions in the
lattice $M\bb Z^s$. It is projective over the hyperplane subdivision.
We then have maps of integral cone complexes
\[
\Sigma^{(\gamma)}|_\tau\longrightarrow
\widetilde\tau\longrightarrow\tau
\quad\text{for every remaining-boundary cone }\tau
\text{ of }\mathsf{Exp}_{\gamma,\lambda}.
\]

Define $\mathsf{Exp}_\lambda^{(\gamma)}$ by taking $M$-th roots
along $E_1,\ldots,E_s$ in $\mathsf{Exp}_\lambda$, followed by the
toroidal modification given by $\Sigma^{(\gamma)}$.
By the previous analysis, this give an extension of whole base
for the fixed $\gamma$, which may not be compactifiable for now.
Set
\[
\mathsf{Exp}_{\gamma,\lambda}^{(\gamma)}
:=
\mathsf{Exp}_\lambda^{(\gamma)}
\times_{\mathsf{Exp}_\lambda}\mathsf{Exp}_{\gamma,\lambda}.
\]
The cone maps above induce
\[
\mathsf{Exp}_{\gamma,\lambda}^{(\gamma)}
\xrightarrow{q_\gamma}
\widetilde{\mathsf{Exp}}_{\gamma,\lambda}
\xrightarrow{b_\gamma}
\mathsf{Exp}_{\gamma,\mu}.
\]
Both maps are proper of degree one.
In local snc coordinates, the roots and subdivisions affect only the 
remaining-boundary directions, leaving a defining coordinate of 
$\mathsf{Exp}_{\gamma,\lambda}$ unchanged, so the smooth refinement yields a smooth ambient 
stack $\mathsf{Exp}_\lambda^{(\gamma)}$  in which the full inverse image 
$\mathsf{Exp}_{\gamma,\lambda}^{(\gamma)}$ remains a smooth reduced Cartier divisor.

\smallskip\noindent\textit{Step 2. Compactifiable refinement.}
Using Step 1 for each $\gamma$, one can take 
a smooth common refinement
$\varphi$ of all $\mathsf{Exp}_\lambda^{(\gamma)}$
and a fixed compactifiable cone structure we have chosen at the beginning
by \cite[Prop.~3.6.1]{MR24}.
Denote the proper morphism
\[
\pi^{(\gamma)}:\mathsf{Exp}_\varphi
\longrightarrow\mathsf{Exp}_\lambda^{(\gamma)},
\]
which is degree-one. Pullback the original and full universal
families of the previous model to obtain
$(\mathfrak X_\varphi,\mathfrak X_\varphi^c)$,
where $\mathfrak X_\varphi$ is the required rank-zero family.
Hence $\varphi$ is compactifiable.

\smallskip\noindent\textit{Step 3. Finish the proof.}
By definition, $W_\gamma$ fits into the Cartesian diagram
\[\begin{tikzcd}
	&& {\mathsf{Exp}_{\gamma,\mu}} & {\widetilde{\mathsf{Exp}}_{\gamma,\lambda}} \\
	{W_\gamma} && {\mathsf{Exp}_{\gamma,\lambda}^{(\gamma)}} & {\mathsf{Exp}_{\gamma,\lambda}} \\
	{\mathsf{Exp}_\varphi} && {\mathsf{Exp}_\lambda^{(\gamma)}} & {\mathsf{Exp}_\lambda}
	\arrow["{b_{\gamma}}"', from=1-4, to=1-3]
	\arrow["{a_{\gamma}}", from=1-4, to=2-4]
	\arrow["{\rho_\gamma}"{description}, dashed, from=2-1, to=1-3]
	\arrow["{\pi^{(\gamma)}|_{W_{\gamma}}}"{description}, from=2-1, to=2-3]
	\arrow[hook, from=2-1, to=3-1]
	\arrow["\square"{description}, draw=none, from=2-1, to=3-3]
	\arrow["{q_{\gamma}}"{description}, from=2-3, to=1-4]
	\arrow[from=2-3, to=2-4]
	\arrow["{i_{\gamma}}", hook, from=2-3, to=3-3]
	\arrow["\square"{description}, draw=none, from=2-3, to=3-4]
	\arrow[hook, from=2-4, to=3-4]
	\arrow["{\pi^{(\gamma)}}"{description}, from=3-1, to=3-3]
	\arrow[from=3-3, to=3-4]
\end{tikzcd}\]
where the proper map $\rho_\gamma$ is defined by composition.
Here we note that $W_\gamma$ need not be reduced.

Now we have
\begin{align*}
	(\pi^{(\gamma)}|_{W_\gamma})_*[W_\gamma]&=(\pi^{(\gamma)}|_{W_\gamma})_*i_\gamma^!
[\mathsf{Exp}_\varphi]=i_\gamma^!\pi^{(\gamma)}_*[\mathsf{Exp}_\varphi]\\
&=i_\gamma^![\mathsf{Exp}_\lambda^{(\gamma)}]=[\mathsf{Exp}_{\gamma,\lambda}^{(\gamma)}].
\end{align*}
Here the pullback of a local equation of
$\mathsf{Exp}_{\gamma,\lambda}^{(\gamma)}$ is a non-zero-divisor:
$\mathsf{Exp}_\varphi$ is smooth, and no irreducible component is mapped
into this divisor by the modification $\pi^{(\gamma)}$.
Thus the square is Tor-independent and
$i_\gamma^![\mathsf{Exp}_\varphi]=[W_\gamma]$, giving the first equality.
The third equality uses the degree one of $\pi^{(\gamma)}$.
By Step 1, $b_\gamma\circ q_\gamma$ also has degree one, so
\[
\rho_{\gamma,*}[W_\gamma]
=(b_\gamma\circ q_\gamma)_*
[\mathsf{Exp}_{\gamma,\lambda}^{(\gamma)}]
=[\mathsf{Exp}_{\gamma,\mu}].
\]
Finally, Lemma \ref{lem-zero-dimensional-starting-comparison} gives
an equality of Cartier divisors
\[
\sum_\gamma W_\gamma
=(\mathsf{Exp}_{\varphi}\to\mathsf{Exp}_{\lambda})^*
\sum_\gamma\mathsf{Exp}_{\gamma,\lambda}
=\operatorname{div}_{\mathsf{Exp}_\varphi}(t).
\]
Taking cycles gives
$\sum_\gamma j_{\gamma,*}[W_\gamma]=[\mathsf{Exp}_{\varphi,0}]$.
All inverse images here are scheme-theoretic, so their
multiplicities are retained.
\end{proof}

From now on, choose $\varphi$ as in
Lemma \ref{lem-zero-dimensional-Cartier-comparison}.

\begin{prp}\label{prp-decompo-derived}
For the choices in Lemma \ref{lem-zero-dimensional-Cartier-comparison},
define $\dHilb_{W_\gamma}:=\dHilb^{R_t}((\mcal X|\mcal D)/\bb A^1)_\varphi \times_{\mathsf{Exp}_\varphi}W_\gamma$.
Then we have the derived pullback diagram:
\[\begin{tikzcd}
	{\dHilb_{W_\gamma}} & {\prod_{\nu\text{ from }\gamma} \dHilb^{\gamma(\nu)}(X_{0,\nu}|\partial X_{0,\nu})_{\mu_\nu}} \\
	{W_\gamma} & {\mathsf{Exp}_{\gamma,\mu}}
	\arrow["{\rho_{\gamma,H}}", from=1-1, to=1-2]
	\arrow["{\varpi^{R_t}_{W_\gamma}}", from=1-1, to=2-1]
	\arrow["\square"{description}, draw=none, from=1-1, to=2-2]
	\arrow["{\varpi_{\gamma,\mu}}", from=1-2, to=2-2]
	\arrow["{\rho_\gamma}", from=2-1, to=2-2]
\end{tikzcd}\]
It identifies the relative shifted symplectic structures and their
symmetric obstruction theories.
\end{prp}
\begin{proof}
The right vertical map is the product of the component structure morphisms.
Let $\mathfrak X_{\nu,W_\gamma}\subset\mathfrak X_{\nu,W_\gamma}^c$
be the pullbacks of the original and full component families to $W_\gamma$.
By \cite[Thm.~5.2.1, proof]{MR24}, a change of cone structures
only adds tube vertices to the same complex, which do not exist for $0$-complexes.
Thus Lemma \ref{lem-zero-dimensional-starting-comparison}
and Lemma \ref{lem-zero-dimensional-Cartier-comparison}
gives
\[
\mathfrak X_\varphi\times_{\mathsf{Exp}_\varphi}W_\gamma
\cong\mathfrak X_\lambda\times_{\mathsf{Exp}_\lambda}W_\gamma
\cong\coprod_{\nu\text{ from }\gamma}\mathfrak X_{\nu,W_\gamma}.
\]
They therefore hold scheme-theoretically, including when $W_\gamma$
is not reduced.

By Remark \ref{rmk-intrinsic-derived-hilb}, we can consider
the flat projective containing family
$\coprod_\nu\mathfrak X_{\nu,W_\gamma}^c$. The desired 
derived pullback follows from the similar proof of Proposition \ref{prp-fiberprod-gluing-general}
Finally, the supported Chern characters and traces split over the components,
so the shifted symplectic structures agree.
\end{proof}

\begin{thm}[Zero-dimensional degeneration formula for cycles]
\label{coro-zero-degeneration}
For the choices and orientations above, we have:
\begin{enumerate}
\item If $0\ne b\in\bb A^1$, then
\[
i_b^![\Hilb^{R_t}((\mcal X|\mcal D)/\bb A^1)_\varphi]^{\vir}
=[\Hilb^{R_b}(X_b|\partial X_b)_\varphi]^{\vir},
\]
where the fibre uses the induced choice.
\item If $b=0$, then
\begin{equation}\label{eqn-zero-dimensional-weighted-specialization}
\begin{aligned}
&i_0^![\Hilb^{R_t}((\mcal X|\mcal D)/\bb A^1)_\varphi]^{\vir}\\
&=n_{\varphi,*}\sum_\delta m_{\delta,\varphi}
[\Hilb^{R_t}((\mcal X|\mcal D)/\bb A^1)_{\delta,\varphi}]^{\vir}\\
&=\sum_\gamma j_{\gamma,H,*}[\Hilb_{W_\gamma}]^{\vir},
\end{aligned}
\end{equation}
where $j_{\gamma,H}:\Hilb_{W_\gamma}\longrightarrow
\Hilb^{R_t}((\mcal X|\mcal D)/\bb A^1)_{0,\varphi}$
is the base change of $j_\gamma$, satisfies
\[
\rho_{\gamma,H,*}[\Hilb_{W_\gamma}]^{\vir}=\bigboxtimes_{\nu\text{ from }\gamma}
[\Hilb^{\gamma(\nu)}
(X_{0,\nu}|\partial X_{0,\nu})_{\mu_\nu}]^{\vir}.
\]
Here $[\Hilb_{W_\gamma}]^{\vir}:=
\sqrt{\varpi^{R_t,!}_{W_\gamma}}[W_\gamma]$.
\end{enumerate}
\end{thm}
\begin{proof}
The first statement is trivial. 
For the second, the first equality of Eqn. (\ref{eqn-zero-dimensional-weighted-specialization}) follows from
Fig.~(\ref{Fig-normalization-Hilb}) and the second equality
follows by applying the family virtual pullback to the second
equality in Eqn. \eqref{eqn-Cartier-product-cycle-comparison}:
\[
\sqrt{\varpi_{0,\varphi}^{R_t,!}}[\mathsf{Exp}_{\varphi,0}]
=\sum_\gamma j_{\gamma,H,*}
\sqrt{\varpi_{W_\gamma}^{R_t,!}}[W_\gamma].
\]
Now we focus on the final equality.
By Lemma \ref{lem-zero-dimensional-Cartier-comparison} and Proposition \ref{prp-decompo-derived}, 
we have
\begin{align*}
\rho_{\gamma,H,*}[\Hilb_{W_\gamma}]^{\vir}
&=\rho_{\gamma,H,*}\sqrt{\varpi_{W_\gamma}^{R_t,!}}[W_\gamma]=\sqrt{\varpi_{\gamma,\mu}^{\,!}}\,\rho_{\gamma,*}[W_\gamma]\\
&=\sqrt{\varpi_{\gamma,\mu}^{\,!}}[\mathsf{Exp}_{\gamma,\mu}]
=\bigboxtimes_{\nu\text{ from }\gamma}
[\Hilb^{\gamma(\nu)}
(X_{0,\nu}|\partial X_{0,\nu})_{\mu_\nu}]^{\vir}.
\end{align*}
This proves the results.
\end{proof}

\begin{nota}\label{rmk-zero-dimensional-multiplicities}
The coefficients $m_{\delta,\varphi}$ need not be one and
the scheme-theoretic cycles $[W_\gamma]$ and stack pushforwards
retain these multiplicities.

The formulas also hold equivariantly for a the case that there exists some Calabi-Yau torus
acting on these spaces with proper fixed loci.
\end{nota}

\section{Applications to zero-dimensional logarithmic \texorpdfstring{$\mathsf{DT}_4$}{DT4}-invariants}\label{sec-applications}
The zero-dimensional $\mathsf{DT}_4$-invariants for $\bb C^4$ were fully computed
in \cite{CZZ24}, with the K-theoretic analogue treated in \cite{KR25}.
The computation extends to all toric Calabi-Yau $4$-folds under suitable
compatibility assumptions on orientations (see \cite[Thm.~A]{Liu25} for the proof
of the case $Y\times\bb C$). To treat non-toric cases,
we use the degeneration formula and cobordism groups
as in \cite{Li06,LP09}.

The objects in the cobordism group in \cite{LP09}, however, do not preserve the Calabi-Yau structure,
even after taking canonical bundles. We therefore keep track of
anti-canonical divisors in the degeneration. This approach was developed in \cite{Guz25}
to compute zero-dimensional log $\mathsf{DT}_3$-invariants.
In this section, we adapt this method to verify the conjectural formula for the 
zero-dimensional logarithmic $\mathsf{DT}_4$-invariants of the local surface pairs 
considered below.

\subsection{Logarithmic cobordism groups}
In \cite{Guz25}, Guzman developed the so-called logarithmic cobordism group,
generalizing the results in \cite{LP09}. Here we give a brief review and a slight modification.

\begin{defi}[Double point relation]\label{def-log-cobordism}
Let $\pi:\mcal X\to (U,0)$ be a flat family of $n$-dimensional varieties over a smooth pointed curve $U$, such that
\begin{enumerate}
    \item $\pi^{-1}(0)=X_1\cup_EX_2$ is the union of two smooth varieties that intersect transversally along a smooth divisor 
    $E=X_1\cap X_2$;
    \item $\mcal X$ is smooth and the morphism $\pi$ is smooth away from $0\in U$.
\end{enumerate}
Furthermore, suppose we are given a tuple of line bundles $(\mcal L_1,...,\mcal L_r)$ and a family of snc divisors, i.e. an 
snc divisor $\mcal D\subset \mcal X$ such that
\begin{itemize}
    \item $\pi|_{\mcal D}$ is flat over $U$;
    \item The intersection $\mcal D\cap\pi^{-1}(t)$ is an snc divisor in $\pi^{-1}(t)$ for all 
	    $t\in U$. In particular, we have an snc divisor $D = X\cap\mcal D$ on the general fibre $X$ 
	    and snc divisors $D_1=X_1\cap\mcal D$ and $D_2 =X_2\cap\mcal D$ on the special fibre.
    \item $\mcal D$ meets $E$ transversally.
\end{itemize}
Then we say this is a \textit{logarithmic double point degeneration}
\[(X,D;\mcal L_1|_X,...,\mcal L_r|_X)\rightsquigarrow^{\pi} 
(X_1,D_1+E;\mcal L_1|_{X_1},...,\mcal L_r|_{X_1})\cup_E (X_2,D_2+E;\mcal L_1|_{X_2},...,\mcal L_r|_{X_2}).
\]
The \textit{double point relation} defined by such data is
\[
[X,D;\mcal L_1|_X,...,\mcal L_r|_X]=[X_1,D_1+E;\mcal L_1|_{X_1},...,\mcal L_r|_{X_1}]+[X_2,D_2+E;\mcal L_1|_{X_2},...,\mcal L_r|_{X_2}].
\]
\end{defi}

\begin{defi}[Logarithmic cobordism group]
Let $\mcal M_{n,1^r}$ be the free abelian group spanned by pairs 
$[X, D;\scr L_1,...,\scr L_r]$ of smooth projective varieties of dimension $n$ and snc divisors $D\subset X$
with line bundles $\scr L_i\in\mathrm{Pic}(X)$. Let 
$\sim$ be the relation spanned by the
double point relations arising from logarithmic double point degenerations
of $\pi:\mcal X\to (U,0)$ for some Zariski open subcurve $U\subset\bb P^1$ containing $0$.
We define the \textit{logarithmic cobordism group} to be
\[
\omega_{n,1^r}^{\log}:=\mcal M_{n,1^r}/\sim.
\]
\end{defi}
\begin{nota}
We use a slightly different definition from \cite{Guz25}; this convention is compatible with 
the ordinary algebraic cobordism groups 
defined in \cite{LP09,LP10} since we can always extend and take resolutions.
Hence all the results, especially for \cite[Prop.~3.2,~Prop.~3.3]{Guz25}, remain valid; see \cite[\S~6]{Hu24} for more details.
\end{nota}

\begin{nota}
In \cite[\S~6]{Guz25}, logarithmic cobordism with 
logarithmic modifications of targets is considered to simplify the calculation of zero-dimensional 
logarithmic DT-invariants of $3$-folds. By arguments
similar to those in \cite[Thm.~1.1.1]{AW18}, the invariants do not change after logarithmic
modifications on the targets $(X,D)$; see also \cite[Rmk.~5.2.2]{MR24}. 
In the present paper, we leave the behavior of $\mathsf{DT}_4$-virtual classes
under target logarithmic modifications open.
\end{nota}

Moreover, Guzman computed the generators
of $\omega_3^{\log}\otimes_{\bb Z}\bb Q$ in \cite[Thm.~3.8,~Rmk.~3.9]{Guz25}.
Similarly, we can prove the 
following result.
\begin{prp}\label{prp-span-omega-2-11-log}
The group $\omega_{2,1^1}^{\log}\otimes_{\bb Z}\bb Q$ is spanned as a vector space by 
$(\bb P^2,\emptyset;\scr O)$, $(\bb P^2,\emptyset;\scr O(1))$, $(\bb P^1\times\bb P^1,\emptyset;\scr O)$,
$(\bb P^1\times\bb P^1,\emptyset;\scr O(1,0))$, $(\bb P^1\times\bb P^1,\bb P^1;\scr O)$,
$(\bb P^1\times\bb P^1,\bb P^1;\scr O(1,0))$ and $(\bb F_1,\bb P^1;\scr O)$.
\end{prp}
\begin{proof}
There are two morphisms to $\omega_{2,1^1}^{\log}$ described in \cite[Prop.~3.2,~3.3]{Guz25}:
\[\begin{tikzcd}
	{\omega_{2,1^1}} & {\omega_{2,1^1}^{\log}} & {\omega_{1,1^2}^{\log}} & {\omega_{1,1^2}}
	\arrow["a", from=1-1, to=1-2]
	\arrow["b"', from=1-3, to=1-2]
	\arrow["c"', from=1-4, to=1-3]
\end{tikzcd}\]
where $a$ (and similarly for $c$) is given by $(X;\scr L)\mapsto(X,\emptyset;\scr L)$
and $b$ is given by $(C,D;\scr L_1,\scr L_2)\mapsto(\bb P_C(\scr L_2\oplus\scr O),C\cup\eta^{-1}(D);\eta^*\scr L_1)$
for $\eta:\bb P_C(\scr L_2\oplus\scr O)\to C$.

\smallskip\noindent\textit{Step 1. $\omega_{2,1^1}^{\log}\otimes_{\bb Z}\bb Q$ 
is spanned by smooth pairs and surfaces without boundaries.}

We first claim that it suffices to prove that this is true for any 
$(\bb P_C(\scr L_2\oplus\scr O),C\cup\eta^{-1}(D);\eta^*\scr L_1)$
for $\eta:\bb P_C(\scr L_2\oplus\scr O)\to C$ arising from $(C,D;\scr L_1,\scr L_2)\in\omega_{1,1^2}^{\log}$.
Indeed, for any $(S,\overline C;\scr L)\in\omega_{2,1^1}^{\log}$, we argue by induction on $k$, where 
$\overline C=C_1\cup\cdots\cup C_k$. The case $k=1$ is trivial.
We consider 
\[
\mathrm{Bl}_{C_k\times\{0\}}(S\times\bb P^1)\to\bb P^1
\]
and the pullback of $\scr L$. This induces the relation 
\[
(S,C_1\cup\cdots\cup C_k;\scr L)=(S,C_1\cup\cdots\cup C_{k-1};\scr L)
-(\bb P_{C_k}(N_{C_k/S}\oplus\scr O),C_k\cup\eta^*Z;\eta^*\scr L)
\]
for $Z=(C_1\cap C_k)\cup\cdots\cup(C_{k-1}\cap C_k)$. Then the claim follows.

To prove step 1, we will use the claim and consider $(\bb P_C(\scr L_2\oplus\scr O),C\cup\eta^{-1}(D);\eta^*\scr L_1)$
for $\eta:\bb P_C(\scr L_2\oplus\scr O)\to C$ arising from $(C,D;\scr L_1,\scr L_2)\in\omega_{1,1^2}^{\log}$.
Let $D=\{p_1,...,p_k\}\subset C$ and argue by induction on $k$. The case $k=0$ is trivial by 
construction. For general $k$, we can consider 
$\mathrm{Bl}_{(p_k,0)}(C\times\bb P^1)\to\bb P^1$ again. Then we have 
\[
(C,p_1+\cdots+p_{k-1};\scr L_1,\scr L_2)=(C,p_1+\cdots+p_{k};\scr L_1,\scr L_2)
+(\bb P^1,p_k;\scr O,\scr O).
\]
It remains to consider $(\bb P^1,p_k;\scr O,\scr O)$.
Its image is equal to $(\bb P^1\times\bb P^1;\bb P^1\times\{0\}+\{0\}\times\bb P^1;\scr O)$.
From deformation to the normal cone along $\{0\}\subset\bb P^1$, one gets 
$(\bb P^1,\emptyset)=2(\bb P^1,\{0\})$.
Then by product, we have 
\[(\bb P^1\times\bb P^1;\bb P^1\times\{0\}+\{0\}\times\bb P^1)=\frac{1}{4}(\bb P^1\times\bb P^1;\emptyset).\]
Applying the natural map $\omega_{2}^{\log}\to\omega_{2,1^1}^{\log}$ defined by 
$(S,C)\mapsto(S,C;\scr O)$, we get the result.

\smallskip\noindent\textit{Step 2. Finish the proof.}

Consider a smooth pair $(S,C;\scr L)$ and $\mathrm{Bl}_{C\times\{0\}}(S\times\bb P^1)\to\bb P^1$.
Then we have a relation 
\[
(S,C;\scr L)=(S,\emptyset;\scr L)-(\bb P_C(N_{C/S}\oplus\scr O),C;\eta^*\scr L)
\]
and the final term follows from $(C,\emptyset;\scr L,N_{C/S})\in\omega_{1,1^2}$.
Then $\omega_{2,1^1}^{\log}\otimes_{\bb Z}\bb Q$ 
is spanned by $\mathrm{Im}(a)\otimes_{\bb Z}\bb Q$ and $\mathrm{Im}(b\circ c)\otimes_{\bb Z}\bb Q$.
By the results in \cite{LP10} or \cite{Tze12}, we know that $\omega_{2,1^1}\otimes_{\bb Z}\bb Q$ is spanned 
by $(\bb P^2;\scr O)$, $(\bb P^2;\scr O(1))$, $(\bb P^1\times\bb P^1;\scr O)$ and $(\bb P^1\times\bb P^1;\scr O(1,0))$
and $\omega_{1,1^2}\otimes_{\bb Z}\bb Q$ is spanned 
by $(\bb P^1;\scr O(1),\scr O)$, $(\bb P^1;\scr O,\scr O(1))$ and $(\bb P^1;\scr O,\scr O)$.
Then our main result follows.
\end{proof}

\subsection{Zero-dimensional logarithmic \texorpdfstring{$\mathsf{DT}_4$}{DT4}-invariants
for local surfaces}
In this section, we will consider an snc pair $(S,C)$ with $S$ a smooth projective surface
with snc divisor $C$ and compute 
$\DT_0(\mathrm{Tot}(\omega_S(C))\times\bb C,\mathrm{Tot}(\omega_C)\times\bb C)_{\scr L,m}$
for the Calabi-Yau torus $\bb T=\{tt'=1\}\subset(\bb C^*)^2$ acting on the fibres of this rank two vector bundle.
Let $\hbar$ be its equivariant parameter, so the two fibre weights are $\hbar$ and $-\hbar$.
In this case $\Hilb^n(\mathrm{Tot}(\omega_S(C))\times\bb C,\mathrm{Tot}(\omega_C)\times\bb C)^{\bb T}$
is proper, and the same holds for the family version; this follows from \cite[Thm.~6.3.1]{MR25}
and \cite[Rmk.~2.23]{CZZ24}.

\begin{prp}\label{prp-homo-omega2-DT}
There exists a group homomorphism
\begin{align*}
&(\omega^{\log}_{2,1^1},+)\to (1+q\mathrm{Frac}(\bb Q[\hbar,m])[[q]],\cdot),\\ 
&(S,C;\scr L)\mapsto\DT_0(\mathrm{Tot}(\omega_S(C))\times\bb C,\mathrm{Tot}(\omega_C)\times\bb C)_{\scr L,m},
\end{align*}
where we still denote the pulled-back line bundle by $\scr L$ by an abuse of notation.
Here we choose the orientations using the presentations $\mathrm{Tot}(\omega_S(C))\times\bb C=\Tot(\omega_{\mathrm{Tot}(\omega_S(C))}(\mathrm{Tot}(\omega_C)))$
and $\mathrm{Tot}(\omega_C)\times\bb C=\Tot(\omega_{\mathrm{Tot}(\omega_C)})$,
which makes the construction compatible.
\end{prp}
\begin{proof}
Let 
\[
[S,C;\mcal L|_S]=[S_-,D_-+E;\mcal L|_{S_-}]+[S_+,D_++E;\mcal L|_{S_+}]
\]
in $\omega^{\log}_{2,1^1}$ connected by a family $\mcal S\to (U,0)$ for some $U\subset\bb P^1$ 
and $\mcal L\in\mathrm{Pic}(\mcal S)$.
If $U=\bb P^1$, after removing an irrelevant fibre, we may work over an open subcurve 
$U=\bb A^1$. If $U\neq\bb P^1$, then $U\subset\bb A^1$.
Let $(X,D;L)=(\mathrm{Tot}(\omega_S(C))\times\bb C,\mathrm{Tot}(\omega_C)\times\bb C;\mcal L|_S)$
and $(Y_{\pm},T_{\pm};L_{\pm})=(\mathrm{Tot}(\omega_{S_{\pm}}(D_{\pm}+E))\times\bb C,\mathrm{Tot}(\omega_{D_{\pm}+E})\times\bb C;\mcal L|_{S_{\pm}})$.
Then they are connected by the family $(\mcal X,\mcal D)\to U$
arising as a total space over $\mcal S$.

Then Theorem \ref{coro-zero-degeneration}, Remark \ref{rmk-zariski-open-replace}
and Lemma \ref{lem-compactifiable-local-families} give
\[
\begin{aligned}
i_b^![\Hilb^n((\mcal X|\mcal D)/U)_\varphi]^{\vir}
&=[\Hilb^n(X|D)_\varphi]^{\vir},
\qquad b\ne0,\\
i_0^![\Hilb^n((\mcal X|\mcal D)/U)_\varphi]^{\vir}
&=\sum_{\gamma=(n_-,n_+)}
j_{\gamma,H,*}[\Hilb_{W_\gamma}]^{\vir},
\end{aligned}
\]
where $n_-+n_+=n$, and
\[
\rho_{\gamma,H,*}[\Hilb_{W_\gamma}]^{\vir}
=
[\Hilb^{n_-}(Y_-|T_-)_{\mu_-}]^{\vir}
\boxtimes
[\Hilb^{n_+}(Y_+|T_+)_{\mu_+}]^{\vir}.
\]
Let $\mcal L_\varphi^{[n]}$ be the tautological vector bundle. Finite flat base change and the disjoint-union
decomposition of the universal subscheme therefore give
\[
j_{\gamma,H}^*
\bigl(\mcal L_\varphi^{[n]}|_{\Hilb_{\varphi,0}}\bigr)
\cong
\rho_{\gamma,H}^*
\bigl(\scr L_-^{[n_-]}\boxplus\scr L_+^{[n_+]}\bigr).
\]
Apply the projection formula to the left-hand maps of the
correspondences with $e^{\bb T\times\bb C_m^*}
\left((\mcal L_\varphi^{[n]})^\vee\otimes e^m\right)$
and combining Theorem \ref{thm-independent-cone-models} and
Corollary \ref{coro-0dimlogDT4-independence-cone-models}, we have
\begin{align*}
\int_{[\Hilb^n(X|D)]^{\vir}_{\bb T}}
 e^{\bb T\times\bb C_m^*}
 \left((\scr L^{[n]})^\vee\otimes e^m\right)&=\sum_{n_-+n_+=n}
\left(
\int_{[\Hilb^{n_-}(Y_-|T_-)]^{\vir}_{\bb T}}
 e^{\bb T\times\bb C_m^*}
 \left((\scr L_-^{[n_-]})^\vee\otimes e^m\right)
\right)\\
&\cdot\left(
\int_{[\Hilb^{n_+}(Y_+|T_+)]^{\vir}_{\bb T}}
 e^{\bb T\times\bb C_m^*}
 \left((\scr L_+^{[n_+]})^\vee\otimes e^m\right)
\right).
\end{align*}
Summing over $n$ therefore yields
\[
\DT_0(X|D)_{\scr L,m}
=
\DT_0(Y_-|T_-)_{\scr L_-,m}
\cdot
\DT_0(Y_+|T_+)_{\scr L_+,m}.
\]
This proves the result.
\end{proof}

By Propositions \ref{prp-span-omega-2-11-log} and \ref{prp-homo-omega2-DT}, to compute the zero-dimensional 
logarithmic invariants of total spaces of (relative-) canonical bundles
with tautological insertion, we only need to compute those for total spaces 
of toric projective smooth pairs and toric projective surfaces with empty boundaries.

\begin{prp}[{\cite[Thm.~A]{Liu25}}]\label{prp-toric-empty-boundaries}
Let $X=Y\times\bb C$, where $Y$ is any toric Calabi-Yau $3$-fold, and let $\scr L$ be any $\bb T'$-equivariant line bundle on $X$.
Then we have 
\[
\mathsf{DT}_0(X|\emptyset)_{\scr L,m}=M(-q)^{\int_Xc_1^{\bb T'\times\bb C^*_m}(\scr L\otimes e^{-m})c_3^{\bb T'}(T_X)},
\]
where $\bb T'=\{t_1t_2t_3t_4=1\}\subset(\bb C^*)^4$ with the canonical action.
\end{prp}

\begin{prp}\label{prp-toric-smooth-pair-surface-total}
Let $(S,C)$ be a smooth toric surface pair. Let $(X=\Tot(\omega_S(C))\times\bb C,D=\Tot(\omega_C)\times\bb C)$
and let $\scr L$ be any $\bb T'$-equivariant line bundle on $X$.
Then we have 
\[
\mathsf{DT}_0(X|D)_{\scr L,m}=M(-q)^{\int_Xc_1^{\bb T'\times\bb C^*_m}(\scr L\otimes e^{-m})c_3^{\bb T'}(T_X(-\log D))},
\]
where $\bb T'=\{t_1t_2t_3t_4=1\}\subset(\bb C^*)^4$ with the canonical action.
\end{prp}
\begin{proof}
Let $\{U_a\cong\bb C^4\}_{a}$ be the maximal $(\bb C^*)^4$-invariant open subsets
corresponding to the $(\bb C^*)^4$-fixed points $\{p_a\}$ in $X$.
For $p_a\in D$, we choose $\lambda_1^a$ as the normal weight of $D$ and let $-\lambda_2^a,-\lambda_3^a,-\lambda_4^a$ be the weights 
tangent to $D$.
Our orientation agrees with the one used in \cite[Prop.~3.5]{Liu25}. Therefore, the localization and orientation arguments of \cite[Prop.~3.6, Thm.~3.9]{Liu25} give 
\[
\mathsf{DT}_0(X|D)_{\scr L,m}=\prod_{p_a\notin D}\mathsf{DT}_0(U_a)_{\scr L|_{U_a},m}\cdot 
\prod_{p_a\in D}W_{\infty,a},
\]
where $W_{\infty,a}$ is the rubber invariant as in \cite[\S~6]{CZZ24}. 
By Proposition \ref{prp-toric-empty-boundaries} and \cite[Cor.~6.10]{CZZ24}, using
$e(E^\vee\otimes e^m)=(-1)^{\mathrm{rk}E}e(E\otimes e^{-m})$ to convert the tautological insertion in the latter, we have
\begin{align*}
\mathsf{DT}_0(X|D)_{\scr L,m}&=\prod_{p_a\notin D}M(-q)^{\int_{U_a}c_1^{\bb T'\times\bb C^*_m}(\scr L|_{U_a}\otimes e^{-m})c_3^{\bb T'}(T_{U_a})}\cdot 
\prod_{p_a\in D}M(-q)^{c_1^{\bb T'\times\bb C^*_m}(\scr L|_{p_a}\otimes e^{-m})/\lambda_1^a}\\ 
&=M(-q)^{\int_Xc_1^{\bb T'\times\bb C^*_m}(\scr L\otimes e^{-m})c_3^{\bb T'}(T_X(-\log D))},
\end{align*}
where the last formula follows from Atiyah-Bott localization formula and the fact that 
the weights of $T_X(-\log D)$ at $p_a$ for $p_a\in D$ are $0,-\lambda_2^a,-\lambda_3^a,-\lambda_4^a$;
see also \cite[Exam.~6.14]{CZZ24} for the local description.
\end{proof}

\begin{nota}\label{rmk-equivariant-prp2}
Assume that $S$ is projective.
Inside the Calabi-Yau torus $\bb T'$, the canonical lift of the surface torus and the
fibrewise Calabi-Yau torus are, respectively,
\[
\widetilde{\bb T}_S
=
\{(a,b,(ab)^{-1},1):a,b\in\bb C^*\},\quad\text{and}\quad
\bb T_{\mathrm f}
=
\{(1,1,\lambda,\lambda^{-1}):\lambda\in\bb C^*\}.
\]
Consequently, we have an isomorphism $\widetilde{\bb T}_S\times\bb T_{\mathrm f}
\xrightarrow{\ \cong\ }\bb T'$ given by $(a,b;\lambda)
\longmapsto\bigl(a,b,(ab)^{-1}\lambda,\lambda^{-1}\bigr)$.
Let $u_1,u_2$ denote the additive equivariant parameters of
$\widetilde{\bb T}_S$, and let $\hbar$ denote that of
$\bb T_{\mathrm f}$. Thus restriction from $\bb T'$ to
$\bb T_{\mathrm f}$ means taking the limit $u_1,u_2\to0$.

Let $\scr L=\eta^*\scr L'$ for $\scr L'\in\operatorname{Pic}(S)$.
Choose any $\bb T_S$-linearization of $\scr L'$ and let
$\bb T_{\mathrm f}$ act trivially on $\scr L$. Then we claim that the
$\bb T'$-equivariant formulae in Proposition
\ref{prp-toric-empty-boundaries} and
\ref{prp-toric-smooth-pair-surface-total} admit well-defined
restrictions to $\bb T_{\mathrm f}$. Here we only consider the latter case.

Indeed, one may apply the localization theorem for
$\bb T_{\mathrm f}$ in $\bb T'$-equivariant Chow theory.
Note that $\bb T_{\mathrm f}$-fixed locus $S$ is proper.
Writing $i:S\hookrightarrow X$ for the zero section and
\[
\alpha=
c_1^{\bb T'\times\bb C_m^*}(\scr L\otimes e^{-m})
c_3^{\bb T'}(T_X(-\log D)),
\]
the $\bb T'$-equivariant zero-section formula gives
\[
\int_X\alpha
=
\int_S\frac{i^*\alpha}{e^{\bb T'}(N_{S/X})}
=\int_S\frac{i^*\alpha}{\bigl(c_1^{\widetilde{\bb T}_S}(\omega_S(C))+\hbar\bigr)(-\hbar)}.
\]
Note that since
$i^*\alpha$ is an honest ${\bb T'}$-equivariant Chow class, it has no
pole in $u_1,u_2$.
Moreover, $\lim_{u_1,u_2\to0}e^{\bb T'}(N_{S/X})=\bigl(c_1(\omega_S(C))+\hbar\bigr)(-\hbar)$
which is invertible in
$\CH^*(S)_{\bb Q}[\hbar^{\pm1}]$.
Consequently,
\[
\mathsf{DT}^{\bb T_{\mathrm f}}_0(X|D)_{\scr L,m}=\lim_{u_1,u_2\to0}
\mathsf{DT}^{\bb T'}_0(X|D)_{\scr L,m}
=
M(-q)^{
\int_X
c_1^{\bb T_{\mathrm f}\times\bb C_m^*}
(\scr L\otimes e^{-m})
c_3^{\bb T_{\mathrm f}}(T_X(-\log D))
}.
\]
The restriction is independent of the chosen
$\bb T_S$-linearization of $\scr L'$.

\end{nota}

Similar to \cite[Thm.~4.1]{Guz25}, we have the following result.

\begin{lem}\label{lem-c1L-TX-logD}
There is a morphism 
\begin{align*}
\omega_{2,1^1}^{\log}&\to\mathrm{Frac}(\bb Q[\hbar,m]),\\
(S,C;\scr L)&\mapsto \int_Xc_1^{\bb T\times\bb C^*_m}(\scr L\otimes e^{-m})c_{3}^{\bb T}(T_X(-\log D)),
\end{align*}
for $X=\Tot(\omega_S(C))\times\bb C$ and $D=\Tot(\omega_C)\times\bb C$.
Here $\int_X$ is the torus equivariant pushforward.
\end{lem}
\begin{proof}
Let $(\mcal S,\mcal C;\scr L)\longrightarrow (U,0)\subset\bb P^1$
be a degeneration as in $\omega_{2,1^1}^{\log}$. Let $S$ be a
general fibre and let $\mcal S_0=A\cup_E B$
be the central fibre. Write $C_A=\mcal C\cap A, C_B=\mcal C\cap B$.
Thus the induced double point relation is
\[
(S,C;\scr L|_S)
=
(A,C_A+E;\scr L|_A)
+
(B,C_B+E;\scr L|_B).
\]
Consider the associated family of total spaces
\[
\mcal X
:=
\Tot(\omega_{\mcal S/U}(\mcal C))\times \bb C
=
\Tot(\omega_{\mcal S/U}(\mcal C)\oplus \scr O_{\mcal S})
\longrightarrow U
\]
and the divisor
$\mcal D
:=
\Tot(\omega_{\mcal C/U})\times \bb C
\subset \mcal X$.
Let $X=\mcal X_{\zeta}$ be a smooth fibre over a closed point $\zeta\in U\setminus\{0\}$ and
$
\mcal X_0=Y_-\cup_F Y_+,
$
where
\[
Y_-=\Tot(\omega_A(C_A+E))\times \bb C,
\quad
Y_+=\Tot(\omega_B(C_B+E))\times \bb C,
\]
and $
F=Y_-\cap Y_+
$
is the total space over the double locus $E$. Put
$
D_-=\mcal D\cap Y_-,
D_+=\mcal D\cap Y_+,
$
and
$
\scr L_-=\scr L|_{Y_-},
\scr L_+=\scr L|_{Y_+}.
$
By adjunction, the pairs appearing on the central fibre are precisely
\[
(Y_-,D_-+F;\scr L_-),
\qquad
(Y_+,D_++F;\scr L_+).
\]

We now put on $\mcal X$ the divisorial logarithmic structure induced by
$\mcal D+\mcal X_0$
and put on $U$ the divisorial logarithmic structure induced by
$0\in U$. The morphism
\[
(\mcal X,\mcal D+\mcal X_0)\longrightarrow (U,0)
\]
is logarithmically smooth (the family $\mcal S\to U$ is a log smooth
degeneration by definition of $\omega_{2,1^1}^{\log}$, and taking
$\Tot(\omega_{\mcal S/U}(\mcal C))\times\bb C$ preserves log smoothness).
Consider the relative logarithmic tangent bundle
\[
\mcal T
:=
T_{\mcal X/U}(-\log(\mcal D+\mcal X_0)/-\log\{0\}).
\]
It is a $\bb T\times\bb C^*_m$-equivariant vector bundle, and its
restrictions to the fibres are 
\[
\mcal T|_X
\simeq
T_X(-\log D),\quad
\mcal T|_{Y_-}
\simeq
T_{Y_-}(-\log(D_-+F)),
\qquad
\mcal T|_{Y_+}
\simeq
T_{Y_+}(-\log(D_++F)).
\]
Since $\pi:\mcal X\to U$ is flat, the fibre classes satisfy
\[
[X]=[\mcal X_{\zeta}]=[\mcal X_0]=[Y_-]+[Y_+].
\]
Let $j_\zeta:\{\zeta\}\hookrightarrow U$ and
$j_0:\{0\}\hookrightarrow U$ be the inclusions.
Then compatibility of refined Gysin pullback with the relative
localized pushforward gives
\[
\begin{aligned}
&
\int_X
c_1^{\bb T\times\bb C^*_m}(\scr L|_X\otimes e^{-m})\,
c_3^{\bb T}\bigl(T_X(-\log D)\bigr)\\
&=
\int_X
c_1^{\bb T\times\bb C^*_m}(\scr L\otimes e^{-m})\,
c_3^{\bb T}(\mcal T|_X)=
j_{\zeta}^!\pi_*^{\bb T\times\bb C_m^*}\left(
c_1^{\bb T\times\bb C^*_m}(\scr L\otimes e^{-m})\,
c_3^{\bb T}(\mcal T)\cap [\mcal X]\right)\\
&=
j_{0}^!\pi_*^{\bb T\times\bb C_m^*}\left(
c_1^{\bb T\times\bb C^*_m}(\scr L\otimes e^{-m})\,
c_3^{\bb T}(\mcal T)\cap [\mcal X]\right)\\
&=
\int_{Y_-}
c_1^{\bb T\times\bb C^*_m}(\scr L_-\otimes e^{-m})\,
c_3^{\bb T}(\mcal T|_{Y_-})+
\int_{Y_+}
c_1^{\bb T\times\bb C^*_m}(\scr L_+\otimes e^{-m})\,
c_3^{\bb T}(\mcal T|_{Y_+})\\
&=
\int_{Y_-}
c_1^{\bb T\times\bb C^*_m}(\scr L_-\otimes e^{-m})\,
c_3^{\bb T}\bigl(T_{Y_-}(-\log(D_-+F))\bigr)\\ 
&\qquad+
\int_{Y_+}
c_1^{\bb T\times\bb C^*_m}(\scr L_+\otimes e^{-m})\,
c_3^{\bb T}\bigl(T_{Y_+}(-\log(D_++F))\bigr).
\end{aligned}
\]
Here the localized pushforward $\pi_*^{\bb T\times\bb C_m^*}$ is defined since
$\mcal X^{\bb T}=\mcal S$ is proper over $U$.
This proves the lemma.
\end{proof}

\begin{thm}\label{thm-total-surface-log-DT4}
Consider an snc pair $(S,C)$, where $S$ is a smooth projective surface
and $C$ is an snc divisor. Let $X=\mathrm{Tot}(\omega_S(C))\times\bb C$
and $D=\mathrm{Tot}(\omega_C)\times\bb C$, and 
consider the Calabi-Yau torus $\bb T=\{(t_1,t_1^{-1})\}\subset(\bb C^*)^2_{t_1,t_2}$ acting on the fibres of this rank two vector bundle.
Let $\scr L=\eta^*\scr L'$ for $\scr L'\in\mathrm{Pic}(S)$ and $\eta:\mathrm{Tot}(\omega_S(C))\times\bb C\to S$,
with its canonical $\bb T$-equivariant structure.
Then there exists a choice of orientation such that
\[
\mathsf{DT}_0(X|D)_{\scr L,m}=M(-q)^{\int_Xc_1^{\bb T\times\bb C^*_m}(\scr L\otimes e^{-m})c_3^{\bb T}(T_X(-\log D))},
\]
where $M(q)=\prod_{n\geq1}(1-q^n)^{-n}$ is the MacMahon function.
\end{thm}
\begin{proof}
By Proposition \ref{prp-span-omega-2-11-log}, we have 
\[
\lambda(S,C;\scr L)=\sum_{i=1}^ra_i(S_i,C_i;\scr L_i)
\]
for some integers $\lambda>0$ and $a_i\in\bb Z_{\neq0}$, where $(S_i,C_i;\scr L_i)$ are cases in Proposition
\ref{prp-span-omega-2-11-log}. Let 
$X_i=\mathrm{Tot}(\omega_{S_i}(C_i))\times\bb C$
and $D_i=\mathrm{Tot}(\omega_{C_i})\times\bb C$.
By Remark \ref{rmk-equivariant-prp2}, Proposition \ref{prp-toric-empty-boundaries}, Proposition \ref{prp-toric-smooth-pair-surface-total}
and Lemma \ref{lem-c1L-TX-logD}
we have 
\begin{align*}
\mathsf{DT}_0(X|D)_{\scr L,m}^{\lambda}&=\prod_{i=1}^r\mathsf{DT}_0(X_i|D_i)_{\scr L_i,m}^{a_i}\\ 
&=\prod_{i=1}^rM(-q)^{a_i\int_{X_i}c_1^{\bb T\times\bb C^*_m}(\scr L_i\otimes e^{-m})c_3^{\bb T}(T_{X_i}(-\log D_i))}\\
&=M(-q)^{\lambda\int_Xc_1^{\bb T\times\bb C^*_m}(\scr L\otimes e^{-m})c_3^{\bb T}(T_X(-\log D))}.
\end{align*}
Since the characteristic of the field $\mathrm{Frac}(\bb Q[\hbar,m])\cong\bb Q(\hbar,m)$ is zero
and the constant terms of $\mathsf{DT}_0(X|D)_{\scr L,m}$ and $M(-q)$ are both $1$,
this forces 
\[\mathsf{DT}_0(X|D)_{\scr L,m}=M(-q)^{\int_Xc_1^{\bb T\times\bb C^*_m}(\scr L\otimes e^{-m})c_3^{\bb T}(T_X(-\log D))}\]
and the result follows.
\end{proof}

\begin{nota}[Log modification invariance for local surfaces]
Let $f:(\widetilde S,\widetilde C)\to(S,C)$ be a logarithmic
modification of smooth projective snc surface pairs, i.e. a sequence of
blow-ups along strata, with $\widetilde C$ the reduced total transform.
Set
$X=\Tot(\omega_S(C))\times\mathbb C,
D=\Tot(\omega_C)\times\mathbb C$
and similarly
$\widetilde X=\Tot(\omega_{\widetilde S}(\widetilde C))\times\mathbb C,
\widetilde D=\Tot(\omega_{\widetilde C})\times\mathbb C$.
Then for any line bundle $L\in\mathrm{Pic}(S)$, after choosing the compatible
total-space orientations,
\[
\DT_0(\widetilde X|\widetilde D)_{f^*L,m}
=
\DT_0(X|D)_{L,m}.
\]
This follows directly from the fact
$\omega_{\widetilde S}(\widetilde C)\cong f^*\omega_S(C)$
and Theorem \ref{thm-total-surface-log-DT4}.
\end{nota}

\begin{coro}
Consider any smooth projective surface $S$ and a line bundle $\scr L\in\mathrm{Pic}(S)$.
Then there exists a choice of orientation such that
\[
\mathsf{DT}_0(\mathrm{Tot}(\omega_S)\times\bb C)_{\scr L,m}=M(-q)^{\int_{\mathrm{Tot}(\omega_S)\times\bb C}c_1^{\bb T\times\bb C^*_m}(\scr L\otimes e^{-m})c_3^{\bb T}(T_{\mathrm{Tot}(\omega_S)\times\bb C})},
\]
where $M(q)=\prod_{n\geq1}(1-q^n)^{-n}$ is the MacMahon function.
\end{coro}

Finally, we can consider the invariants with free insertions by specialization.
Let $(X,D)$ be a log Calabi-Yau pair with reduced simple 
normal crossing divisor $D$ and $\dim X=4$, and suppose there exists an orientation $o$
and a Calabi-Yau torus $\bb T$-action (possibly trivial) with proper fixed locus. Let $\scr L$ be a $\bb T$-equivariant
line bundle.
Consider the following invariant with free insertion
\begin{align*}
    \mathsf{DT}_0(X|D)_{\mathrm{free}}&:=1+\sum_{n>0}q^n\int_{[\Hilb^n(X|D)]^{\vir}_{\bb T}}1\\ 
    &=:1+\sum_{n>0}q^nI_{n,\mathrm{free}}(X|D).
\end{align*}

Then by a proof similar to that of \cite[Thm.~6.8]{CKM22}, we have the following result.
\begin{thm}
We have 
\[
\lim_{m\to\infty}I_{n}(X|D;\scr L,m)\cdot q^n|_{q'=qm}=I_{n,\mathrm{free}}(X|D)\cdot (q')^n,
\]
where we denote $\mathsf{DT}_0(X|D)_{\scr L,m}=1+\sum_{n>0}I_{n}(X|D;\scr L,m)\cdot q^n$.
\end{thm}

Combining this result with Theorem \ref{thm-total-surface-log-DT4} and \cite[Thm.~6.16]{CZZ24},
we have the following corollaries of invariants with free insertions.
\begin{coro}
Under the same assumptions as Theorem \ref{thm-total-surface-log-DT4}, 
there exists a choice of orientation such that
\begin{align*}
\mathsf{DT}_0(X|D)_{\mathrm{free}}&
=\exp\left(q\int_Xc_3^{\bb T}(T_X(-\log D))\right)=1.
\end{align*}
\end{coro}

\begin{coro}
Let $C$ be a smooth projective curve with three line bundles $\scr L_i\in\mathrm{Pic}(C)$
and $r$ distinct points $S=\{p_1,...,p_r\}$ such that $\scr L_1\otimes\scr L_2\otimes\scr L_3\cong\omega_C(S)$.
Let $X:=\Tot_{C}(\scr L_1\oplus\scr L_2\oplus\scr L_3)\xrightarrow{\eta}C$ with $D:=\eta^{-1}(S)$,
and let the torus $\bb T=\{t_1t_2t_3=1\}\subset(\bb C^*)^3$ act by scaling the fibres of $X$.
Write $-s_i$ for the equivariant weight of $\scr L_i$, following the convention of \cite[Exam.~6.14]{CZZ24}.
Then there exists a choice of orientation such that
\begin{align*}
\mathsf{DT}_0(X|D)_{\mathrm{free}}&
=\exp\left(-q\left(\frac{1}{s_1}+\frac{1}{s_2}+\frac{1}{s_3}\right)(2-2g(C)-r)\right).
\end{align*}
\end{coro}

\appendix
\section{Splitting and genus reduction formulae of GLSMs}\label{appendix-GLSM}
The splitting and genus reduction formulae recorded in this appendix
essentially follow from the constructions of \cite{CZ23,KP25},
together with the shifted Lagrangian classes in
Conjecture \ref{conj-Joyce-conj},
especially of Remark \ref{rmk-KKPS-lagrangian-classes}. For the reader's convenience, we formulate the relevant
statements in the restricted non-orbifold setting used here and sketch
how the existing arguments apply.

\subsection{Derived gauged linear sigma models}
We restrict to the original setting of gauged linear sigma models as follows.
\begin{defi}\label{defi-GLSM}
A \textit{GLSM input datum} consists of the following.
\begin{itemize}
	    \item A finite-dimensional vector space $V$ with an action of $H:=G\times F$ for a reductive group $G$ (referred to as the \textit{gauge group}) and an algebraic torus $F$ 
		(referred to as the \textit{flavor group}).
    \item Let $\chi:F\to\bb C^*$ be a nontrivial character
	and define the Calabi-Yau torus as $F_0:=\ker\chi$. We extend this to $\chi:H\to\bb C^*$ which does not depend on $G$.
    \item (GIT quotient) Fix a character $\theta:G\to\bb C^*$ such that $V^{\theta\text{-st}}=V^{\theta\text{-ss}}$ and $G$ acts freely on this locus.
    \item An $H$-equivariant function $\widehat{w}:V\to\bb C(\chi)$ which descends to
	$w:[V/G]\to\bb C(\chi)$.
	\item An $R$-charge is a group morphism $R:\bb C^*\to F$ and we denote the composition by $R_{\chi}:\bb C^*\xrightarrow{R}F\xrightarrow{\chi}\bb C^*$.
	Here we will assume $\ker R_{\chi}=\{1\}$, which is the setting of \cite{CZ23}, to avoid twisted orbicurves.
\end{itemize}
\end{defi}

Let $\mathfrak M_{g,n}^{\mathrm{pre}}$ be the moduli stack of $n$-marked prestable curves of genus $g$, which is a smooth Artin stack. It has a universal curve $\mcal C_{g,n}\to\mathfrak M_{g,n}^{\mathrm{pre}}$ with sections $\Sigma_{g,n;k}\subset\mcal C_{g,n}$ for $k=1,...,n$. Define $\omega^{\log}:=\omega\left(\sum_k\Sigma_{g,n;k}\right)$.
Then we consider the derived Artin stack 
\begin{align*}
\dMap_{g,n}^{R_{\chi}=\omega^{\log}}(\dCrit_{[V/G]}(w)):=&
	\dMap_{\mathfrak M_{g,n}^{\mathrm{pre}}}^{\bb C^*}(\omega^{\log,\circ},\dCrit_{[V/G]}(w)\times\mathfrak M_{g,n}^{\mathrm{pre}})\\ 
=&\Res_{\mcal C_{g,n}/\mathfrak M_{g,n}^{\mathrm{pre}}}\left([\dCrit_{[V/G]}(w)/_{R'}\bb C^*]\times_{B\bb C^*,\omega^{\log}}\mcal C_{g,n}\right),
\end{align*}
where the $\bb C^*$-action on $V$ is given by $R'=R\circ R_{\chi}^{-1}$ which is available
since $\ker R_{\chi}=\{1\}$.
Since $\chi\circ R'=\mathrm{id}_{\bb C^*}$, the function $\widehat w$, and hence its
descent $w$, has weight one with respect to the induced
$\bb C^*$-action, satisfying the conditions in \cite[Prp.~7.7]{KP25}.
Here, $\omega^{\log,\circ}$ is the corresponding principal $\bb C^*$-bundle,
and $\Res$ is the Weil restriction; see \cite{HKR25,KP25} for details.
See \cite[Rmk.~7.14]{KP25} for the relation between this and \cite{CZ23}.


Here we consider the $0^+$-stability of \cite{FJR18}; see also
\cite{CZ23}. Then we have an open substack
\[
	\QM_{g,n}^{R_{\chi}=\omega^{\log}}(\Crit_{V\sslash_{\theta}G}(w),\beta)\subset\left(\dMap_{g,n}^{R_{\chi}=\omega^{\log}}(\dCrit_{[V/G]}(w))\right)_{\mathrm{cl}}
\]
given by $0^+$-stable objects. Here $\beta$ is the class defined as in \cite[Def.~2.7]{CZ23}. 
Here we note that $(\dCrit_{[V/G]}(w))_{\mathrm{cl}}=[\dCrit_{V}(\widehat{w})/G]_{\mathrm{cl}}$, so 
the $0^+$-stability and class $\beta$ make sense using the language of \cite{CZ23}.
By \cite[Prop.~2.1]{STV15}, there exists a unique open derived enhancement
\[
	\dQM_{g,n}^{R_{\chi}=\omega^{\log}}(\dCrit_{V\sslash_{\theta}G}(w),\beta)\subset\dMap_{g,n}^{R_{\chi}=\omega^{\log}}(\dCrit_{[V/G]}(w)),
\]
which is called the \textit{(derived)
moduli stack of genus $g$, $n$-pointed stable $R$-twisted quasimaps of class $\beta$}.
\begin{nota}
Although we write the moduli stacks as 
$(\mathsf{d})\QM_{g,n}^{R_{\chi}=\omega^{\log}}(\dCrit_{V\sslash_{\theta}G}(w),\beta)$,
this does not mean that we are counting curves contained in $V\sslash_{\theta}G$.
The objects may have base points, disjoint from nodes and markings, in
$[V/G]\backslash(V\sslash_{\theta}G)$;
see \cite[Def.~2.7(1)]{CZ23} for details.
\end{nota}

\begin{prp}
In this case, $\QM_{g,n}^{R_{\chi}=\omega^{\log}}(\Crit_{V\sslash_{\theta}G}(w),\beta)$ is a separated Deligne-Mumford stack of finite type. 
The fixed locus $\QM_{g,n}^{R_{\chi}=\omega^{\log}}(\Crit_{V\sslash_{\theta}G}(w),\beta)^{F_0}$ is proper if $\Crit_{V\sslash_{\theta}G}(w)^{F_0}$ is.
\end{prp}
\begin{proof}
See \cite[Prop.~5.2.3,~5.3.1,~5.4.1]{FJR18} and \cite[Thm.~2.12,~2.15]{CZ23}.
\end{proof}

\begin{itemize}
	\item From now on, we assume $\Crit_{V\sslash_{\theta}G}(w)^{F_0}$ is proper.
\end{itemize}

By \cite{CZ23} we have an evaluation map together with the forgetful map
\[
	\mathrm{ev}_{g,n}:=(\mathrm{ev}_{g,n;1},...,\mathrm{ev}_{g,n;n},f):\dMap_{g,n}^{R_{\chi}=\omega^{\log}}(\dCrit_{[V/G]}(w))\to (\dCrit_{[V/G]}(w))^{\times n}\times\mathfrak M_{g,n}^{\mathrm{pre}}.
\]
By \cite[Prop.~7.7,~Prop.~A.12(2)]{KP25}, this is an exact $(-1)$-shifted Lagrangian
relative to $\mathfrak M_{g,n}^{\mathrm{pre}}$.
Then restricting it to the stable locus and using $F_0$-equivariant shifted Lagrangian classes in Remark \ref{rmk-KKPS-lagrangian-classes}, we have 
the following diagram
\[\begin{tikzcd}
	{H^*_{F_0}(V\sslash_{\theta}G,w)^{\otimes n}\otimes H^*\left(\overline{\mcal M}_{g,n}\right)} & {H^*_{F_0}(V\sslash_{\theta}G,w)^{\otimes n}\otimes H^*\left(\mathfrak M_{g,n}^{\mathrm{pre}}\right)} \\
	{H^{F_0}_*(\mathrm{pt})_{\mathrm{loc}}} & {H_*^{\BM,F_0}\left(\dQM_{g,n}^{R_{\chi}=\omega^{\log}}(\dCrit_{V\sslash_{\theta}G}(w),\beta)\right)}
	\arrow["{\mathrm{st}^*}", from=1-1, to=1-2]
	\arrow["{\Phi_{g,n,\beta}}"', dashed, from=1-1, to=2-1]
	\arrow["{[\mathrm{ev}_{g,n}]^{\mathrm{Lag}}}", from=1-2, to=2-2]
	\arrow["{p_{\QM,*}}", from=2-2, to=2-1]
\end{tikzcd}\]
and we call $\Phi_{g,n,\beta}$ the \textit{quasimap invariants}. Here 
we use that $\QM_{g,n}^{R_{\chi}=\omega^{\log}}(\Crit_{V\sslash_{\theta}G}(w),\beta)^{F_0}$ is proper 
and we have the $F_0$-equivariant pushforward $p_{\QM,*}$ for the structure morphism
$p_{\QM}:\QM_{g,n}^{R_{\chi}=\omega^{\log}}(\Crit_{V\sslash_{\theta}G}(w),\beta)\to\spec\bb C$. We also use the fact that
the stabilization map
\[
    \mathrm{st}:\mathfrak M_{g,n}^{\mathrm{pre}}\to\overline{\mcal M}_{g,n}
\]
is flat, which follows from \cite[Prop.~2]{Beh97}, and we always assume $2g-2+n>0$.

\begin{exam}[Quasimap invariants of dimensional reduction data]
Consider the compatible dimensional reduction data 
$(X,Y,s,\phi,w)$ and $(X',Y',s',\phi',w')$ in \cite[Def.~6.2]{COZZ-1}. Then by the 
$(-1)$-shifted version of \cite[Cor.~5.3]{KP25} we have
	\[
		\dCrit_X(w)\cong\dCrit_{Z^{\mathrm{der}}(s)}(\phi)\cong\dCrit_{Z^{\mathrm{der}}(s')}(\phi')\cong\dCrit_{X'}(w')
	\]
	of $(-1)$-shifted symplectic stacks with canonical orientations. 
	If these identifications are compatible with the ambient GLSM data, then the corresponding quasimap invariants are equivalent. See also \cite[Lem.~4.4]{KKPS2} for the 
	special case and \cite[Prop.~6.5,~Prop.~6.9]{COZZ-1} for the direct proof of the isomorphism 
	of classical critical loci and critical cohomology.

One of the most important examples of compatible dimensional reduction data is
given by Nakajima quiver varieties and triple quiver varieties with cubic potential; see \cite[Exam.~6.13,~6.14]{COZZ-1} for details.
In this case, the results of
\cite[Thm.~1.3,~1.5]{COZZ-2} show that the corresponding
two-point functions are compatible with dimensional reduction.
\end{exam}

\subsection{Split-gluing formula for quasimap invariants}
Let $g_1+g_2=g$ and $n_1+n_2=n$. We consider the following gluing morphisms:
\[\begin{tikzcd}
	{\mathfrak M_{g_1,n_1+1}^{\mathrm{pre}}\times\mathfrak M_{g_2,n_2+1}^{\mathrm{pre}}} & {\mathfrak M_{g,n}^{\mathrm{pre}}} \\
	{\overline{\mcal M}_{g_1,n_1+1}\times\overline{\mcal M}_{g_2,n_2+1}} & {\overline{\mcal M}_{g,n}.}
	\arrow["{\mathrm{gl}^{\mathrm{pre}}}", from=1-1, to=1-2]
	\arrow[from=1-1, to=2-1]
	\arrow[from=1-2, to=2-2]
	\arrow["{\mathrm{gl}}", from=2-1, to=2-2]
\end{tikzcd}\]
Let $\QM$ and $\dQM$ be $\QM_{g,n}^{R_{\chi}=\omega^{\log}}(\Crit_{V\sslash_{\theta}G}(w),\beta)$ 
and $\dQM_{g,n}^{R_{\chi}=\omega^{\log}}(\dCrit_{V\sslash_{\theta}G}(w),\beta)$. 
Let $\QM_i(\beta_i)$ and $\dQM_i(\beta_i)$ be 
$\QM_{g_i,n_i+1}^{R_{\chi}=\omega^{\log}}(\Crit_{V\sslash_{\theta}G}(w),\beta_i)$ and 
$\dQM_{g_i,n_i+1}^{R_{\chi}=\omega^{\log}}(\dCrit_{V\sslash_{\theta}G}(w),\beta_i)$. 

For use below, we record a homotopical formulation of the gluing
diagram. The corresponding classical pullback diagram and
cotangent-complex calculation appear in \cite[Lem.~4.24]{CZ23}.

\begin{lem}\label{lem-gluing-QM-pullback-nonorbifold}
We have the following homotopical pullback diagram of derived stacks:
\[\begin{tikzcd}
    {\coprod_{\beta_1+\beta_2=\beta}\dQM_1(\beta_1)\times_{+,\dCrit_{V\sslash_{\theta}G}(w),-}\dQM_2(\beta_2)} & \dQM \\
    {\dCrit_{V\sslash_{\theta}G}(w)^n\times\mathfrak M_{g_1,n_1+1}^{\mathrm{pre}}\times\mathfrak M_{g_2,n_2+1}^{\mathrm{pre}}} & {\dCrit_{V\sslash_{\theta}G}(w)^n\times\mathfrak M_{g,n}^{\mathrm{pre}}.}
    \arrow["{\mathrm{gl}_{\QM}}", from=1-1, to=1-2]
    \arrow["{\mathrm{ev}_{g,n}^{\mathrm{node}}}", from=1-1, to=2-1]
    \arrow["\square"{description}, draw=none, from=1-1, to=2-2]
    \arrow["{\mathrm{ev}_{g,n}}", from=1-2, to=2-2]
    \arrow["{\mathrm{gl}^{\mathrm{pre}}}", from=2-1, to=2-2]
\end{tikzcd}\]
Here the fiber product 
is given by $\mathrm{ev}_{g_1,n_1+1;n_1+1}$ and
$\mathrm{ev}_{g_2,n_2+1;n_2+1}^{-}:=R'(-1)\circ\mathrm{ev}_{g_2,n_2+1;n_2+1}$.
\end{lem}

\begin{proof}
We denote $\mathfrak M_{12}:=\mathfrak M_{g_1,n_1+1}^{\mathrm{pre}}\times\mathfrak M_{g_2,n_2+1}^{\mathrm{pre}}$.
By \cite[Lem.~2.3]{HKR25} and the following diagram of universal curves,
\[\begin{tikzcd}
    {\mcal C_{\mathrm{node}}} & {\mathfrak M_{12}} \\
    {\mcal C} & {\mathfrak M_{g,n}^{\mathrm{pre}},}
    \arrow[from=1-1, to=1-2]
    \arrow[from=1-1, to=2-1]
    \arrow["\square"{description}, draw=none, from=1-1, to=2-2]
    \arrow["{\mathrm{gl}^{\mathrm{pre}}}", from=1-2, to=2-2]
    \arrow[from=2-1, to=2-2]
\end{tikzcd}\]
we have
\begin{align*}
&\Res_{\mcal C/\mathfrak M_{g,n}^{\mathrm{pre}}}
\left([\dCrit_{[V/G]}(w)/_{R'}\bb C^*]\times_{B\bb C^*}\mcal C\right)
\times_{\mathfrak M_{g,n}^{\mathrm{pre}}}\mathfrak M_{12}\\
\cong&\Res_{\mcal C_{\mathrm{node}}/\mathfrak M_{12}}
\left([\dCrit_{[V/G]}(w)/_{R'}\bb C^*]\times_{B\bb C^*}\mcal C_{\mathrm{node}}\right).
\end{align*}

Note that we also have the following homotopical pushout diagram in the category of derived 
Artin stacks by \cite[Thm.~5.6.4]{Lur04} and étale descent:
\[\begin{tikzcd}
    \mathfrak M_{12} & {\mcal C_1\times\mathfrak M_{g_2,n_2+1}^{\mathrm{pre}}} \\
    {\mathfrak M_{g_1,n_1+1}^{\mathrm{pre}}\times\mcal C_2} & {\mcal C_{\mathrm{node}}}
    \arrow["{p_1}", from=1-1, to=1-2]
    \arrow["{p_2}", from=1-1, to=2-1]
    \arrow["\circ"{description}, draw=none, from=1-1, to=2-2]
    \arrow[from=1-2, to=2-2]
    \arrow[from=2-1, to=2-2]
\end{tikzcd}\]
where $p_1,p_2$ correspond to the two extra marked points.

We claim that we have the following pullback diagram of derived Artin stacks:
\begin{tiny}
\[\begin{tikzcd}
    {\Res_{\mcal C_{\mathrm{node}}/\mathfrak M_{12}}\left([\dCrit_{[V/G]}(w)/_{R'}\bb C^*]\times_{B\bb C^*}\mcal C_{\mathrm{node}}\right)} & {\Res_{\mcal C_1/\mathfrak M_{g_1,n_1+1}^{\mathrm{pre}}}\left([\dCrit_{[V/G]}(w)/_{R'}\bb C^*]\times_{B\bb C^*}\mcal C_1\right)\times\mathfrak M_{g_2,n_2+1}^{\mathrm{pre}}} \\
    {\Res_{\mcal C_2/\mathfrak M_{g_2,n_2+1}^{\mathrm{pre}}}\left([\dCrit_{[V/G]}(w)/_{R'}\bb C^*]\times_{B\bb C^*}\mcal C_2\right)\times\mathfrak M_{g_1,n_1+1}^{\mathrm{pre}}} & {\dCrit_{[V/G]}(w)\times\mathfrak M_{12}.}
    \arrow[from=1-1, to=1-2]
    \arrow[from=1-1, to=2-1]
    \arrow["\square"{description}, draw=none, from=1-1, to=2-2]
    \arrow["{\mathrm{ev}_{g_1,n_1+1;n_1+1}}", from=1-2, to=2-2]
    \arrow["{\mathrm{ev}_{g_2,n_2+1;n_2+1}^{-}}", from=2-1, to=2-2]
\end{tikzcd}\]
\end{tiny}
In the last square, the node target is trivialized using the residue on the first branch.
The opposite residue on the second branch accounts for the superscript $-$.
Indeed, for a derived affine scheme $T\to\mathfrak M_{12}$, 
the pushout property gives an equivalence of mapping-space functors
on derived Artin stacks over $B\bb C^*$:
\begin{align*}
	\operatorname{Map}_{B\bb C^*}(\mcal C_{\mathrm{node}}\times_{\mathfrak M_{12}}T,-)
\cong&
\operatorname{Map}_{B\bb C^*}(\mcal C_1\times_{\mathfrak M_{g_1,n_1+1}^{\mathrm{pre}}}T,-)\\
&{}\times_{\operatorname{Map}_{B\bb C^*}(T,-)}
\operatorname{Map}_{B\bb C^*}(\mcal C_2\times_{\mathfrak M_{g_2,n_2+1}^{\mathrm{pre}}}T,-).
\end{align*}
Evaluating at $[\dCrit_{[V/G]}(w)/_{R'}\bb C^*]$, the three
mapping spaces involving curves are precisely the $T$-points
of the corresponding Weil restrictions.
The residue trivialization on the first branch identifies the node
term with $\operatorname{Map}(T,\dCrit_{[V/G]}(w))$.
These identifications are natural in $T$.

Restricting to the stable locus gives the stated fibre product.
The map $\mathrm{ev}_{g,n}^{\mathrm{node}}$ records the evaluations at the remaining
$n$ markings and the underlying pair of curves.
\end{proof}

\begin{nota}
In the general case of orbicurves as in \cite{FJR18,KP25}, we can also generalize this to the following
homotopical pullback diagram of derived stacks:
\[\begin{tikzcd}
    {\coprod_{\beta_1+\beta_2=\beta}\dQM_1(\beta_1)\times_{+,I(\dCrit_{V\sslash_{\theta}G}(w)),-}\dQM_2(\beta_2)} & \dQM \\
    {I(\dCrit_{V\sslash_{\theta}G}(w))^n\times\mathfrak M_{g_1,n_1+1}^{\mathrm{tw}}\times_{B^2\mu}\mathfrak M_{g_2,n_2+1}^{\mathrm{tw}}} & {I(\dCrit_{V\sslash_{\theta}G}(w))^n\times\mathfrak M_{g,n}^{\mathrm{tw}}}
    \arrow["{\mathrm{gl}_{\QM}}", from=1-1, to=1-2]
    \arrow["{\mathrm{ev}_{g,n}^{\mathrm{node}}}", from=1-1, to=2-1]
    \arrow["\square"{description}, draw=none, from=1-1, to=2-2]
    \arrow["{\mathrm{ev}_{g,n}}", from=1-2, to=2-2]
    \arrow["{\mathrm{gl}^{\mathrm{tw}}}", from=2-1, to=2-2]
\end{tikzcd}\]
where $I(-):=\mathrm{Res}_{\mathrm{pt}/B^2\mu}(-)$ is the moduli stack of cyclotomic gerbes
from \cite[Def.~7.4]{KP25}, and $\mathfrak M_{g,n}^{\mathrm{tw}}$ is the moduli stack of
$n$-marked twisted prestable curves of genus $g$, which is a smooth Artin stack.
\end{nota}

We can consider
\begin{align*}
&\Phi_{g_1,n_1+1,\beta_1}\otimes\Phi_{g_2,n_2+1,\beta_2}:\\
&H^*_{F_0}(V\sslash_{\theta}G,w)^{\otimes(n+2)}
\otimes H^*\left(\overline{\mcal M}_{g_1,n_1+1}\times\overline{\mcal M}_{g_2,n_2+1}\right)
\to H^{F_0}_*(\mathrm{pt})_{\mathrm{loc}}.
\end{align*}

The preceding formulation yields the following version of \cite[Thm.~5.7]{CZ23} in our notation.

\begin{thm}\label{thm-quasimap-gluing}
For any $\alpha\in H^*\left(\overline{\mcal M}_{g_1,n_1+1}\times\overline{\mcal M}_{g_2,n_2+1}\right)$
and $\gamma\in H^*_{F_0}(V\sslash_{\theta}G,w)^{\otimes n}$, we have
\[
\Phi_{g,n,\beta}((\mathrm{gl}_*\alpha)\boxtimes\gamma)
=\sum_{\beta_1+\beta_2=\beta}
\Phi_{g_1,n_1+1,\beta_1}\otimes\Phi_{g_2,n_2+1,\beta_2}
(\alpha\boxtimes(\gamma\boxtimes\Delta^{\mathrm{Lag}})),
\]
where $\Delta^{\mathrm{Lag}}\in H^*_{F_0}(V\sslash_{\theta}G,w)^{\otimes2}$
is represented by
\begin{align*}
\bb Q\to(\Delta^-)_{*}\bb Q
&\to
(\Delta^-)_*(\Delta^-)^!\left(\varphi_w^{\boxtimes2}\right)
\xrightarrow{\mathrm{can}}\varphi_w^{\boxtimes2}.
\end{align*}
Here $\Delta^-:=(\mathrm{id},R'(-1)):
\dCrit_{V\sslash_{\theta}G}(w)
\longrightarrow\dCrit_{V\sslash_{\theta}G}(w)^{\times2}$.
\end{thm}

\begin{proof}
We have the following diagram:
\[\begin{tikzcd}
    {\mathfrak M_{g_1,n_1+1}^{\mathrm{pre}}\times\mathfrak M_{g_2,n_2+1}^{\mathrm{pre}}} & {\mathfrak Q} & {\mathfrak M_{g,n}^{\mathrm{pre}}} \\
    & {\overline{\mcal M}_{g_1,n_1+1}\times\overline{\mcal M}_{g_2,n_2+1}} & {\overline{\mcal M}_{g,n}}
    \arrow["{i_{\mathfrak Q}}"{description}, from=1-1, to=1-2]
    \arrow["{\mathrm{gl}^{\mathrm{pre}}}"{description}, bend left=15, from=1-1, to=1-3]
    \arrow["{\mathrm{st}_1\times\mathrm{st}_2}"{description}, from=1-1, to=2-2]
    \arrow["j"{description}, from=1-2, to=1-3]
    \arrow["p", from=1-2, to=2-2]
    \arrow["\square"{description}, draw=none, from=1-2, to=2-3]
    \arrow["{\mathrm{st}}", from=1-3, to=2-3]
    \arrow["{\mathrm{gl}}", from=2-2, to=2-3]
\end{tikzcd}\]

By \cite[Proof of Thm.~5.7, Eqn.~(5.9)]{CZ23}, we know that the morphism
$i_{\mathfrak Q}$ is finite and
\begin{align*}
\mathrm{st}^*\mathrm{gl}_*\alpha
&=j_*p^*\alpha
=j_*i_{\mathfrak Q,*}(\mathrm{st}_1\times\mathrm{st}_2)^*\alpha\\
&=\mathrm{gl}^{\mathrm{pre}}_*(\mathrm{st}_1\times\mathrm{st}_2)^*\alpha.
\end{align*}
By Lemma \ref{lem-gluing-QM-pullback-nonorbifold}, we have:
\[\begin{tikzcd}
    {\coprod_{\beta_1+\beta_2=\beta}\dQM_1(\beta_1)\times_{+,\dCrit_{V\sslash_{\theta}G}(w),-}\dQM_2(\beta_2)} && \dQM \\
    {\dCrit_{V\sslash_{\theta}G}(w)^n\times\mathfrak M_{g_1,n_1+1}^{\mathrm{pre}}\times\mathfrak M_{g_2,n_2+1}^{\mathrm{pre}}} && {\dCrit_{V\sslash_{\theta}G}(w)^n\times\mathfrak M_{g,n}^{\mathrm{pre}}.}
    \arrow["{\mathrm{gl}_{\QM}}", from=1-1, to=1-3]
    \arrow["{\mathrm{ev}_{g,n}^{\mathrm{node}}}", from=1-1, to=2-1]
    \arrow["\square"{description}, draw=none, from=1-1, to=2-3]
    \arrow["{\mathrm{ev}_{g,n}}", from=1-3, to=2-3]
    \arrow["{\mathrm{gl}^{\mathrm{pre}}}", from=2-1, to=2-3]
\end{tikzcd}\]
Moreover, by \cite[Cor.~3.9]{Knu83}, $\mathrm{gl}$ and hence $j$ are finite.
Therefore the morphism $\mathrm{gl}^{\mathrm{pre}}=j\circ i_{\mathfrak Q}$ is finite.
Combining this with Conjecture \ref{conj-Joyce-conj}(3)
and Remark \ref{rmk-KKPS-lagrangian-classes}, we get
\[
[\mathrm{ev}_{g,n}]^{\mathrm{Lag}}\circ\mathrm{gl}^{\mathrm{pre}}_*
=\mathrm{gl}_{\QM,*}\circ[\mathrm{ev}_{g,n}^{\mathrm{node}}]^{\mathrm{Lag}}.
\]
Hence we have
\[
\Phi_{g,n,\beta}((\mathrm{gl}_*\alpha)\boxtimes\gamma)
=\sum_{\beta_1+\beta_2=\beta}
\left(p_{\QM_1(\beta_1)\times_{+,\dCrit(w)}\QM_2(\beta_2),-}\right)_*
[\mathrm{ev}_{g,n}^{\mathrm{node},\beta_1,\beta_2}]^{\mathrm{Lag}}
\left((\mathrm{st}_1\times\mathrm{st}_2)^*\alpha\boxtimes\gamma\right).
\]

Consider the following analogous diagram, which is a composition of two Lagrangian
correspondences relative to
$\mathfrak M_{g_1,n_1+1}^{\mathrm{pre}}\times\mathfrak M_{g_2,n_2+1}^{\mathrm{pre}}$:
\[\begin{tikzcd}
    {\dQM_1(\beta_1)\times_{+,\dCrit_{V\sslash_{\theta}G}(w),-}\dQM_2(\beta_2)} & {\dCrit_{V\sslash_{\theta}G}(w)\times\left(\mathfrak M_{g_1,n_1+1}^{\mathrm{pre}}\times\mathfrak M_{g_2,n_2+1}^{\mathrm{pre}}\right)} \\
    {\dQM_1(\beta_1)\times\dQM_2(\beta_2)} & {\dCrit_{V\sslash_{\theta}G}(w)^{\times2}\times\left(\mathfrak M_{g_1,n_1+1}^{\mathrm{pre}}\times\mathfrak M_{g_2,n_2+1}^{\mathrm{pre}}\right)} \\
    {\dCrit_{V\sslash_{\theta}G}(w)^{\times n}\times\left(\mathfrak M_{g_1,n_1+1}^{\mathrm{pre}}\times\mathfrak M_{g_2,n_2+1}^{\mathrm{pre}}\right).}
    \arrow["{\mathrm{ev}^{\Delta^-}_{1,2}}", from=1-1, to=1-2]
    \arrow["{i_{\Delta^-}^{\beta_1,\beta_2}}", from=1-1, to=2-1]
    \arrow["\square"{description}, draw=none, from=1-1, to=2-2]
    \arrow["{\Delta^-\times\mathrm{id}}", from=1-2, to=2-2]
    \arrow["{\mathrm{ev}^{\mathrm{pre}}_{\beta_1}\times\mathrm{ev}_{\beta_2}^{\mathrm{pre}}}", from=2-1, to=2-2]
    \arrow["{\mathrm{ev}_n}", from=2-1, to=3-1]
\end{tikzcd}\]
Here $\mathrm{ev}^{\mathrm{pre}}_{\beta_1}\times\mathrm{ev}^{\mathrm{pre}}_{\beta_2}$
uses the two unchanged node evaluations; the sign is encoded by $\Delta^-$.
By arguments similar to those in Theorem \ref{thm-degeneration-general-curves},
using the associativity of shifted Lagrangian classes in
Conjecture \ref{conj-Joyce-conj}(2)
and Remark \ref{rmk-KKPS-lagrangian-classes}, we have
\begin{align*}
&i_{\Delta^-,*}^{\beta_1,\beta_2}
[\mathrm{ev}_{g,n}^{\mathrm{node},\beta_1,\beta_2}]^{\mathrm{Lag}}
\left((\mathrm{st}_1\times\mathrm{st}_2)^*\alpha\boxtimes\gamma\right)\\
&=[\mathrm{ev}_{\beta_1}\times\mathrm{ev}_{\beta_2}]^{\mathrm{Lag}}
\left((\mathrm{st}_1\times\mathrm{st}_2)^*\alpha
\boxtimes\left(\gamma\boxtimes\Delta^{\mathrm{Lag}}\right)\right).
\end{align*}
Hence we have
\begin{align*}
\Phi_{g,n,\beta}((\mathrm{gl}_*\alpha)\boxtimes\gamma)
&=\sum_{\beta_1+\beta_2=\beta}
p_{\QM_1(\beta_1)\times\QM_2(\beta_2),*}\circ i_{\Delta^-,*}^{\beta_1,\beta_2}
[\mathrm{ev}_{g,n}^{\mathrm{node},\beta_1,\beta_2}]^{\mathrm{Lag}}
\left((\mathrm{st}_1\times\mathrm{st}_2)^*\alpha\boxtimes\gamma\right)\\
&=\sum_{\beta_1+\beta_2=\beta}
p_{\QM_1(\beta_1)\times\QM_2(\beta_2),*}
[\mathrm{ev}_{\beta_1}\times\mathrm{ev}_{\beta_2}]^{\mathrm{Lag}}
\left((\mathrm{st}_1\times\mathrm{st}_2)^*\alpha
\boxtimes\left(\gamma\boxtimes\Delta^{\mathrm{Lag}}\right)\right)\\
&=\sum_{\beta_1+\beta_2=\beta}
\Phi_{g_1,n_1+1,\beta_1}\otimes\Phi_{g_2,n_2+1,\beta_2}
\left(\alpha\boxtimes\left(\gamma\boxtimes\Delta^{\mathrm{Lag}}\right)\right).
\end{align*}
This proves the result.
\end{proof}

\begin{nota}
In \cite{CTZ25} and \cite{CO}, K-theoretic Lagrangian classes of $(-1)$-shifted Lagrangians
over derived critical loci are described
by different constructions of specialization maps. In \cite[\S~8]{CTZ25}, they also prove 
the K-theoretic degeneration and gluing formulae.
\end{nota}

\subsection{Small quantum critical cohomology}\label{sec-quantum-critical-coho}

Following \cite[\S~5.6]{CZ23}, we use the preceding gluing formula to define a small quantum product in the present setting.

\begin{defi}
Consider the Lagrangian diagonal $\Delta^{\mathrm{Lag}}=\sum_i\Delta_i\boxtimes\Delta^i\in H^*_{F_0}(V\sslash_{\theta}G,w)^{\otimes 2}$ and let $\gamma_1,...,\gamma_n\in H^*_{F_0}(V\sslash_{\theta}G,w)$.
\begin{enumerate}
    \item Define $\left\langle\gamma_1,...,\gamma_n\right\rangle_{g,n,\beta}:=\Phi_{g,n,\beta}\left([\overline{\mcal M}_{g,n}]\boxtimes\gamma_1\boxtimes\cdots\boxtimes\gamma_n\right)$
	and define
    \[\left\langle\gamma_1,...,\gamma_n,*\right\rangle_{g,n+1,\beta}:=\sum_i\left\langle\gamma_1,...,\gamma_n,\Delta_i\right\rangle_{g,n+1,\beta}\Delta^i\in H^*_{F_0}(V\sslash_{\theta}G,w)_{\mathrm{loc}}.\]

	    \item The \textit{small quantum product} is given by 
    \[\gamma_1\star\gamma_2:=\sum_{\beta\in N_+(\Crit(w))}
	\left\langle\gamma_1,\gamma_2,*\right\rangle_{0,3,\beta}z^{\beta}\in 
	H^*_{F_0}(V\sslash_{\theta}G,w)_{\mathrm{loc}}[[z]].\]
    This can be extended to $H^*_{F_0}(V\sslash_{\theta}G,w)_{\mathrm{loc}}[[z]]$ naturally.
\end{enumerate}
\end{defi}

The gluing formula implies associativity of the product by arguments similar to those for the Witten-Dijkgraaf-Verlinde-Verlinde (WDVV) equation.
\begin{prp}[associativity]
For any $\gamma_1,\gamma_2,\gamma_3\in H^*_{F_0}(V\sslash_{\theta}G,w)$, we have
\[(\gamma_1\star\gamma_2)\star\gamma_3=\gamma_1\star(\gamma_2\star\gamma_3).\]
\end{prp}
\begin{proof}
Fix a total class $\beta$. Let
$\xi_{12|34},\xi_{14|23}\in H^2(\overline{\mcal M}_{0,4})$
be the classes of the corresponding boundary points.
By Theorem \ref{thm-quasimap-gluing}, we have
\begin{align*}
&\sum_{\beta_1+\beta_2=\beta}
\left\langle
\left\langle\gamma_1,\gamma_2,*\right\rangle_{0,3,\beta_1},
\gamma_3,*
\right\rangle_{0,3,\beta_2}\\
&=\sum_{\beta_1+\beta_2=\beta}\sum_{i,j}
\left\langle\gamma_1,\gamma_2,\Delta_i\right\rangle_{0,3,\beta_1}
\left\langle\Delta^i,\gamma_3,\Delta_j\right\rangle_{0,3,\beta_2}
\Delta^j\\
&=\sum_j\Phi_{0,4,\beta}
\left(\xi_{12|34}\boxtimes\gamma_1\boxtimes\gamma_2
\boxtimes\gamma_3\boxtimes\Delta_j\right)\Delta^j.
\end{align*}
This is the coefficient of $z^\beta$ in
$(\gamma_1\star\gamma_2)\star\gamma_3$.
Similarly, the coefficient of $z^\beta$ in
$\gamma_1\star(\gamma_2\star\gamma_3)$ is the same expression
with $\xi_{12|34}$ replaced by $\xi_{14|23}$.
Since these point classes agree in
$H^2(\overline{\mcal M}_{0,4})$, the result follows.
\end{proof}

\begin{nota}
Let $\gamma_1,\gamma_2,\tau\in H^*_{F_0}(V\sslash_{\theta}G,w)$.
Then we can define the \textit{big quantum product}:
\[ 
\gamma_1\star_{\mathrm{big},\tau}\gamma_2
:=\sum_{n\geq0}\sum_{\beta\in N_+(\Crit(w))}\sum_i\frac{1}{n!}\Phi_{0,n+3,\beta} 
\left([\overline{\mcal M}_{0,n+3}]\boxtimes\gamma_1\boxtimes\gamma_2\boxtimes\Delta_i\boxtimes\tau^{\boxtimes n}
\right)\Delta^i z^{\beta}.
\]
The gluing formula yields a WDVV-type equation as above and hence the associativity.
\end{nota}

\subsection{Genus reduction formula for quasimap invariants}
In this section we will consider the behavior of quasimap invariants along another 
gluing morphism:
\[\begin{tikzcd}
	{\mathfrak M_{g-1,n+2}^{\mathrm{pre}}} & {\mathfrak M_{g,n}^{\mathrm{pre}}} \\
	{\overline{\mcal M}_{g-1,n+2}} & {\overline{\mcal M}_{g,n}}
	\arrow["{\psi_{g,n}^{\mathrm{pre}}}", from=1-1, to=1-2]
	\arrow["{\mathrm{st}_{g-1,n+2}}", from=1-1, to=2-1]
	\arrow["{\mathrm{st}_{g,n}}", from=1-2, to=2-2]
	\arrow["{\psi_{g,n}}", from=2-1, to=2-2]
\end{tikzcd}\]
by gluing the final two marked points together.

\begin{thm}[Genus reduction formula]\label{thm-quasimap-self-gluing}
For $\alpha\in H^*(\overline{\mcal M}_{g-1,n+2})$
and $\gamma\in H_{F_0}^*(V\sslash_{\theta}G,w)^{\otimes n}$, we have
\[
\Phi_{g,n,\beta}((\psi_{g,n,*}\alpha)\boxtimes\gamma)
=\Phi_{g-1,n+2,\beta}
(\alpha\boxtimes\gamma\boxtimes\Delta^{\mathrm{Lag}})
\]
for the Lagrangian diagonal $\Delta^{\mathrm{Lag}}$
defined in Theorem \ref{thm-quasimap-gluing}.
\end{thm}

\begin{proof}
The result follows by the same argument as in Theorem \ref{thm-quasimap-gluing},
applied to the following pullback diagram:
\[\begin{tikzcd}
    {\mathfrak M_{g-1,n+2}^{\mathrm{pre}}} & {\mathfrak D} & {\mathfrak M_{g,n}^{\mathrm{pre}}} \\
    & {\overline{\mcal M}_{g-1,n+2}} & {\overline{\mcal M}_{g,n}}
    \arrow["{i_{\mathfrak D}}"{description}, from=1-1, to=1-2]
    \arrow["{\psi^{\mathrm{pre}}_{g,n}}"{description}, bend left=15, from=1-1, to=1-3]
    \arrow["{\mathrm{st}_{g-1,n+2}}"{description}, from=1-1, to=2-2]
    \arrow["j"{description}, from=1-2, to=1-3]
    \arrow["p", from=1-2, to=2-2]
    \arrow["\square"{description}, draw=none, from=1-2, to=2-3]
    \arrow["{\mathrm{st}_{g,n}}", from=1-3, to=2-3]
    \arrow["{\psi_{g,n}}", from=2-2, to=2-3]
\end{tikzcd}\]
and the following composition of Lagrangian correspondences
relative to $\mathfrak M_{g-1,n+2}^{\mathrm{pre}}$:
\[\begin{tikzcd}
    {\dQM_{g-1,n+2}^{R_{\chi}=\omega^{\log}}\left(\dCrit_{V\sslash_{\theta}G}(w),\beta\right)\times_{\dCrit_{V\sslash_{\theta}G}(w)^2,\Delta^-}\dCrit_{V\sslash_{\theta}G}(w)} & {\dCrit_{V\sslash_{\theta}G}(w)\times\mathfrak M_{g-1,n+2}^{\mathrm{pre}}} \\
    {\dQM_{g-1,n+2}^{R_{\chi}=\omega^{\log}}\left(\dCrit_{V\sslash_{\theta}G}(w),\beta\right)} & {\dCrit_{V\sslash_{\theta}G}(w)^2\times\mathfrak M_{g-1,n+2}^{\mathrm{pre}}} \\
    {\dCrit_{V\sslash_{\theta}G}(w)^n\times\mathfrak M_{g-1,n+2}^{\mathrm{pre}}}
    \arrow[from=1-1, to=1-2]
    \arrow[from=1-1, to=2-1]
    \arrow["\square"{description}, draw=none, from=1-1, to=2-2]
    \arrow["{\Delta^-}", from=1-2, to=2-2]
    \arrow[from=2-1, to=2-2]
    \arrow[from=2-1, to=3-1]
\end{tikzcd}\]
and the following homotopical pullback diagram of derived stacks:
\[\begin{tikzcd}
    {\dQM_{g-1,n+2}^{R_{\chi}=\omega^{\log}}\left(\dCrit_{V\sslash_{\theta}G}(w),\beta\right)\times_{\dCrit_{V\sslash_{\theta}G}(w)^2,\Delta^-}\dCrit_{V\sslash_{\theta}G}(w)} & {\dQM_{g,n}^{R_{\chi}=\omega^{\log}}\left(\dCrit_{V\sslash_{\theta}G}(w),\beta\right)} \\
    {\dCrit_{V\sslash_{\theta}G}(w)^n\times\mathfrak M_{g-1,n+2}^{\mathrm{pre}}} & {\dCrit_{V\sslash_{\theta}G}(w)^n\times\mathfrak M_{g,n}^{\mathrm{pre}}.}
    \arrow[from=1-1, to=1-2]
    \arrow["{\mathrm{ev}^{\mathrm{node}}_{g,n}}", from=1-1, to=2-1]
    \arrow["\square"{description}, draw=none, from=1-1, to=2-2]
    \arrow["{\mathrm{ev}_{g,n}}"', from=1-2, to=2-2]
    \arrow["{\psi_{g,n}^{\mathrm{pre}}}", from=2-1, to=2-2]
\end{tikzcd}\]
Here we note that the final diagram follows from a proof similar to that of
Lemma \ref{lem-gluing-QM-pullback-nonorbifold} by considering the pullback
and the pushout
diagrams induced by the gluing construction:
\[\begin{tikzcd}
    {\mcal C_{g,n}^{\mathrm{node}}} & {\mcal C_{g,n}} && {\mathfrak M_{g-1,n+2}^{\mathrm{pre}}\sqcup\mathfrak M_{g-1,n+2}^{\mathrm{pre}}} & {\mcal C_{g-1,n+2}} \\
    {\mathfrak M_{g-1,n+2}^{\mathrm{pre}}} & {\mathfrak M_{g,n}^{\mathrm{pre}},} && {\mathfrak M_{g-1,n+2}^{\mathrm{pre}}} & {\mcal C_{g,n}^{\mathrm{node}}.}
    \arrow[from=1-1, to=1-2]
    \arrow[from=1-1, to=2-1]
    \arrow["\square"{description}, draw=none, from=1-1, to=2-2]
    \arrow[from=1-2, to=2-2]
    \arrow[from=1-4, to=1-5]
    \arrow[from=1-4, to=2-4]
    \arrow["\circ"{description}, draw=none, from=1-4, to=2-5]
    \arrow[from=1-5, to=2-5]
    \arrow["{\psi_{g,n}^{\mathrm{pre}}}", from=2-1, to=2-2]
    \arrow[from=2-4, to=2-5]
\end{tikzcd}\]
All these spaces have compatible morphisms to $B\bb C^*$ induced by $\omega^{\log}$.
In the pushout square, the two fibres at the node are identified using opposite residues;
the induced morphism
$\mathfrak M_{g-1,n+2}^{\mathrm{pre}}\to B\bb C^*$
classifies the trivial bundle.
Thus the matching of the two unchanged evaluations at the last two markings
is along $\Delta^-$.
The compatibility of shifted Lagrangian classes with these pullback and gluing diagrams
gives the stated identity.
\end{proof}

\end{document}